\documentclass[a4paper,twoside,11pt]{article}
\usepackage[utf8]{inputenc}
\usepackage[english]{babel}
\usepackage{amssymb, amsfonts, amsthm, amsmath}
\usepackage[hidelinks]{hyperref}
\usepackage{pgf,tikz} 
\usepackage[margin=2.5cm]{geometry} 
\usepackage[mathscr]{euscript}
\usepackage{psfrag}
\usepackage[percent]{overpic}
\usepackage[normalem]{ulem}
\usepackage{graphicx,accents}
\usepackage{caption} 
\usepackage{mathtools} 
\mathtoolsset{showonlyrefs}

\makeatletter
\DeclareRobustCommand{\cev}[1]{%
  \mathpalette\do@cev{#1}%
}
\newcommand{\do@cev}[2]{%
  \fix@cev{#1}{+}%
  \reflectbox{$\m@th#1\vec{\reflectbox{$\fix@cev{#1}{-}\m@th#1#2\fix@cev{#1}{+}$}}$}%
  \fix@cev{#1}{-}%
}
\newcommand{\fix@cev}[2]{%
  \ifx#1\displaystyle
    \mkern#23mu
  \else
    \ifx#1\textstyle
      \mkern#23mu
    \else
      \ifx#1\scriptstyle
        \mkern#22mu
      \else
        \mkern#22mu
      \fi
    \fi
  \fi
}

\makeatother

\usepackage{blkarray}

\usepackage{color}

\usetikzlibrary{math,arrows,calc,through,shapes}

\numberwithin{equation}{section}

\theoremstyle{plain}
\newtheorem{theorem}{Theorem}[section]
\newtheorem{corollary}[theorem]{Corollary}
\newtheorem{conjecture}[theorem]{Conjecture}
\newtheorem{lemma}[theorem]{Lemma}
\newtheorem{proposition}[theorem]{Proposition}

\theoremstyle{definition}
\newtheorem{definition}[theorem]{Definition}

\newtheorem{question}[theorem]{Question}

\theoremstyle{remark}
\newtheorem{remark}[theorem]{Remark}

\newcommand{\N}{\mathbb{N}}

\newcommand{\Z}{\mathbb{Z}}
\newcommand{\R}{\mathbb{R}}
\newcommand{\C}{\mathbb{C}}
\newcommand{\CP}{{\mathbb{C}\mathrm{P}}}
\newcommand{\RP}{{\mathbb{R}\mathrm{P}}}

\newcommand{\calL}{\mathcal{L}}

\newcommand{\hC}{\hat{\C}}

\newcommand{\Ms}{\mathsf{M}}
\newcommand{\M}{\mathscr{M}}

\newcommand{\vphi}{\varphi}

\newcommand{\az}[2]{{A_{#1}\kern-0.1em\left[#2\right]}}

\newcommand{\RE}{\mathrm{Re}}
\newcommand{\IM}{\mathrm{Im}}
\newcommand{\dd}{{\mathrm{d}}}

\DeclareMathOperator{\cro}{cr} 

\tikzstyle{bvert}=[draw,circle,fill=black,minimum size=5pt,inner sep=0pt]
\tikzstyle{blvert}=[draw,circle,fill=blue,draw=blue,minimum size=5pt,inner sep=0pt]
\tikzstyle{wvert}=[draw,circle,fill=white,minimum size=5pt,inner sep=0pt]
\tikzstyle{blwvert}=[draw,circle,draw=blue,fill=white,minimum size=5pt,inner sep=0pt]
\tikzset{circle through 3 points/.style n args={3}{%
insert path={let    \p1=($(#1)!0.5!(#2)$),
                    \p2=($(#1)!0.5!(#3)$),
                    \p3=($(#1)!0.5!(#2)!1!-90:(#2)$),
                    \p4=($(#1)!0.5!(#3)!1!90:(#3)$),
                    \p5=(intersection of \p1--\p3 and \p2--\p4)
                    in
                 node at (\p5) [draw,circle through= {(#1)}]{}}
}}
\newcommand*{\Scale}[2][4]{\scalebox{#1}{$#2$}}%

\title{The Schwarzian octahedron recurrence (dSKP equation) II: geometric systems}
\author{Niklas Christoph Affolter
     \thanks{TU Berlin, Institute of Mathematics, Strasse des 17. Juni 136, 10623 Berlin, Germany.
      Départment de mathématiques, ENS, Université PSL, 45 rue d'Ulm, 75005 Paris, France. \textit{E-mail address}: \texttt{affolter~at~posteo.net}} ,
      Béatrice de Tilière
      \thanks{PSL University-Dauphine, CNRS, UMR 7534, CEREMADE, 75016 Paris, France. \textit{E-mail address}: \texttt{detiliere at ceremade.dauphine.fr}} ,
      Paul Melotti
      \thanks{Université Paris-Saclay, CNRS, Laboratoire de mathématiques d’Orsay, 91405, Orsay, France.
                \textit{E-mail address}: \texttt{melotti at posteo.net}}
}

\date{March 6, 2024} 

\begin{document}

\maketitle

\begin{abstract}
We consider nine geometric systems: Miquel dynamics, P-nets, integrable cross-ratio maps, discrete holomorphic functions, orthogonal circle patterns, polygon recutting, circle intersection dynamics, (corrugated) pentagram maps and the short diagonal hyperplane map. Using a unified framework, for each system we prove an explicit expression for the solution as a function of the initial data; more precisely, we show that the solution is equal to the ratio of two partition functions of an oriented dimer model on an Aztec diamond whose face weights are constructed from the initial data. Then, we study the Devron property~\cite{gdevron}, which states the following: if the system starts from initial data that is singular for the backwards dynamics, this singularity is expected to reoccur after a finite number of steps of the forwards dynamics. Again, using a unified framework, we prove this Devron property for all of the above geometric systems, for different kinds of singular initial data. In doing so, we obtain new singularity results and also known ones~\cite{gdevron,yao}. Our general method consists in proving that these nine geometric systems are all related to the Schwarzian octahedron recurrence (dSKP equation), and then to rely on the companion paper~\cite{paper1}, where we study this recurrence in general, prove explicit expressions and singularity results.
\end{abstract}

\textbf{Keywords:} discrete differential geometry, projective geometry, dynamical systems, discrete integrable systems, dSKP equation, pentagram map, dimer model.

\textbf{MSC classes:} 37K10, 39A36, 51A05, 82B20.

\newpage
\setcounter{tocdepth}{1}
\tableofcontents

\section{Introduction}\label{sec:introduction}

The dSKP equation is a relation on six variables that arises in the study of the Krichever-Novikov equation \cite[Equation (30)]{dndskp}, and as a
discretization of the Schwarzian Kadomtsev-Petviashvili
hierarchy \cite{BK1,BK2}, hence its name. Note that it can be traced back to \cite[Equation (4)]{ncwqdskp} as a special case when  $p,q,r = 0$, $\alpha = -\beta = \varepsilon^{-1}$ in the limit $\varepsilon \rightarrow 0$.
In this paper, this relation is embedded in the \emph{octahedral-tetrahedral lattice $\calL$} defined by:
\begin{align*}
  \calL = \left\{p=(i,j,k) \in \Z^3 : i+j+k \in 2\Z \right\}.
\end{align*}
Consider a function $x:\calL \to \hat{\C}$, or more generally from a subset of
$\calL$ to $\hat{\C}$. We say that $x$ \emph{satisfies
the dSKP recurrence}, or \emph{Schwarzian octahedron recurrence}, if
\begin{align}\label{eq:dskp_x_intro}
  \frac{(x_{-e_3}-x_{e_2})(x_{-e_1}-x_{e_3})(x_{-e_2}-x_{e_1})}{
  (x_{e_2}-x_{-e_1})(x_{e_3}-x_{-e_2})(x_{e_1}-x_{-e_3})} = -1,
\end{align}
where $(e_1,e_2,e_3)$ is the canonical basis of $\Z^3$, $x_q(p) :=
x(p+q)$ for every $q\in(\pm e_i)_{i=1}^3$, and the relation is evaluated
at any $p \in \Z^3\setminus\calL$ such that all terms are defined; see Figure~\ref{fig:octa_evol} where $p=(1,1,1)$. The target space $\hat{\C}$ is an \emph{affine chart of} $\CP^1$.

Suppose that we are given \emph{initial data} $\left( a_{i,j} \right)_{i,j\in \Z^2}$.
One starts with values $x(i,j,[i+j]_2) = a_{i,j}$, where $[n]_p\in\{0,\dots,p-1\}$ denotes the value of $n$ modulo $p$. Then, the
dSKP recurrence allows one to propagate this initial data in the positive $k$ direction, to get any value $x(i,j,k)$ with $(i,j,k) \in \calL$
and $k>1$ (or symmetrically in the negative direction to get $k<0$) see Figure~\ref{fig:octa_evol}.
The resulting values of $x$ on $\calL$, which we simply call the \emph{solution}, is a rational function in the
initial data $a$. One of the main contributions of the companion paper~\cite{paper1}
is an explicit combinatorial expression of this rational function as the ratio of two partition functions of an associated oriented dimer model. Other main contributions consist in the study of singular initial conditions. These results, specified to the
cases needed in the current paper, are recalled in Section~\ref{sec:prereq}.

\begin{figure}[tb]
\centering
\begin{tikzpicture}[x  = {(-0.95cm,-0.0cm)}, y  = {(0.1659cm,-0.30882cm)}, z  = {(-0cm,1cm)}, scale=1.9]
\draw [fill opacity=0.8,fill=white!90!black] (5.5,5.5,-1) -- (0.5,5.5,-1) -- (0.5,0.5,-1) -- (5.5,0.5,-1) -- cycle;
\foreach \i in {1,...,6}
{
  \draw[dotted] (6.5-\i,5.5,-1) -- (6.5-\i,0.5,-1);
  \draw[dotted] (5.5,6.5-\i,-1) -- (0.5,6.5-\i,-1);
}
\foreach \i in {0,...,5}
{
  \foreach \j in {0,...,5}
  {
    \pgfmathparse{Mod(\i+\j,2)==0?1:0}
    \ifnum\pgfmathresult>0
      \draw (5.55-\i,5.5-\j,-1.09) node {$\Scale[1]{a_{\i,\j}}$};
      \node[draw,diamond,aspect=1,color=red,fill,inner sep=1.3pt] at (5.5-\i,5.5-\j,-1) {} ;
    \fi
  }
}
\coordinate (a01) at (5.5,4.5,0);
\coordinate (a10) at (4.5,5.5,0);
\coordinate (a21) at (3.5,4.5,0);
\coordinate (a12) at (4.5,3.5,0);
\coordinate (a11) at (4.5,4.5,-1);
\coordinate (b11) at (4.5,4.5,1);
\draw (a11) -- (a01) -- (a10) -- (a21) -- cycle ;
\draw (a11) -- (a10) ;
\draw [fill opacity=0.7,fill=white!70!black] (a01) -- (a10) -- (a11) -- cycle;
\draw [fill opacity=0.7,fill=white!60!black] (a10) -- (a21) -- (a11) -- cycle;
\draw [fill opacity=0.8,fill=white!90!black] (5.5,5.5,0) -- (0.5,5.5,0) -- (0.5,0.5,0) -- (5.5,0.5,0) -- cycle;
\draw [dashed] (a01) -- (a12) -- (a21) ;
\draw [dashed] (a11) -- (a12) ;
\foreach \i in {1,...,6}
{
  \draw[dotted] (6.5-\i,5.5,0) -- (6.5-\i,0.5,0);
  \draw[dotted] (5.5,6.5-\i,0) -- (0.5,6.5-\i,0);
}
\foreach \i in {0,...,5}
{
  \foreach \j in {0,...,5}
  {
    \pgfmathparse{Mod(\i+\j,2)==0?0:1}
    \ifnum\pgfmathresult>0
      \draw (5.55-\i,5.5-\j,-0.09) node {$\Scale[1]{a_{\i,\j}}$};      
      \node[draw,diamond,aspect=1,color=blue,fill,inner sep=1.3pt] at (5.5-\i,5.5-\j,0) {} ;
    \fi
  }
}
\draw [dashed] (a12) -- (b11) ;
\draw (b11) -- (a01) ;
\draw (b11) -- (a10) ;
\draw (b11) -- (a21) ;
\draw [fill opacity=0.7,fill=white!80!black] (a10) -- (b11) -- (a21) -- cycle;
\draw [fill opacity=0.7,fill=white!85!black] (a01) -- (a10) -- (b11) -- cycle;
\draw [->] (6,5.5,-1) -- (0,5.5,-1);
\draw (0,5.5,-1) node [right] {$i$};
\draw [->] (5.5,6,-1) -- (5.5,0,-1);
\draw (5.5,0,-1) node [left] {$j$};
\draw [->] (5.5,5.5,-1.2) -- (5.5,5.5,2);
\draw (5.5,5.5,2) node [above] {$k$};
\end{tikzpicture}
\caption{The layers $k=0,1$ in $\calL$ (red, blue diamonds respectively). The \emph{initial data} $(a_{i,j})$ are the values of $x$ on these layers. An octahedral cell of $\calL$ is shown. The dSKP recurrence gives the value of $x$ at the top vertex of the octahedron, here $(1,1,2)$, in terms of the initial data. This can be repeated to get $x$ on the whole lattice $\calL$, provided no singularity occurs.}
\label{fig:octa_evol}
\end{figure}
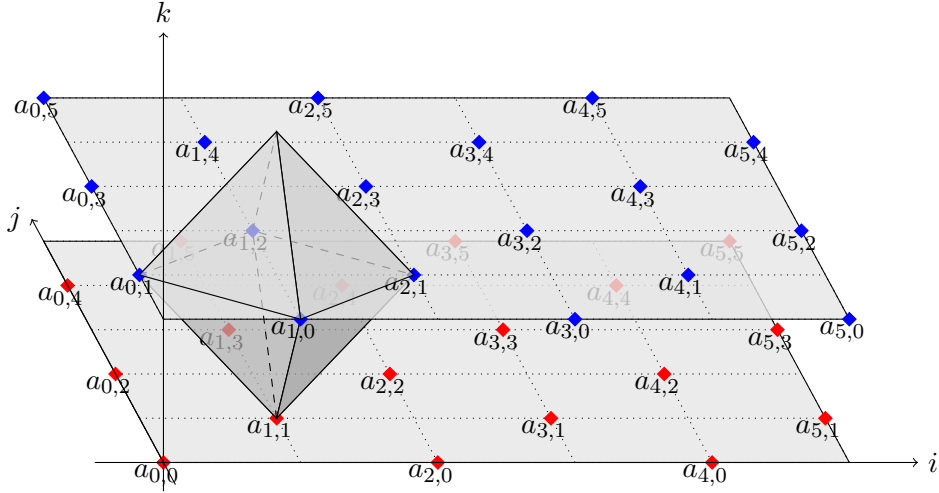

The purpose of this paper is to use the general framework of the companion paper \cite{paper1} to study discrete geometric systems. The systems considered here come in three groups. The first group consists solely of \emph{Miquel dynamics} (Section~\ref{sec:miquel}), it is the only example we consider that has two-dimensional initial data. The second group consists of \emph{integrable cross-ratio maps} (Section~\ref{sec:Backlund_pairs}) and special cases thereof, which are \emph{discrete holomorphic functions} (Section~\ref{sec:dhol}), \emph{polygon recutting} (Section~\ref{sec:recut}) and \emph{circle intersection dynamics} (Section \ref{sec:cid}). As a special case of discrete holomorphic functions we also consider \emph{orthogonal circle patterns} (Section~\ref{sec:orthogonal_circle_patterns}), also known as \emph{Schramm circle packings}. The third group consists of the \emph{corrugated pentagram map} for \emph{$N$-corrugated polygons} (Section~\ref{sec:corrugated}), which for $N=2$ is simply called the \emph{pentagram map} (Section \ref{sec:pent}) as well as the \emph{short diagonal hyperplane map} (Section~\ref{sec:hyppent}). A special role is played by \emph{P-nets} (Section~\ref{sec:pnets}), which arise both as a restriction of discrete holomorphic functions but also correspond to the case $N=1$ of the corrugated pentagram map.

For each system, we prove an explicit expression for the solution $x$ as a
function of the initial data, that is, we express $x$ as a ratio of
  partition functions of oriented dimers on a specific graph. Finding the explicit solution to any one of these systems is a result in its own right. Thus it is all the more practical that we can obtain \emph{all} solutions to all the systems we present here from one underlying result.
Let us mention that in the special case of Schramm's circle packings some \emph{specific} solutions were studied, see for example \cite{abschramm} and references therein. 
%
As an example, let us state our result for P-nets, which were introduced in the study of discrete isothermic nets~\cite[Section 6.2]{bpdiscsurfaces} and appear in various geometric contexts, see Section~\ref{sec:pnets}. A P-net is a map $p:\Z^2 \rightarrow \hat{\C}$ such that, for all $(i,j)\in\Z^2$,
\begin{align}
	\frac{1}{p_{i+1,j}-p_{i,j}} - \frac{1}{p_{i,j+1}-p_{i,j}} + \frac{1}{p_{i-1,j}-p_{i,j}} - \frac{1}{p_{i,j-1}-p_{i,j}} = 0.
\end{align}
We may view this as an evolution equation on $\Z^2$, with initial data $p_0:=(p_{i,0})_{i\in\Z}$ and $p_1:=(p_{i,1})_{i\in\Z}$, propagating in the $j$ direction to get values of $p$ on the whole lattice $\Z^2$ in terms of $p_0,p_1$. Note that for P-nets, the initial data is one-dimensional, while the initial data of dSKP is two-dimensional. Our first contribution is to show that the P-net evolution is in fact a special case of the dSKP evolution, for which the initial data $(a_{i,j})$ is set to certain values depending on $p_0,p_1$ and has an extra periodicity. As a consequence of Lemma~\ref{lem:pnetdskp} we have the following special solution of the dSKP recurrence (note that it does not depend on $j$):
\begin{theorem}
	The function $x(i,j,k)=p_{i,k}$ satisfies the dSKP recurrence.
\end{theorem}
\begin{figure}[tb]
	\centering
	\includegraphics[width=6cm]{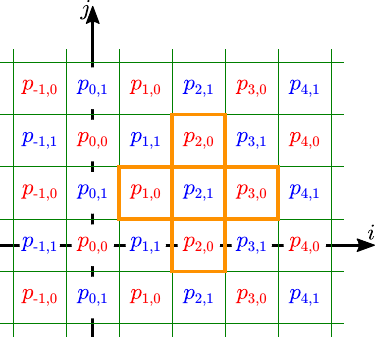}
	\caption{Special initial data for the dSKP evolution, seen from above, that induce the P-net evolution. The orange subgraph is the \emph{Aztec diamond} $A_2[p_{2,1}]$ and corresponds to the computation of
		$p_{2,3}$. Note that faces of the Aztec diamond correspond to vertices of $\calL$ at height $0,1$. Variables in red are those of $p_0$ and are all
		set to $0$ in singular P-nets, see Theorems~\ref{theo:pnetsingularity}~and~\ref{theo:pnetpremature} in the body of the paper.}
	\label{fig:pnet_ic}
\end{figure}
As a result, if one starts the dSKP evolution with initial data $a_{i,j}:=p_{i,[i+j]_2}$, see Figure~\ref{fig:pnet_ic}, then the corresponding dSKP solution contains the values of the P-net evolution. Combining this with general explicit expressions from \cite{paper1}, we get that $p_{i,k}$ can be written combinatorially as a ratio of two partition functions of oriented dimers, with weights that depend on the initial data. This is Theorem~\ref{theo:explpnet} in the body of the paper:
\begin{theorem} \label{theo:pnet_explicit_intro}
	Let $p:\Z^2 \rightarrow \hC$ be a P-net, and consider the graph $\Z^2$ with face-weights $(a_{i,j})_{(i,j)\in\Z^2}$ given by
	\begin{equation*}
		a_{i,j} = p_{i,[i+j]_2}.
	\end{equation*}
	Then, for all $i\in\Z, k\geq 1$, we have
	\begin{align*}
		p_{i,k}=  Y(\az{k-1}{p_{i,[k]_2}},a),
	\end{align*}
	where generically $A_k[a_{i,j}]$ denotes the Aztec diamond of size $k$ centered at a face with weight $a_{i,j}$; $Y(A_k[a_{i,j}],a)$ is the corresponding ratio function of oriented dimers
\end{theorem}

The proofs of these ``explicit expression theorems'' all follow the same pattern. In each case, they rely on a key lemma relating the geometric system to the dSKP equation. Then, they consist in applying our general explicit expression dSKP result \cite[Theorem 1.1]{paper1}, written in the context of this paper as Theorem~\ref{theo:expl_sol} below.

We next turn to proving singularity results for geometric systems, whose meaning we now explain.
Let us assume that $T$ is an operator describing the dynamics, whether it is the iteration of the dSKP equation in general or the dynamics of some geometric system.
For some choices of initial data, referred to as \emph{singular},
$T^{-1}$ is not defined, however proceeding forwards with the dynamics
seems possible to some extent. Integrable systems are
  believed to have a common property, known as the \emph{Devron
    property}~\cite{gdevron}, that can be thought of as a variant of
the more common \emph{singularity confinement} property initially introduced in
relation with integrability of discrete equations~\cite{GRP,HV}. The Devron property says the following: if some periodic initial data is singular for the backwards dynamics, it should become singular after a finite number of steps of the forwards dynamics. 

\begin{figure}[tb]
	\centering
	\frame{\includegraphics[height=4.7cm]{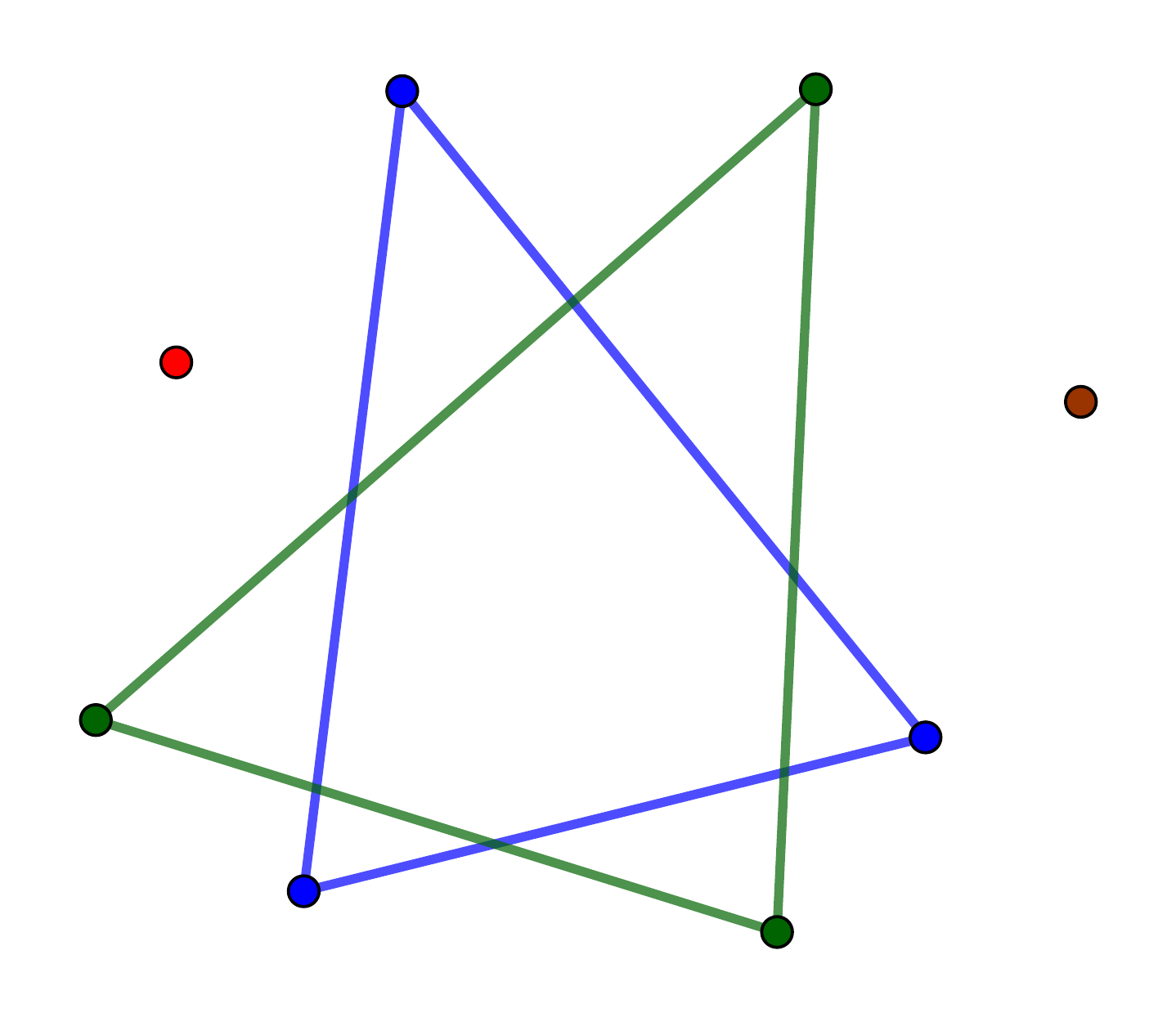}}
	\frame{\includegraphics[height=4.7cm,angle=0]{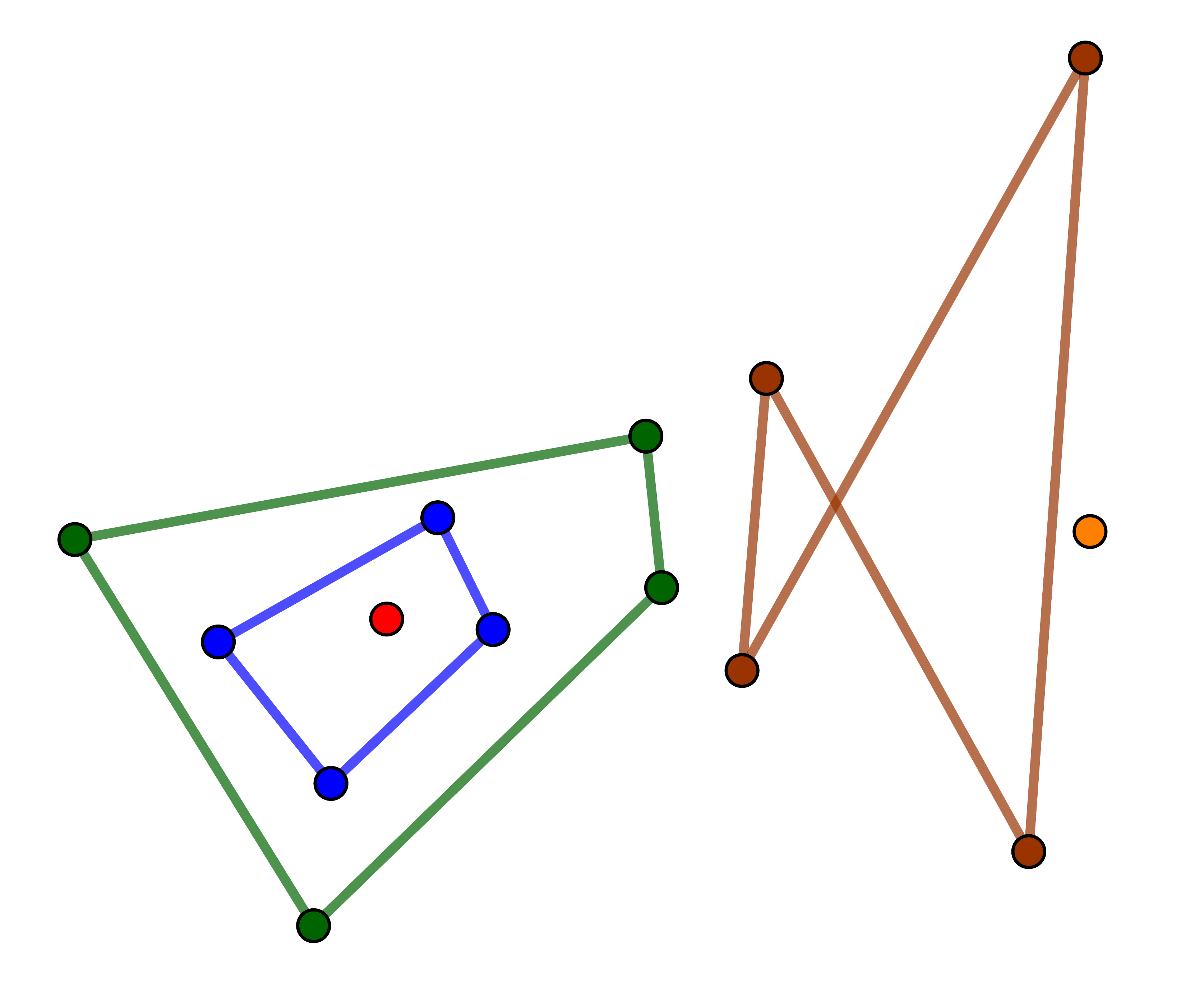}}
	\frame{\includegraphics[height=4.7cm,angle=0]{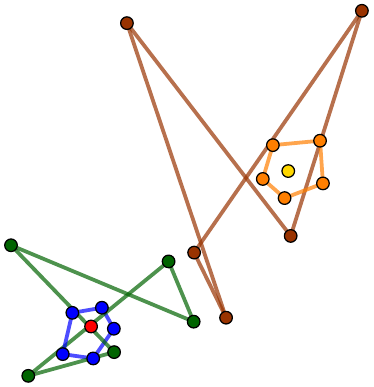}}
	\caption{P-net singularities. The red dot at $0$
			corresponds to $p_0$; the blue, green, brown, orange, yellow dots correspond to
		$p_1$, $p_2$, $p_3$, $p_4$ and $p_5$ respectively. Those are $m$-closed with $m=3$ (left), resp. $m=4$
		(center), resp. $m=5$ (right). Note that $p_m$ is constant, \emph{i.e.}, reduced to a point.}
	\label{fig:pnetsingularity}
\end{figure}

Let us illustrate this, again with the example of P-nets. Let $m\geq 1$, a P-net is said to be \emph{$m$-closed} if, for any $(i,j)\in \Z^2$, $p_{i,j+m}=p_{i,j}$; the reason for this terminology is that $p_0,p_1$ may be seen as discrete curves in $\hC$, and under the $m$-closed condition, those curves close back on themselves after $m$ steps; this closedness property propagates through the P-net evolution. The word \emph{periodic} will be reserved for dSKP initial data to avoid confusion. As a result of Theorem~\ref{theo:pnet_explicit_intro}, $m$-closed P-nets correspond to doubly $m$-periodic initial dSKP data, as can be seen by considering Figure~\ref{fig:pnet_ic} for closed $p_0,p_1$. Let us return to singularities: consider an $m$-closed P-net such that $p_0$ is constant equal to $0$. One checks that this is singular for the backwards dynamics: for any $i\in \Z$, $p_{i,-1}$ is undefined. On the other hand, the forwards dynamics on $p$ seems to be well-defined. For the corresponding dSKP initial data coming from Theorem~\ref{theo:pnet_explicit_intro}, note that there are zeros on the whole layer $k=0$. This kind of dSKP initial data is studied in \cite{paper1}, where the Devron property for dSKP is established. As a direct consequence, we get the Devron property for P-nets: the whole curve $p_m = (p_{i,m})_{i\in \Z}$ also becomes constant, so the forwards dynamics also reaches a singularity at $m$ steps. More precisely we reprove (see Theorem~\ref{theo:pnetsingularity} in the body of the paper):
\begin{theorem}[\cite{gdevron,yao}] Let $m\geq 1$, and
	let $p$ be an $m$-closed P-net such that $p_0\equiv 0$. Assume we can apply the propagation map $T$ to $(p_0,p_1)$ at least $m-1$ times.
	Then, for all $i\in \Z$, we have
	\begin{align}
		p_{i,m} = \Bigl(\frac{1}{m} \sum_{\ell=0}^{m-1} p_{\ell,1}^{-1}\Bigr)^{-1},
	\end{align}
	that is the singularity repeats after $m-1$ steps and its value is the harmonic mean of $p_1$.
\end{theorem}
Geometrically, this corresponds to an \emph{incidence theorem}, where the geometric evolution on singular initial data finishes after a finite number of steps, see Figure~\ref{fig:pnetsingularity}.

In the companion paper \cite[Theorem 1.6]{paper1}, we prove that the dSKP recurrence features the Devron property for a wide family of initial data, that we call \emph{$(m,p)$-Devron initial data}; the importance of this family is that they include as special cases all the singularities of all the geometric systems we consider. We also obtain stronger results for a particularly symmetric case of $(m,1)$-Devron data referred to as \emph{$m$-Dodgson initial data} \cite[Corollary 1.5]{paper1}. For general dSKP, our singularity results are new. However, in several geometric systems that we consider, singularity results have already been obtained by Glick \cite{gdevron} and Yao \cite{yao}. What we provide in this paper is a unified framework, relying on the Devron property of general dSKP, which allows us to identify the number of iterations after which a singular initial data reoccurs and, in some cases, compute the position of the returning singularity. Table~\ref{tab:singularities} summarizes the results of this paper, those that are known and the new ones.

\begin{table}[t]
        \centering
        \begin{tabular}{|l|l|l|l|l|l|}
            \hline
            System & Initial condition & Steps & Reference & Citations \\ \hline
            Miquel dynamics& $m$-Dodgson      & $m-1$      & Theorem \ref{thm:sing_Miquel}   & new       \\ \hline
            P-nets & $m$-Dodgson         & $m-1$      & Theorem \ref{theo:pnetsingularity}   & \cite{gdevron, yao}       \\ \hline
            P-nets & $m$-Dodgson*  & $m-2$      & Theorem \ref{theo:pnetpremature}   & new     \\ \hline
            Int. cr-maps &  $(m,2)$-Devron  &     $2m-2$        &       Theorem \ref{theo:intcrsingular}     &   new        \\ \hline
            D. hol. f., $[m]_2=1$ &  $(m,2)$-Devron &  $2m-2$        &       Theorem \ref{theo:dholsingularity}     &   \cite{yao}        \\ \hline
            D. hol. f., $[m]_2=0$  &  $(m,2)$-Devron   &     $2m-3$        &       Theorem \ref{theo:dholsingularity}      &   new        \\ \hline
            D. hol. f., $[m]_2=0$ &  $(m,2)$-Devron*   &     $2m-4$        &       Theorem \ref{theo:dholpremature}     &   new    \\ \hline
            Orthogonal CP &  $(2m,2)$-Devron  &     $m-2$        &       Corollary \ref{cor:ocpsing}     &   new        \\ \hline
            Polygon recutting &  $(m,1)$-Devron  &     $m-1$        &       Theorem \ref{theo:recutsing}     &   \cite{gdevron}        \\ \hline
            Circle intersection dyn. &  $(m,1)$-Devron  &     $m-1$        &       Theorem \ref{theo:ciddodgson}     &   new        \\ \hline
            Circle intersection dyn. &  $(m,2)$-Devron  &     $2m-4$        &       Theorem \ref{theo:cidsing}     &   new        \\ \hline
            Pentagram map &  $m$-Dodgson  &     $m-1$        &       Theorem \ref{th:pentdodgson}     &   \cite{gdevron, yao}        \\ \hline
            Pentagram map &  $(m,2)$-Devron  &     $2m-4$        &       Theorem \ref{theo:pentsing}     &  new      \\ \hline
                $N$-corrugated pent.~map &  $m$-Dodgson  &     $2m-2$        &       Theorem \ref{th:corpentdodgson}     &   \cite{gdevron, yao}        \\ \hline
                Short diagonal hyp.~map &  $(m,2)$-Devron  &     $m-2$        &       Theorem \ref{theo:short_diagonal_sing}     &  new       \\ \hline
        \end{tabular}
        \caption{An overview of the singularity results presented in this paper and the literature. A * denotes some additional algebraic constraint on the initial conditions.
        }
        \label{tab:singularities}
\end{table}


The generic results for dSKP singularities from \cite{paper1} are recalled in the context of this paper in Section~\ref{sec:prerequ_sing}. The techniques used to obtain them are of a combinatorial nature, and computing the position of a singularity amounts to finding eigenvectors of associated matrices.
In comparison, Glick \cite{gdevron} uses various methods to prove the Devron property, including ``rescaled'' variations of Dodgson condensation for the dKP equation and a similar method for Y-systems, which he then applies to the pentagram maps and its generalizations. He also uses multi-dimensional consistency in the case of polygon recutting. On the other side, Yao \cite{yao} uses lifts to high-dimensional spaces and careful analysis of certain subspaces and invariants to prove the Devron property and the position of the recurring singularity.

Let us note that our singularity theorems provide \emph{upper bounds} for the number of iterations of the respective dynamics. We have done extensive numerical verifications, so we are confident that generically within the respective assumptions, the upper bounds are actually tight. A proof of this would require to provide example solutions in each case. In principle, as we provide explicit formulas for all the dynamics, our methods also provide the algebraic conditions on the initial conditions, such that the dynamics terminate before the upper bound is reached. Apart from the two premature singularity reoccurrences mentioned above, we have not investigated these cases any further.

As a conclusion to this introduction, let us turn to open questions. First, we believe that hexagonal circle patterns with constant intersection angles  \cite{abhexcp, bhhexcp} fall into our framework as well, but we have not investigated them in detail. Moreover, on the one hand, there has not been much research devoted yet to the study of the Devron property, on the other hand, it is easy to verify numerically that it appears in many geometric systems. This naturally leads to a lot of open questions. For example, we propose Conjecture \ref{conj:recutsing} on the position of the recurring singularity in polygon recutting. Moreover, Conjecture \ref{conj:pairsing} proposes a new type of recurring singularity, that does not seem to be a special case of a $(m,p)$-Devron singularity. Answering this conjecture would most likely help to also prove Conjecture \ref{conj:pentsing}, on the geometric nature of the $(m,2)$-Devron singularity in the pentagram map case. We believe these conjectures are well within the power of the general dSKP framework we developed.

There are also some geometric systems that feature the Devron
property, for which it is not clear whether they are describable via
the dSKP equation. In particular, we mention a higher dimensional
version of discrete holomorphic functions in Remark
\ref{rem:disosing}. There are also three conjectures in \cite[Section
9]{gdevron}. The first we prove in Theorem \ref{theo:cidsing}.
The second concerns a generalization of the pentagram map introduced
by Khesin and Soloviev \cite{kspent}; we prove it in
Theorem~\ref{theo:short_diagonal_sing}, using results of Glick and
Pylyavskyy \cite{gpymeshes} in combination with our framework.
The last conjecture mentioned by Glick on the so called \emph{Schubert
  flip} is open, and it is not clear whether there is a relation to the dSKP equation.
  
Another family of open problems concerns systems governed by multi-ratio equations for which it is not immediately obvious how they are governed by dSKP, see for example \cite{ksmultiratios, schiefdarboux}. However, in \cite{athesis}, it was shown that these systems are actually special cases of dSKP on a higher dimensional lattice, that is $A_n$ with $n > 3$. So far, we have only given explicit solutions on $A_3 = \mathcal L$. We suspect that it is possible to solve dSKP on general $A_n$ as well, possibly even using oriented dimers, but we do not know. Certainly, this would be an interesting direction of research.

Yet another open problem is to relate our methods with previous work that deals with periodic discrete holmorphic functions, see \cite{hmnpperiodic}. In this periodic case, it is possible to define propagation in a different direction, perpendicular to the coordinate axes instead of diagonally. This is called \emph{cross-ratio dynamics}, and the integrability of these dynamics was shown in \cite{afitcrossratio}. Then, in  \cite{agrcrdyn} it was shown how these dynamics can be captured by combining dSKP, dimer cluster integrable systems and geometric R-matrices. It would be interesting to see if it is possible to combine the techniques of the current paper with those developed in the study of R-matrices \cite{ilprmatrices}, to obtain explicit expressions for cross-ratio dynamics as well.

Finally, there is an intriguing coincidence of the evaluation of two Aztec diamonds of different sizes and with different weights. Both describe the propagation of discrete holomorphic functions, once via P-nets and once via integrable cross-ratio maps, see Question \ref{que:pnetdhol}.

\subsection*{Outline of the paper}
In Section~\ref{sec:prereq} we recall the prerequisites from the companion paper~\cite{paper1}. Then, the following sections are dedicated to the different geometric systems. Section~\ref{sec:miquel}: Miquel dynamics, Section~\ref{sec:pnets}: P-nets, Section~\ref{sec:Backlund_pairs}: integrable cross-ratio maps, Section~\ref{sec:dhol}:
discrete holomorphic functions and orthogonal circle patterns, Section \ref{sec:recut}: polygon recutting, Section \ref{sec:cid}: circle intersection dynamics, Section~\ref{sec:pent}: $N$-corrugated polygons and pentagram map, Section~\ref{sec:hyppent}: short diagonal hyperplane map.
All geometric systems sections essentially follow the same structure which consists of: describing the geometric system, establishing the key lemma relating it to the dSKP equation, proving the explicit expression result, and then showing singularity results.

\subsection*{Acknowledgments}
We would like to thank the anonymous referees for their thorough reading of the manuscript and for many useful remarks that helped increase the quality of the paper.
The first author would like to thank Max Glick and Sanjay Ramassamy for helpful discussions. He is supported by the Deutsche Forschungsgemeinschaft (DFG) Collaborative Research Center TRR 109 ``Discretization in Geometry and Dynamics'' as well as the by the MHI and Foundation of the ENS through the ENS-MHI Chair in Mathematics. The second and third authors
are partially supported by the DIMERS project ANR-18-CE40-0033 funded by the French National Research Agency. The authors would also like to thank Claude-Michel Viallet for interesting discussions on singularity confinements.

Data sharing not applicable to this article as no datasets were generated or analysed during the current study.

\section{Prerequisites}
\label{sec:prereq}

In this section we recall the results of the first paper
\cite{paper1}, in special cases that appear in geometric systems.

The dSKP recurrence lives on vertices of the \emph{octahedral-tetrahedral
  lattice $\calL$} defined as:
        \begin{align*}
                \calL = \left\{p=(i,j,k) \in \Z^3 : i+j+k \in 2\Z \right\}.
        \end{align*}

\begin{definition}\label{def:dSKP_recurrence}
        A function $x: \calL \rightarrow \hC$ satisfies the
        \emph{dSKP recurrence}, if
        \begin{align}\label{eq:dskp_x}
          \frac{(x_{-e_3}-x_{e_2})(x_{-e_1}-x_{e_3})(x_{-e_2}-x_{e_1})}{
          (x_{e_2}-x_{-e_1})(x_{e_3}-x_{-e_2})(x_{e_1}-x_{-e_3})} = -1,
        \end{align}
        holds evaluated at every point $p$ of $\Z^3\setminus \calL$, that
        is $x_q$ stands for $x_q(p):=x(q+p)$ for every $q\in (\pm e_i)_{i=1}^3$.
        More generally, if $A \subset \calL$ and $x:A\to \hat{\C}$, we
        say that $x$ \emph{satisfies the dSKP recurrence on} $A$ when
        \eqref{eq:dskp_x} holds whenever all the points are in $A$.
\end{definition}

The target space $\hC$ of the dSKP recurrence is an affine chart of the \emph{complex projective line} $\CP^1$, defined as follows. Consider the equivalence relation $\sim$ on $\C^2$ such that for $v,v'\in \C^2$ we have $v\sim v'$ if there is a $\lambda \in \C\setminus \{0\}$ such that $v = \lambda v'$. Every point in the projective line is an equivalence class $[v] = \{v' : v' \sim v\}$ for some $v \in \C^2 \setminus \{(0,0)\}$, that is
\begin{align*}
\CP^1= \{[v] : v \in \C^2 \setminus \{(0,0)\} \} = \left(\C^2 \setminus \{(0,0)\}\right)/\sim.
\end{align*}

It is practical to consider an \emph{affine chart} $\C$ of $\CP^1$, and the set $\hC = \C \cup \{\infty \}$ which we identify with $\CP^1$. 
Every point $z\in \C \subset \hC$ corresponds to $[z,1]$ in $\CP^1$ and $\infty\in \hC$ corresponds to $[1,0]$. In $\hC$ one can perform the usual arithmetic operations on $\C$. One can even apply the naive calculation rules $z + \infty = \infty, z/\infty = 0$ etc., see \cite[Section 17]{rgbook}. Note that solutions to the dSKP equation are invariant under \emph{projective transformations}, and the use of dSKP to get the value on a new point as in Figure~\ref{fig:octa_evol} has a natural interpretation in terms of a \emph{projective involution}, see \cite[Remark 2.3]{paper1}.

Given $(a_{i,j})_{(i,j)\in\Z^2}$, define the \emph{initial condition}
\begin{equation}\label{equ:init_cond}
\forall\ (i,j)\in\Z^2,\quad x(i,j,[i+j]_2)=a_{i,j},
\end{equation}
where $[n]_p$ generically denotes the value of $n$ modulo $p$,
taken in $\{0,\dots, p-1\}$. The idea is that if $x$ satisfies the
dSKP recurrence on $\calL$, its values are determined by the initial
conditions $\left(a_{i,j}\right)_{(i,j)\in\Z^2}$. Giving a
combinatorial expression for this function of the initial conditions
is the point of the next section.

Note that in the following sections we assume that all initial data for dSKP are generic, in the sense that the data propagates to the whole lattice $\calL$ as explained in the introduction. We consider propagation in the positive $k$ direction, but the negative direction is clearly symmetric. When we study singularities, we assume that the data is as generic as possible under the assumption of the initial singularity. We will not repeat these assumptions in the following.

\subsection{Explicit solution} \label{sec:explicit_solution_gen}

Consider the square lattice, with faces indexed by $\Z^2$. The face
indexed by $(i,j)$ is equipped with the weight $a_{i,j}$. For
$k\geq 1$, define the \emph{Aztec diamond of size $k$ centered at
  $a_{i,j}$}, denoted by\footnote{These two notations
    are useful to emphasize either the position of the central face or its
    weight, although we will only use the first one in Section~\ref{sec:hyppent}.} $A_k(i,j)$ or $A_k[a_{i,j}]$, to be the subgraph centered at $(i,j)$ whose
\emph{internal faces} are those with labels $(i',j')$ such that
$|i'-i| + |j'-j| < k$, see Figure~\ref{fig:ad}. We call \emph{open faces} those with label $(i',j')$
such that $|i'-i| + |j'-j| = k$, and denote by $F$ the disjoint union of all internal and all open faces of $A_k[a_{i,j}]$.
Let $V$ be the set of vertices of $A_k[a_{i,j}]$, that is, all vertices that belong to at least one internal face. These vertices can themselves be colored either black or white depending on their parity, making $A_k[a_{i,j}]$ a bipartite graph. We denote such a decomposition by $V=B\sqcup W$, see again Figure~\ref{fig:ad}. 

\begin{figure}[tb]
  \centering
  \includegraphics[width=12cm]{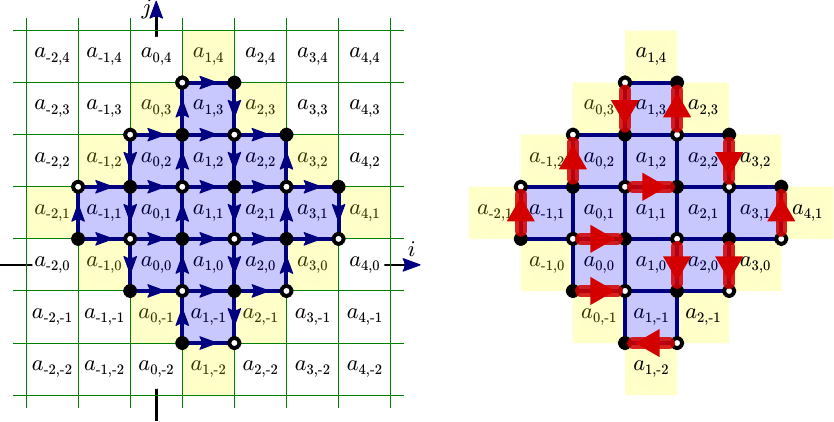}
  \caption{Left: the Aztec diamond of size $3$, $A_{3}[a_{1,1}]$, in blue, shown as a
    bipartite graph with a Kasteleyn orientation, with its
    internal faces in blue, and its open faces in yellow. Right: an
    example of an oriented dimer configuration on this graph. The
    ratio of partition functions is equal to $x(1,1,4)$.}
  \label{fig:ad}
\end{figure}

In general, for a finite,
planar bipartite graph $G$, let $\vec{E}$ be
the set of its directed edges.
A \emph{Kasteleyn orientation} \cite{Kasteleyn} is, in our case, a skew-symmetric
function $\vphi$ from $\vec{E}$ to $\{-1,1\}$ such that, for every
internal face $f$ of degree $2k$ we have
\[
  \prod_{wb\in\partial f}\vphi_{(w,b)}= (-1)^{k+1}.
\]
This corresponds to an orientation of edges of the graph: an edge
$e=wb$ is oriented from $w$ to $b$ when $\varphi_{(w,b)}=1$, and from
$b$ to $w$ when $\varphi_{(w,b)}=-1$. By Kasteleyn~\cite{Kasteleyn2}, such an
orientation of $G$ exists; in our case, an example is displayed in
Figure~\ref{fig:ad}.

An \emph{oriented dimer configuration} of $A_k[a_{i,j}]$ is a subset of oriented
edges $\vec{\Ms}$ such that its undirected version $\Ms$ is a dimer
configuration - that is, it touches every vertex exactly once. Denote
by $\vec{\M}$ the set of oriented dimer configurations of
$A_k[a_{i,j}]$.

Given a Kasteleyn orientation $\varphi$, and an oriented edge $\vec{e}$, denote by $f(\vec{e})$ the (internal or open) face to the right of $\vec{e}$.
The \emph{weight} of an oriented dimer configurations $\vec{\Ms}$ is
\begin{equation}
  \label{eq:defordimweight}
  w(\vec{\Ms}) = \prod_{\vec{e} \in \vec{\Ms}} \varphi_{\vec{e}} \ a_{f(\vec{e})},
\end{equation}
and the corresponding \emph{partition function} is
\begin{equation*}
  Z(A_k[a_{i,j}],a,\varphi) = \sum_{\vec{\Ms}\in\vec{\M}} w(\vec{\Ms}).
\end{equation*}

We consider the weighted ratio of partition functions, defined by
\begin{equation}\label{eq:defYpref}
  Y(A_k[a_{i,j}],a) =   \biggl(\,\prod_{f\in F}a_f\biggr) \
  \frac{Z(A_k[a_{i,j}],a^{-1},\varphi)}{Z(A_k[a_{i,j}],a,\varphi)}.
\end{equation}
It is also referred to as the \emph{ratio function of oriented dimers}.
Note that the partition function in the numerator uses the
inverse face weights $a^{-1}_{i,j}$ to define the weight of
configurations \eqref{eq:defordimweight}. The ratio of partition function does not depend on the choice of $\varphi$, see
\cite[Proposition~3.2]{paper1}.

All these definitions can be extended consistently to Aztec diamonds of size $0$: they have no internal face and only one open face, and a single (empty) oriented dimer configuration, which has weight $1$; taking the prefactor of \eqref{eq:defYpref} into account, this gives $Y(A_0[a_{i,j}],a) = a_{i,j}$.

The following is a special case of \cite[Theorem~3.4]{paper1},
detailed in \cite[Example~3.5]{paper1}.
\begin{theorem}[\cite{paper1}]
  \label{theo:expl_sol}
  If $x: \calL \rightarrow \hC$ satisfies the dSKP recurrence with initial condition~\eqref{equ:init_cond}, then
  for every $(i,j,k) \in \calL$ such that $k\geq 1$,
  \begin{equation*}
    x(i,j,k) = Y(A_{k-1}[a_{i,j}],a).
  \end{equation*}
\end{theorem}

In~\cite{paper1}, we also prove an algebraic expression for the solution of the dSKP recurrence.
To state it, it is
easier to rotate the previous Aztec diamond by $45$ degrees, as in
Figure~\ref{fig:D}.
Recall that $B$ denotes the set of black vertices
of this Aztec diamond, and that $F$ denotes the set of (internal or open)
faces. Each face $f\in F$ corresponds to a weight, simply denoted by
$a_f$\footnote{For the following definition, it is more appropriate to index the weight of the face $f$ by its name rather than by its coordinates.}. We define a linear operator $D$ going from $\C^B \oplus \C^B$
to $\C^F$, in the following way: identify the canonical basis of $\C^B \oplus \C^B$ with two independent ``copies'' of $B$, and that of $\C^F$ with $F$. Then we define the matrix entries $D_{f,b}$ in these bases, that is, for $f\in F$ and for $b$ in either of the two copies of $B$. If $b$ is in the first
copy of $B$ and $f\in F$ is adjacent to $b$, $D_{f,b}:=\pm 1$, with
$+1$ if $b$ is on the right or above $f$ and $-1$ otherwise (relative to the orientation obtained by rotating the Aztec diamond by 45 degrees as in Figure~\ref{fig:D}). If $b'$
is in the second copy of $B$ and $f\in F$ is adjacent to $b'$,
$D_{f,b'}:=\pm a_f$, with the same sign rule as before, see Figure~\ref{fig:D}. All the other entries are zero.

For example, let $b$ be the bottom-left black vertex in Figure~\ref{fig:D}, and let $f_W,f_E,f_N$ be the (inner or open) faces respectively West, East and North of $b$. Let also $b'$ be the counterpart of $b$ in the second copy of $B$. Then we have
\begin{alignat*}{3}	
	& D_{f_W,b} = +1, && D_{f_E,b} = -1, && D_{f_N,b} = -1, \\
	& D_{f_W,b'} = a_{-2,1}, \ \ && D_{f_E,b'} = - a_{-1,0}, \ \ && D_{f_N,b'} = a_{-1,1}.
\end{alignat*}

\begin{figure}[tb]
  \centering
  \includegraphics[width=6cm]{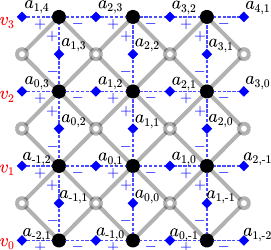}
  \caption{ Aztec diamond $A_3[a_{1,1}]$ rotated by 45 degrees.
    The elements
    of $B$ are shown as black dots, and those of $F$ are shown as blue
    diamonds. The non-zero entries of the operator $D$ correspond to
    the blue dashed lines, and their sign is indicated. The red labels
    are entries of the vector $v\in \C^F$ in the statement of
    Theorem~\ref{theo:D}.}
  \label{fig:D}
\end{figure}

Note that $2|B|=|F|-1$, so $D^T$ always has a
nontrivial kernel. We can now express the dSKP solution, using
\cite[Theorem~5.3]{paper1} after a change of index convention.

\begin{theorem}[\cite{paper1}] \label{theo:D}
  Let $x: \calL \rightarrow \hC$ satisfy the dSKP recurrence with initial condition~\eqref{equ:init_cond}.
  Let $(i,j,k)\in\calL$ such that $k\geq 1$, and $D$ be the
  operator defined from $(A_{k-1}[a_{i,j}],a)$.
  Let $v \in \C^F$ be a non-zero vector such that
  \begin{equation}
    D^T \ v = 0.
  \end{equation}
  Denote by $v_{0},\dots, v_{k-1}$ the entries of $v$ corresponding to
  the $k$ leftmost elements of $F$, with respective face weights
  $a_{i-k+1,j},\dots, a_{i,j+k-1}$. Then, the ratio function of oriented dimers
  can be expressed as:
  \begin{equation*}
    Y(A_{k-1}[a_{i,j}],a) = \frac{\sum_{\ell =0}^{k-1} a_{i-k+\ell+1,j+\ell}\, v_{\ell}}{\sum_{\ell=0}^{k-1} v_{\ell}}.
  \end{equation*}
\end{theorem}

As argued in the proof of \cite[Theorem~5.3]{paper1},
this statement should be understood as an equality of expressions in
formal variables $(a_{i,j})$; the entries $(v_\ell)$ are themselves
rational functions of the $(a_{i,j})$. One might wonder if this equality holds
when $\dim \ker D^T \geq 2$. The following proposition, combined  with
the fact that there exist global solutions to the dSKP recurrence
(see \cite[Example~2.4]{paper1}) implies that
this is generically not the case: if one sees $D^T$ as a formal matrix
in the $(a_{i,j})$ variables,
then its kernel is always one-dimensional.

\begin{proposition}\label{prop:dimker2}
With the same notation as in Theorem~\ref{theo:D}, suppose that the
initial conditions $(a_{i,j})$ are such that $\dim \ker D^T \geq
2$. Then $Y(A_{k-1}[a_{i,j}],a)$ is undefined in $\hC$, and so is
$x(i,j,k)$.
\end{proposition}
\begin{proof}
  Suppose that there are two linearly independent vectors
  in $\ker D^T$. By a linear combination, there is a non-zero vector
  $v \in \ker D^T$ such that $\sum_{\ell=0}^{k-1} v_{\ell}=0$, so by
  Theorem~\ref{theo:D}, $Y(A_{k-1}[a_{i,j}],a)$ has to be $\infty$ or
  undefined (if at the same time $\sum_{\ell =0}^{k-1} a_{i-k+\ell,j+\ell}\, v_{\ell}=0$). Similarly, we can get another non-zero vector
  $v'\in \ker D^T$ such that $\sum_{\ell =0}^{k-1} a_{i-k+\ell,j+\ell}\, v'_{\ell}=0$,
  so that $Y(A_{k-1}[a_{i,j}],a)$ has to be $0$ or
  undefined (if at the same time $\sum_{\ell=0}^{k-1} v'_{\ell}=0$). Since Theorem~\ref{theo:D} holds for any choice of non-zero vector in $\ker D^T$, the only solution that satisfies both conditions is that it is undefined. Using
  Theorem~\ref{theo:expl_sol}, $x(i,j,k)$ is undefined as well.
\end{proof}

\subsection{Singularities}\label{sec:prerequ_sing}

We now recall the results of \cite{paper1} about singular initial
conditions. Note that we return to the unrotated version of the Aztec diamond as in Figure~\ref{fig:ad}.

In some special cases, we prove a simpler
expression for the solution $x(i,j,k)$ than the one given by
Theorem~\ref{theo:D}. This is the case when there is a constant $d \in \C$ such that $a_{i,j}=d$
whenever $(i,j)\in\Z^2$ and $[i+j]_2=0$. Then $x(i,j,k)$ can be expressed in terms of the
inverse of a matrix $N$, of size $k\times k$, whose coefficients are the shifted inverses of the non-zero initial conditions
that enter into the computation of $x(i,j,k)$.
The following is taken from \cite[Corollary~5.10]{paper1}.
\begin{proposition}[\cite{paper1}]
  \label{prop:Nmat}
  Suppose that for all $(i,j)\in \Z^2$ such that $[i+j]_2=0$,
  $a_{i,j}=d$. Let $(i,j,k) \in \calL$ with $k\geq 1$. Consider the
  matrix $N=\left( N_{i',j'} \right)_{0\leq i',j' \leq k-1}$ with entries
  \begin{equation*}
        N_{i',j'} =\frac{1}{a_{i-i'+j',j+k-1-i'-j'}-d}.
  \end{equation*}
  Then
  \begin{equation*}
    x(i,j,k) = d + \sum_{0\leq i',j' \leq k-1} N^{-1}_{i',j'},
  \end{equation*}
where the sum is over entries of the inverse matrix $N^{-1}$.
\end{proposition}
For example, if $d=0$, in the case of Figure~\ref{fig:ad} (that is
$(i,j,k)=(1,1,4)$), this matrix is
\begin{equation*}
  N =
  \begin{pmatrix}
    a_{1,4}^{-1} & a_{2,3}^{-1} & a_{3,2}^{-1} & a_{4,1}^{-1} \\
    a_{0,3}^{-1} & a_{1,2}^{-1} & a_{2,1}^{-1} & a_{3,0}^{-1} \\
    a_{-1,2}^{-1} & a_{0,1}^{-1} & a_{1,0}^{-1} & a_{2,-1}^{-1} \\
    a_{-2,1}^{-1} & a_{-1,0}^{-1} & a_{0,-1}^{-1} & a_{1,-2}^{-1}
  \end{pmatrix}.
\end{equation*}

We now turn to initial conditions with periodicity, which occur in singular closed geometric systems. The following definitions are illustrated in Figure~\ref{fig:sing_data}.
\begin{definition}\label{def:sing}
	
	Let $m\geq 1$. Initial conditions $(a_{i,j})_{(i,j)\in \Z^2}$ are said to be $m$-\emph{simply periodic} if 
	\begin{equation*}
		\forall (i,j)\in \Z^2, \ a_{i,j}=a_{i+m,j+m}.
	\end{equation*}
	They are said to be $m$-\emph{doubly periodic} if
	\begin{equation*}
	    \forall (i,j)\in \Z^2, a_{i,j}=a_{i+m,j+m}=a_{i+m,j-m}.
	\end{equation*}
 	We say that they are
	$m$-\emph{Dodgson initial conditions} if they are $m$-doubly periodic
	and in addition, there is a constant $d\in \hC$ such that
	\begin{equation*}
		\forall (i,j)\in\Z^2 \text{ with } [i+j]_2=0, \ a_{i,j}=d.
	\end{equation*}
	For $m,p\geq 1$, we say that they are
	$(m,p)$-\emph{Devron initial conditions} if they are $m$-simply
	periodic and in addition, 
	\begin{equation*} 
		\forall (i,j)\in \Z^2 \text{ with } [i-j]_{2p}=0, \ a_{i,j} = a_{i+1,j+1}.
	\end{equation*}
\end{definition} 
Note that $(m,p)$-Devron initial conditions amount to having every $p$-th SW-NE
diagonal at height $0$ constant, see Figure~\ref{fig:sing_data},
bottom right.

Then, as a corollary to Proposition~\ref{prop:Nmat}, we prove that for $m$-Dodgson initial conditions at
height $m$ (that is, after $m-1$ iterations of the dSKP recurrence),
the same kind of special conditions reappear. This implies that the
values of $x$ at height $m+1$ are not defined. The following is part of~\cite[Theorem~1.4]{paper1}.
\begin{corollary}[\cite{paper1}]\label{thm:Dodgson_prerequ}
  For $m$-Dodgson initial conditions, for every $(i,j)\in\Z^2$ such that $(i,j,m) \in \calL$, $x(i,j,m)$ is
  independent of $i,j$.
\end{corollary}

\begin{figure}[tb]
  \centering
  \includegraphics[width=11cm]{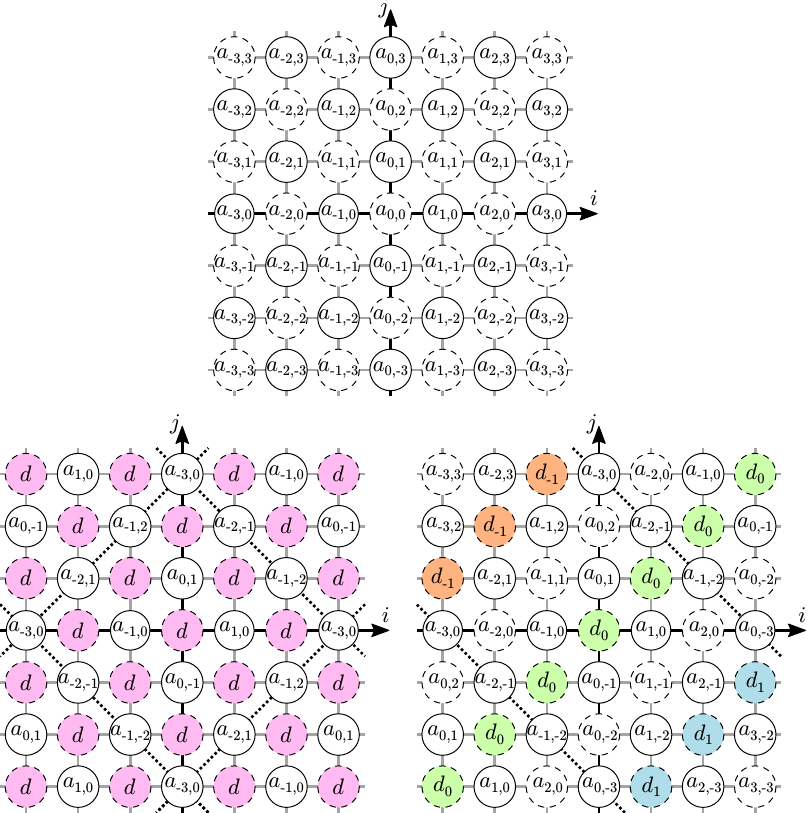}
  \caption{Initial data $(a_{i,j})$ for the dSKP recurrence. Variables in dashed circles lie
    at height $k=0$, while those in solid circles lie at height
    $k=1$. Top: generic. Bottom left: $3$-Dodgson; an
    elementary pattern is shown as a dashed square, and particular
    values at height $0$ are shown in purple. Bottom right:
    $(3,2)$-Devron; constant diagonals are shown in
    orange, green, blue.}
  \label{fig:sing_data}
\end{figure}

In some cases it is possible to give a simple expression for this
constant value on the layer at height $m$. More precisely, we prove such an expression when the NW-SE diagonals at height $1$
  contain data that are cyclic permutations of each other. This corresponds to
\cite[Corollary~1.5]{paper1}.

\begin{corollary}[\cite{paper1}]
  \label{cor:harm_mean}
  Suppose that the initial conditions are $m$-Dodgson, and suppose in addition
  that $d=0$ and that for some $p\not\in m\Z$, when $[i+j]_2=1$,
  $a_{i,j}=a_{i+p+1,j-p+1}$. Then, for all $(i,j)\in\Z^2$ such that $(i,j,m)\in\calL$, $x(i,j,m)$ is the harmonic mean of the $m$ different
  values of the initial data:
  \begin{equation*}
    x(i,j,m) = \left( \frac{1}{m} \sum_{i=0}^{m-1} a_{i,1-i}^{-1} \right) ^{-1}.
  \end{equation*}
\end{corollary}

For $(m,p)$-Devron initial conditions, we show in \cite[Theorem~1.6]{paper1} that this kind of special
conditions reappears at height $(m-2)p+2$ (that is, after $(m-2)p+1$
applications of the dSKP recurrence):
\begin{theorem}[\cite{paper1}]
  \label{theo:devron_sing}
  For $(m,p)$-Devron initial data, let $k=(m-2)p+2$. Then, for all $(i,j)\in\Z^2$ such that $[i-j-mp]_{2p}=0$, we have
  \begin{equation*}
    x(i,j,k) = x(i+1,j+1,k).
  \end{equation*}
\end{theorem}

\section{Miquel dynamics} \label{sec:miquel}

\subsection{Circle patterns and Miquel dynamics}

Any line or circle in the Euclidean sense in $\C \subset \hC$ is considered to be a (generalized) \emph{circle} in $\CP^1$. It is straightforward to verify that this definition of circles is invariant under projective transformations, as defined in \cite[Remark 2.3]{paper1}. \emph{Möbius transformations} are all transformations that are compositions of projective transformations and complex conjugations $z \mapsto \bar z$. Clearly, the definition of circles above is also more generally invariant under all Möbius transformations. In an affine chart $\hC$ the \emph{center} of a circle in the Euclidean sense is the Euclidean center and the center of a Euclidean line is $\infty$. Circle centers are not a projective or Möbius invariant notion but they are still useful in a fixed affine chart.

\begin{definition}\label{def:cp}
A \emph{(square grid) circle pattern} is a map $p: \Z^2\rightarrow \hC$, such that for all $(i,j)\in \Z^2$, there is a circle $c_{i,j}$ such that $p_{i,j},p_{i+1,j},p_{i+1,j+1},p_{i,j+1}\in c_{i,j}$.
Denote by $t_{i,j}$ the center of the circle $c_{i,j}$, by $c:\Z^2\rightarrow \{\text{Circles of $\CP^1$}\}$, the map corresponding to circles, and by $t:
\Z^2 \rightarrow \C$ the one corresponding to circle centers.
\end{definition}

In the generic case, having both the circles $c_{i,j}$ and the intersection points $p_{i,j}$ is redundant, as the circles define the intersection points and vice versa. However we will look at non-generic cases later, so it is handy to have both descriptions ready. Note that we will use the terminology \emph{circle pattern} both for the map $p$ and for the map $c$.

\begin{definition}\label{def:t_embedding}
A \emph{(square grid) \emph{t-realization}} is a map $t: \Z^2\rightarrow \C$, such that, for all $(i,j)\in\Z^2$,
\begin{align}
\frac{(t_{i+1,j} - t_{i,j})(t_{i-1,j} - t_{i,j})}{(t_{i,j+1} - t_{i,j})(t_{i,j-1} - t_{i,j})} \in \R.
\end{align}
\end{definition}
We use the term t-\emph{realization} rather than t-\emph{embedding} \cite{clrtembeddings} since in Definition~\ref{def:t_embedding} we do not ask the full requirements of a t-embedding, \emph{i.e.}, we do not ask that it corresponds to an embedded graph with convex faces. Note that t-embeddings have also previously appeared under the name of Coulomb gauge in \cite{klrr}. The term t-realization has been used as a relaxation of t-embedding before \cite[Section~4.1]{clrlorentz}, although we relax even more as we do not require the quantity in Definition~\ref{def:t_embedding} to be real positive. In the discrete differential geometry community, t-realizations are the planar case of conical nets, which are also studied separately \cite{muellerconical}.

The circle centers of a circle pattern are a t-realization \cite{amiquel, klrr}. A t-realization does not uniquely define a circle pattern, but it does so up to the choice of one of the intersection points.

\begin{figure}[tb]
  \centering
  \includegraphics[scale=0.38]{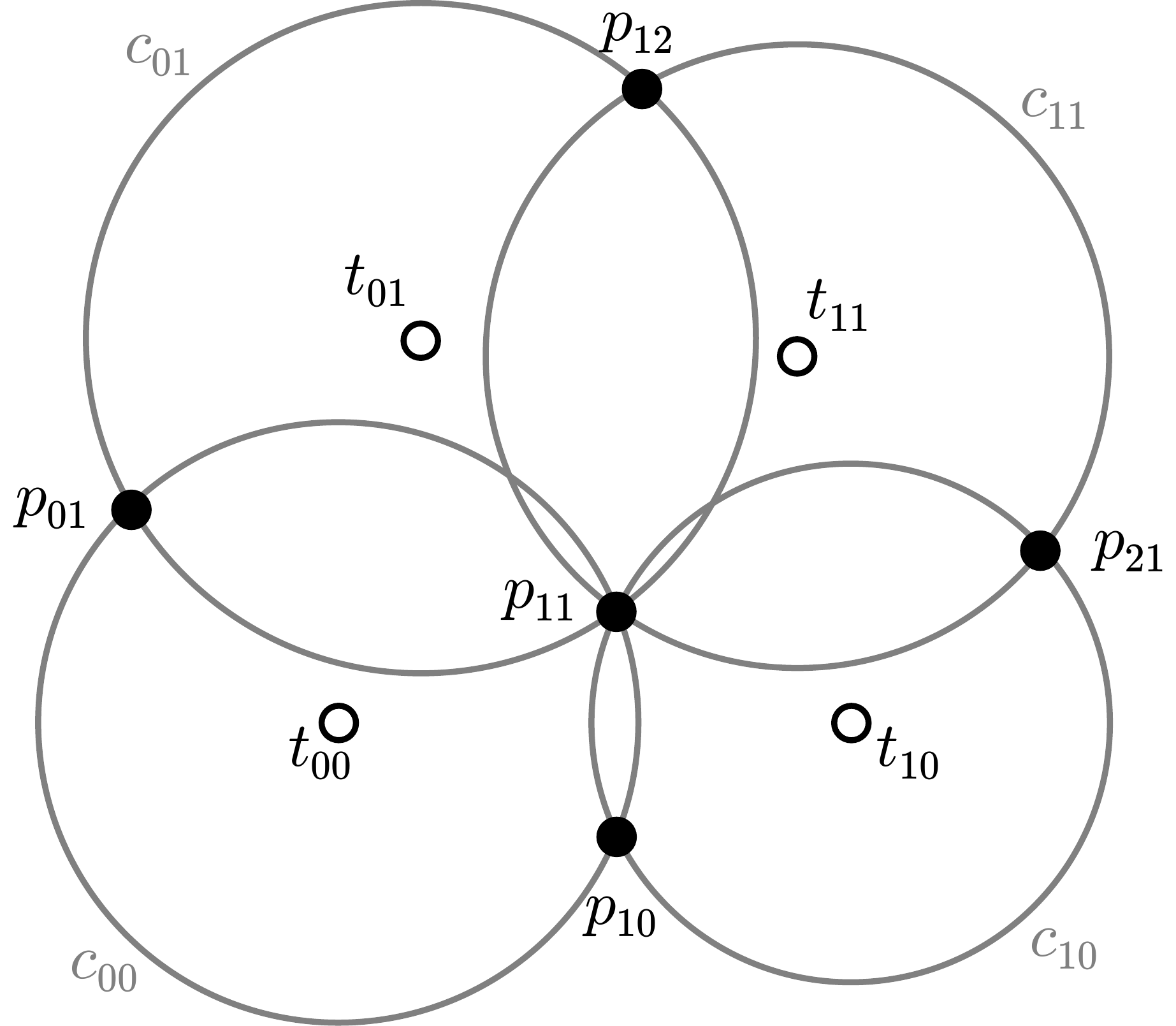}
  \hspace{5mm}
  \frame{\includegraphics[scale=0.5]{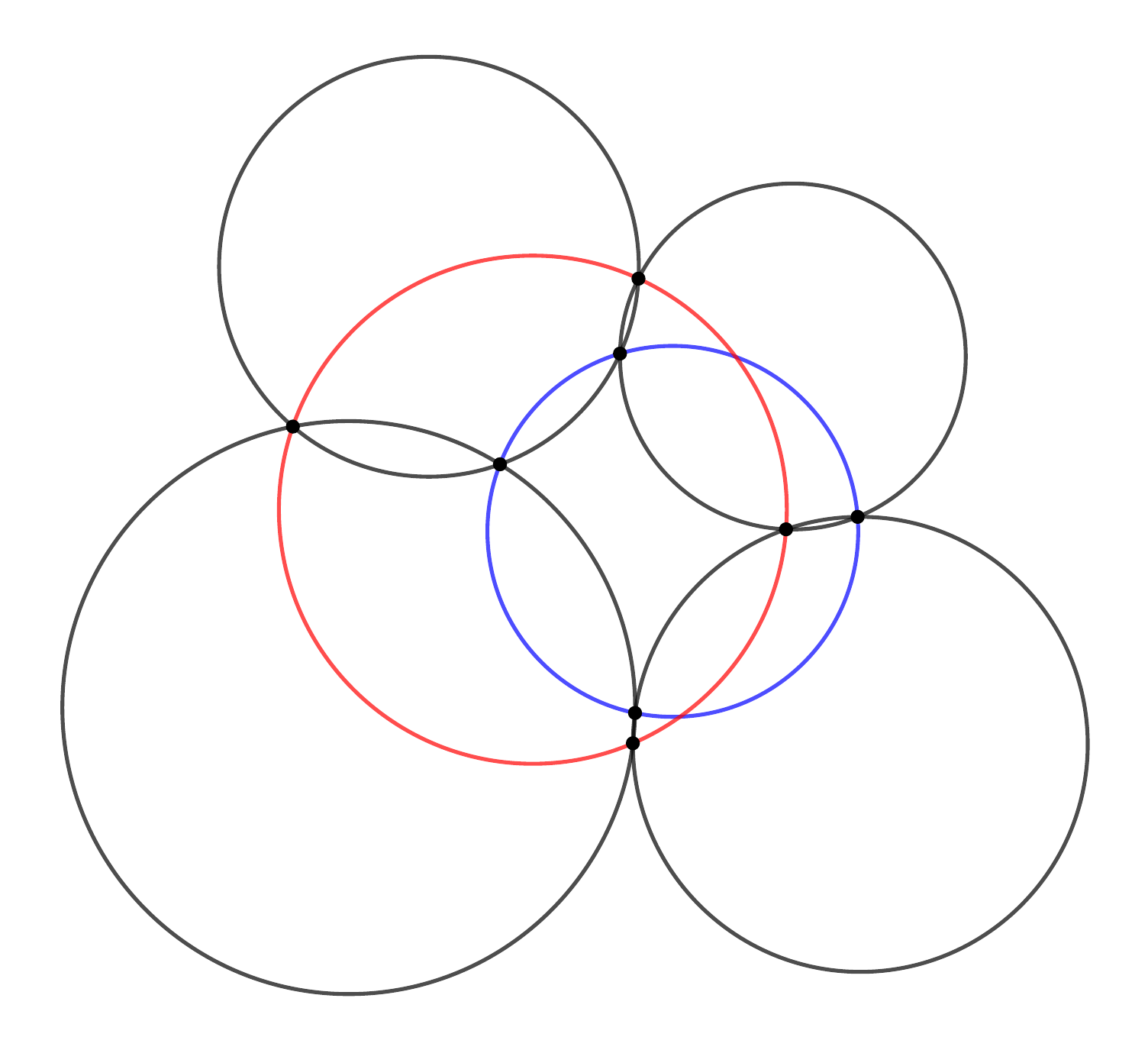}}
  \caption{Left: Labeling in a circle pattern. Right: The six circles of a Miquel configuration.}
  \label{fig:miquelconfig}
\end{figure}

In a square grid circle pattern, every circle has exactly four neighbouring circles, see Figure \ref{fig:miquelconfig}. These four circles intersect in eight points, four of which belong to the original circle. Miquel's six circles theorem \cite{miquel} states that the four other points belong to a sixth circle. Following~\cite{ramassamymiquel}, we can now introduce Miquel dynamics. The \emph{parity of a circle} $c_{i,j}$ is defined to be the parity of $i+j$.

\begin{definition}
\emph{Miquel dynamics} is the dynamics $T$ mapping circle patterns to circle patterns such that, for every $k\geq 1$, $T^k(c)$ is $T^{k-1}(c)$, except that if $k$ is even (resp. odd) every odd (even) circle is replaced  with the sixth circle that exists due to Miquel's theorem; $T^k(t)$ denotes the corresponding dynamics on circle centers.
\end{definition}

A relation of Miquel dynamics to dSKP is given by the next lemma.

\begin{lemma}[\cite{klrr,amiquel}]\label{lem:miqueldskp}
Let $t$ be a t-realization given as the centers of a circle pattern~$c$. Then, for all $(i,j)\in\Z^2$ such that $[i+j]_2=0$, we have
\begin{align}
\frac{(t_{i,j}-t_{i+1,j})(t_{i,j+1}-T(t)_{i,j})(t_{i-1,j}-t_{i,j-1})}{(t_{i+1,j}-t_{i,j+1})(T(t)_{i,j}-t_{i-1,j})(t_{i,j-1}-t_{i,j})} = -1. \label{eq:miqueldskp}
\end{align}
\end{lemma}

A consequence of Lemma \ref{lem:miqueldskp} is that $T(t)$ only depends on the circle centers $t$ and needs no additional information on the actual circle pattern $c$.

\subsection{Explicit solution}

Using Theorem~\ref{theo:expl_sol}, we provide an explicit expression for $T^k(t)$ using the ratio function of oriented dimers of the Aztec diamond. Recall from Section~\ref{sec:explicit_solution_gen} that for the square lattice $\Z^2$ with face weights $(a_{i,j})_{(i,j)\in\Z^2}$, $A_{k}[a_{i,j}]$ denotes the Aztec diamond of size $k$ centered at $a_{i,j}$.

\begin{theorem}\label{theo:explmiquel}
Let $t:\Z^2\rightarrow \C$ be a t-realization, and consider the graph $\Z^2$ with face-weights $(a_{i,j})_{(i,j)\in\Z^2}$ given by
\[
a_{i,j}=t_{i,j}.
\]
Then, for all $(i,j)\in\Z^2,\,k\geq 1$ such that $[i+j+k]_2=1$, we have
\begin{align*}
T^k(t)_{i,j}=Y(A_k[t_{i,j}],t).
\end{align*}
\end{theorem}

\begin{proof}
Consider the function $x:\calL \to \hC$ given by
        \begin{align}
                x(i,j,k) = \begin{cases}
                        T^{k-1}(t)_{i,j} & \mbox{for } k > 1,\\
                        t_{i,j} & \mbox{for } k\in\{0,1\},
                \end{cases}
        \end{align}
for every $(i,j,k)$ such that $i+j+k\in2\Z$, $k\geq 0$.
As a consequence of Lemma~\ref{lem:miqueldskp} we have that, for $k\geq 1$, the function $x$ satisfies the dSKP recurrence.
Moreover, for all $(i,j)\in\Z^2$, the function $x$ satisfies the initial condition
\[
a_{i,j}:=x(i,j,h(i,j))=x(i,j,[i+j]_2)=t_{i,j},
\]
giving the face weights of the statement.

As a consequence of Theorem~\ref{theo:expl_sol}, we know that $x(i,j,k)=Y(A_{k-1}[t_{i,j}],t)$. Using that $T^k(t)_{i,j}=x(i,j,k+1)$ ends the proof.
\end{proof}

\subsection{Singularities}

Using Corollary~\ref{thm:Dodgson_prerequ}, we now study singularities of Miquel dynamics; this discussion is illustrated in Figure~\ref{fig:miquel_dodgson}. The idea is to perform Miquel dynamics starting from a singular circle pattern (referred to as a Dodgson circle pattern) and prove that, if the singular circle pattern is moreover closed, a similar singular circle pattern appears after a determined number of steps. This is the content of Theorem~\ref{thm:sing_Miquel} below.

\begin{figure}[tb]
  \centering
  \includegraphics[scale=1.3]{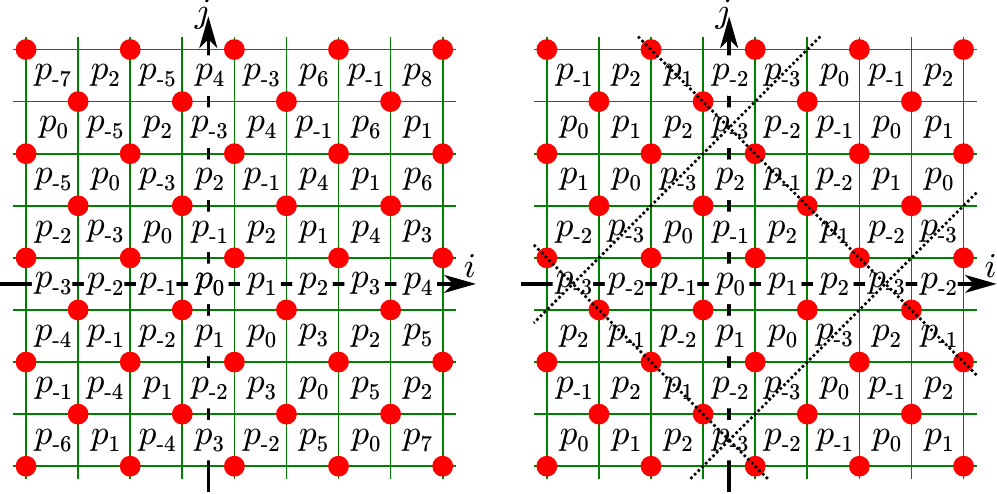}
  \caption{Left: a piece of an even Dodgson circle pattern. Every red dot
    corresponds to the same circle $\mathscr{D}$. Every face corresponds to a
    point $p^0_i$ on $\mathscr{D}$, abbreviated by $p_i$ to alleviate the
    picture. Right: a doubly $m$-closed case for $m=3$; an
    elementary pattern corresponds to a dotted square and repeats on
    the whole graph; the initial
    data consists in $2m=6$ points $p_{-3},p_{-2},p_{-1},p_0,p_1,p_2$
    on $\mathscr{D}$, and $m^2=9$ circles distinct from $\mathscr{D}$, each passing
    through two points $p_i,p_j$ with $i \not\equiv j \mod 2$. A
    geometric realisation is shown in Figure~\ref{fig:miqdodgson},
    where the green circle can be thought of as $\mathscr{D}$ and its six
    vertices as $p_{-3},\dots,p_2$ in cyclic order.
  }
  \label{fig:miquel_dodgson}
\end{figure}

\begin{definition}
Let $\mathscr{D} \subset \C$ be a fixed circle and let $p^0: \Z \rightarrow \mathscr{D}$. An \emph{even Dodgson circle pattern (with respect to $\mathscr{D}$)} is a circle pattern $p^\mathscr{D}: \Z^2 \rightarrow \C$ such that, for all $(i,j)\in\Z^2$,
\begin{align} \label{eq:dodgson_circle_patterns}
p^\mathscr{D}_{i,j} =
\begin{cases}
    p^0_{i-j} & \mbox{if } i+j\in 2\Z+1,\\
    p^0_{i+j} & \mbox{if } i+j\in 2\Z.
\end{cases}
\end{align}
An \emph{odd Dodgson circle pattern} is defined similarly exchanging the parity in~\eqref{eq:dodgson_circle_patterns}.
\end{definition}

Note that in a circle pattern, diagonally opposite circles intersect and, in the definition of an even Dodgson circle pattern $p^\mathscr{D}$, even circles share three points of $\mathscr{D}$. This implies that all even circles coincide with $\mathscr{D}$, and thus all the even circle centers coincide as well. Similarly, in an odd Dodgson circle pattern, all odd circles, resp. odd circle centers, coincide. Moreover, note that the intersection points $p^0$ of a Dodgson circle pattern do not completely determine the circle pattern. Thus in this case it is necessary to consider a circle pattern to be defined by the combination of intersection points $p$ and circles $c$.

Thus from the perspective of dSKP we are in Dodgson initial conditions from Definition~\ref{def:sing}, explaining the above terminology.

\begin{figure}[tb]
	\centering
	\frame{\includegraphics[scale=0.22]{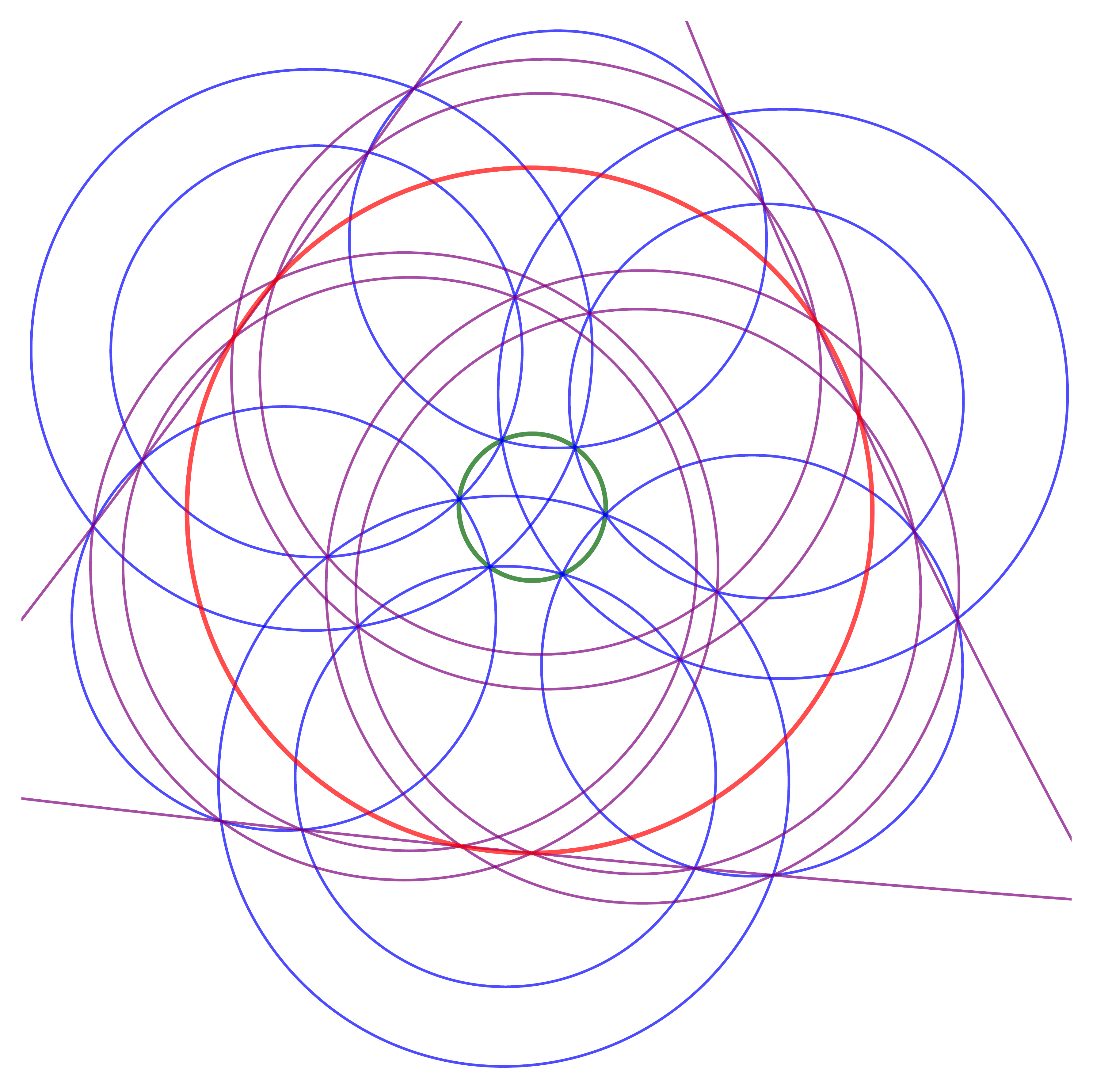}}
	\caption{Miquel Dodgson for $m=3$ yields an incidence theorem involving 20 circles. The four layers of circles, in order, are green, blue, purple, red. Green + blue is an even Dodgson circle pattern, and red + purple an odd Dodgson circle pattern, and the second pair is obtained from the first by two applications of Miquel dynamics.}
	\label{fig:miqdodgson}
\end{figure}

Let $m\geq 1$.
We call a circle pattern \emph{$m$-doubly closed} if, for all $(i,j)\in \Z^2$,
\[
c_{i+m,j+m} = c_{i,j} = c_{i+m,j-m}.
\]
\begin{theorem}\label{thm:sing_Miquel}
Let $m\geq 1$, and let $p^\mathscr{D}$ be a doubly $m$-closed even Dodgson circle pattern with respect to a circle $\mathscr{D}$. Assume we can apply Miquel dynamics on $p^\mathscr{D}$ at least $m-1$ times. Then there is a circle $\mathscr{D}'$ such that, when $m$ is even, resp. odd, $T^{m-1}(p^\mathscr{D})$ is an even, resp. odd, Dodgson circle pattern with respect to $\mathscr{D}'$.
\end{theorem}

\begin{proof}
By Theorem~\ref{theo:explmiquel}, we know that for $k\geq 1$, and all $(i,j)\in\Z^2$ such that $[i+j+k]_2=1$, $T^k(t)_{i,j}$ can be expressed as the ratio function of oriented dimers. The initial condition in the claim are $m$-Dodgson and therefore Corollary~\ref{thm:Dodgson_prerequ} applies, telling us that the singularity appears after $m-1$ steps. That is that $T^{m-1}(t)_{i,j}$ is independent of $(i,j)$ for all $(i,j)\in\Z^2$, such that $[i+j+m]_2=0$. This means that after $m-1$ iterations, when $m$ is even, resp. odd, all the even, resp. odd, circle centers coincide.
\end{proof}

Due to Thereom \ref{thm:sing_Miquel}, iterating Miquel dynamics on doubly $m$-closed initial data yields a finite configuration of circles and intersection points. In general, looking at all the iterations of Miquel dynamics, on each circle appearing there are $8$ intersection points that appear. By contrast, for a Dodgson circle pattern, on the circle $\mathscr{D}$ in the initial layer there are $2m$ intersection points, and also $2m$ points on the final circle $\mathscr{D}'$. Moreover, since some of the initial intersection points coincide, circles in the second (and the second to last) layer lose two intersection points with respect to the expected $8$. Additionally, no matter which layer, four circles pass through every intersection point. This leads to a particularly symmetric configuration. Let us describe it more precisely for small values of $m$.

In the case $m=2$ there are 3 layers of circles. The second and second to last layer coincide and so there are $8-2-2=4$ points on the circles of the middle layer. Therefore, there are $4$ intersection points on every circle, and four circles pass through every intersection point. There are $6$ circles and $8$ intersection points in total, and this is just the configuration in Miquel's theorem, see Figure~\ref{fig:miquelconfig}.

In the case $m=3$ there are 4 layers of circles. The layers are the first, the second, the second to last and the last layer. Therefore there are $6$ intersection points on all circles, and four circles pass through every intersection point. In total, there are $20$ circles and $30$ intersection points, which leads to another symmetric configuration, see Figure \ref{fig:miqdodgson}.

In the case $m>3$ there are $(m+1)$ layers, and now the number of intersection points in the layers varies, so the configurations are less symmetric.

Other singularities like column-wise and single-column coincidences of circles in the initial-data exist as well. The same arguments as in the Dodgson Miquel case apply.

\section{P-nets}\label{sec:pnets}

\subsection{Definitions}

We now consider P-nets, which were introduced in the study of discrete isothermic nets~\cite[Section 6.2]{bpdiscsurfaces} and relations of these surfaces to discrete integrable systems. P-nets are also related to discrete quadratic holomorphic differentials by Lam \cite{lamminimal}.
Although more abstract, P-nets have a geometric realisation in terms of discrete holomorphic functions, see Section~\ref{sec:discrete_holom_P_nets_0}, and occur in orthogonal circle patterns, see Section~\ref{sec:orthogonal_circle_patterns}.

\begin{definition}\label{def:pnet}
A \emph{P-net} is a map $p:\Z^2 \rightarrow \hat{\C}$ such that, for all $(i,j)\in\Z^2$,
\begin{align}
\frac{1}{p_{i+1,j}-p_{i,j}} - \frac{1}{p_{i,j+1}-p_{i,j}} + \frac{1}{p_{i-1,j}-p_{i,j}} - \frac{1}{p_{i,j-1}-p_{i,j}} = 0\label{eq:pnetsum}.
\end{align}
\end{definition}
In fact, this definition is equivalent to requiring that for all $(i,j)\in \Z^2$, in any affine chart of $\CP^1$ such that $p_{i,j}$ is at infinity, the quad $(p_{i+1,j}, p_{i,j+1}, p_{i-1,j}, p_{i,j-1})$ is a parallelogram. This property justifies the \emph{P} in P-net.  Note that the defining equation of P-nets also occurs in the context of Wynn's identity, see \cite[Equation (15)]{wynn}.

P-nets are exactly those maps from $\Z^2$ to $\hat{\C}$ that satisfy the following identity, see also \cite[Section 5.2]{ksclifford}, or in the context of Wynn's identity also \cite{dswynn}. As we will see in the next section, this relation underlies the occurrence of the dSKP recurrence in P-nets.

\begin{lemma}\label{lem:pnetdskp}
A map $p:\Z^2 \rightarrow \hC$ is a P-net if and only if, for all $(i,j)\in\Z^2$,
\begin{align}\label{eq:pnetmr}
\frac{(p_{i,j} - p_{i+1,j})(p_{i,j+1} - p_{i,j})(p_{i-1,j} - p_{i,j-1})}{(p_{i+1,j} - p_{i,j+1})(p_{i,j} - p_{i-1,j})(p_{i,j-1} - p_{i,j})} = -1.
\end{align}
\end{lemma}

\subsection{Explicit solution}

Let $p_j:=(p_{i,j})_{i\in\Z}$ denote the points of the $j$-th row. The recurrence~\eqref{eq:pnetmr} of Lemma \ref{lem:pnetdskp} implies that the points $p_{j+1}$ are determined by the points $p_j$ and $p_{j-1}$. Therefore we view two rows of data as initial data, and this determines the whole P-net.

Note that this framework is different from Miquel dynamics, where we viewed the \emph{whole} circle pattern as initial data.

The next theorem makes this more explicit and proves that, for all
$i\in\Z,j\geq 1$, the point $p_{i,j}$ is equal to the ratio function of
oriented dimers of an Aztec diamond subgraph of $\Z^2$ with face
weights a subset of $(p_{i,0})_{i\in\Z},(p_{i,1})_{i\in\Z}$, see Figure~\ref{fig:pnet_ic}.


\begin{theorem}\label{theo:explpnet}
Let $p:\Z^2 \rightarrow \hC$ be a P-net, and consider the graph $\Z^2$ with face-weights $(a_{i,j})_{(i,j)\in\Z^2}$ given by
\begin{equation*}
a_{i,j} = p_{i,[i+j]_2}.
\end{equation*}
Then, for all $i\in\Z, j\geq 1$, we have
\begin{align*}
        p_{i,j}=  Y(\az{j-1}{p_{i,[j]_2}},a).
\end{align*}
\end{theorem}
\begin{proof}

Consider the function $x:\calL\rightarrow\hat{\C}$ given by
\[
x(i,j,k)=p_{i,k}.
\]
As a consequence of Lemma~\ref{lem:pnetdskp}, we have that, for $k\geq1$,
the function $x$ satisfies the dSKP recurrence.
Note that, for all $(i,j)\in\Z^2$, the function $x$ satisfies the initial condition
\[
a_{i,j}=x(i,j,[i+j]_2)=p_{i,[i+j]_2},
\]
giving the face weights of the statement.
As a consequence of Theorem~\ref{theo:expl_sol}, we know that $x(i,j,k)=Y(A_{k-1}[a_{i,j}],a)$. This implies that
\[x(i,i+j,j) = p_{i,j}, \text{ and }\ a_{i,i+j}=p_{i,[2i+j]_2}=p_{i,[j]_2}. \]
The proof is concluded by using that the face weighted graph $\Z^2$ is invariant by vertical translations of length 2.\qedhere
\end{proof}

\subsection{Singularities}

Assume we only know the restrictions $p_0,p_1:\Z \rightarrow \hC$ of
a P-net $p$ to two consecutive rows $p_0,p_1$. Then, recall that by iterating Lemma \ref{lem:pnetdskp},
all of $p$ is uniquely reconstructable from $p_0,p_1$.

Let us denote the propagation of data in a P-net as the map
\begin{align*}
        T: (\hC)^{\Z_2 \times \Z } \rightarrow (\hC)^{\Z_2 \times \Z },\quad  (p_j,p_{j+1}) \mapsto (p_{j+1},p_{j+2}).
\end{align*}

Let $m\geq 1$. A P-net $p$ is said to be \emph{$m$-closed} if $p_{i,j} = p_{i+m,j}$ for all $i,j\in \Z$.
We now study the occurrence of singularities when we start from a singular $m$-closed P-net, \emph{i.e.}, an $m$-closed P-net such that $p_0\equiv 0$, see also Figure~\ref{fig:pnetsingularity}. The recurrence of the singularity has already been proven by Glick \cite[Theorem 6.15]{gdevron}, and the precise position by Yao \cite[Theorem 4.1]{yao}. Note that Glick and Yao call the propagation of P-nets the \emph{lower pentagram map}, as it bears some resemblance to the pentagram map algebraically. We are able to obtain the result as an immediate corollary of our general singularity theorems.

\begin{theorem}\label{theo:pnetsingularity} Let $m\geq 1$, and
let $p$ be an $m$-closed P-net such that $p_0\equiv 0$. Assume we can apply the propagation map $T$ to $(p_0,p_1)$ at least $m-1$ times.
Then, for all $i\in \Z$, we have
\begin{align}
p_{i,m} = \Bigl(\frac{1}{m} \sum_{\ell=0}^{m-1} p_{\ell,1}^{-1}\Bigr)^{-1},
\end{align}
that is the singularity repeats after $m-1$ steps and its value is the harmonic mean of $p_1$.
\end{theorem}
\begin{proof}
By Theorem \ref{theo:explpnet} we know that $p_j$ for $j\geq 1$ can be
expressed via a solution of the dSKP recurrence.
Since the P-net $p$ is $m$-closed, the corresponding initial conditions for dSKP are $m$-doubly periodic, moreover they satisfy $a_{i,j}=a_{i,j+2}$, and the fact that $p_0 \equiv 0$ implies that they are $m$-Dodgson; see Figure~\ref{fig:pnet_ic} and Definition~\ref{def:sing}. As a result,
they satisfy the hypothesis of
Corollary~\ref{cor:harm_mean}. Therefore, the values of the dSKP
solution at height $m$ (which are also the values of $p_m$) are
all equal to the harmonic mean of $p_1$.
\end{proof}

In the case of $m$-closed singular initial data with even $m$, it turns out that one may add an additional constraint, which forces the singularity to appear a step earlier than in Theorem~\ref{theo:pnetsingularity}. The constraint is essentially that the harmonic mean of the even parity points of $p_1$ equals the harmonic mean of the odd parity points of $p_1$. We formalize this observation in the next theorem.

\begin{theorem}\label{theo:pnetpremature} Let $m\in2\N+2$, and
        let $p$ be an $m$-closed P-net such that $p_0\equiv 0$, and such that
        \begin{align}
                \sum_{\ell=0}^{m-1} (-1)^\ell p_{\ell,1}^{-1} = 0.
        \end{align}
        Assume we can apply the propagation map $T$ to $(p_0,p_1)$ at least $m-2$ times.
        Then, for all $i\in\Z$, we have
        \begin{align}
                p_{i,m-1} = \Bigl(\frac{1}{m} \sum_{\ell = 0}^{m-1} p_{\ell,1}^{-1}\Bigr)^{-1},
        \end{align}
        that is the singularity repeats after $m-2$ steps and its value is the harmonic mean of $p_1$.
\end{theorem}
\proof{
  By Theorem \ref{theo:explpnet} we know that $p_j$ for $j>1$ can be
  expressed via a solution of the dSKP recurrence.
  The initial conditions in the claim satisfy the hypothesis of
  Proposition~\ref{prop:Nmat} with $d=0$, as all initial data at height $0$ are
  equal to $0$ (see also Figure~\ref{fig:pnet_ic}). Therefore, the
  value of $p_{i,m-1}$ is given by
  \begin{equation*}
    p_{i,m-1}=\sum_{0\leq i',j' \leq m-2} N^{-1}_{i',j'}
  \end{equation*}
  where $N$ is the matrix of size $m-1$:
  \begin{equation*}
    N = \begin{pmatrix}
      p^{-1}_{i,1} & p^{-1}_{i+1,1} & p^{-1}_{i+2,1} & \dots & p^{-1}_{i+m-2,1} \\
      p^{-1}_{i-1,1} & p^{-1}_{i,1} & p^{-1}_{i+1,1} & \dots &
      p^{-1}_{i+m-3,1} \\
       \vdots & & & & \vdots  \\
      p^{-1}_{i-m+2,1} & \dots &  &  & p^{-1}_{i,1}
    \end{pmatrix},
  \end{equation*}
  with indices taken modulo $m$. Since we suppose
  \begin{equation*}
    p^{-1}_0 + p^{-1}_2 + \dots + p^{-1}_{m-2} = p^{-1}_1 + p^{-1}_3 +
    \dots + p^{-1}_{m-1} =: \lambda,
  \end{equation*}
  we see that $N$ applied to the vector $(1,0,1,0,\dots,1)^T$
  gives the constant vector $\lambda (1,1,\dots,1)^T$. Therefore,
  \begin{equation*}
    (1,0,1,0,\dots,1)^T = \lambda N^{-1} (1,1,\dots,1)^T,
  \end{equation*}
  and then
\begin{align*}
      p_{i,m-1} &= \sum_{0\leq i',j' \leq m-2} N^{-1}_{i',j'}
       = (1, 1, \dots, 1) N^{-1} (1,1,\dots,1)^T \\
       &= \lambda^{-1} (1, 1, \dots, 1) (1,0,1,0,\dots,1)^T
       = \frac{m}{2} \lambda^{-1} = \Bigl(\frac{2}{m} \lambda \Bigr)^{-1} \\
      & = \Bigl(\frac{1}{m} \sum_{\ell = 0}^{m-1} p_{\ell,1}^{-1}\Bigr)^{-1}. \hspace{8cm}\qedhere
\end{align*}
}

\section{Integrable cross-ratio maps and Bäcklund pairs}\label{sec:Backlund_pairs}

Integrable cross-ratio maps were introduced in relation to the discrete KdV equation \cite[Section 2]{nc95} and discrete isothermic surfaces \cite[Section 4]{bpdisosurfaces}. They are solutions to one of the discrete integrable equations on quad-graphs \cite[(Q1) with $\delta = 0$]{absquads}. They are also of interest because they contain many other examples as special cases, in particular the discrete holomorphic functions (see Section \ref{sec:dhol}), orthogonal circle packings (see Section \ref{sec:orthogonal_circle_patterns}), polygon recutting (see Section \ref{sec:recut}) and circle intersection dynamics (see Section \ref{sec:cid}).

\subsection{Definitions and properties}\label{subsec:Backlund_pairs_defi}

Let us briefly discuss the concept of an \emph{edge-labeling}, see also \cite{bsintquad}. An edge-labeling is a function $\xi$ from the edges of $\Z^N$ to $\C \setminus \{0\}$, such that in each quad of $\Z^N$ the values of $\xi$ on opposite edges agree. As a consequence, an edge-labeling $\xi$ is just a sequence of functions $\chi^1, \chi^2,\dots,\chi^N$ from $\Z$ to $\C \setminus \{0\}$, where $\chi^i$ encodes all the values of $\xi$ corresponding to edges parallel to the $i$-th coordinate direction. We call the functions $\chi^1, \chi^2,\dots,\chi^N$ the \emph{edge-labels}. In the following, we work mostly on $\Z^2$, and we will use the letters $\alpha$ for $\chi^1$ and $\beta$ for $\chi^2$. 
In addition, in Definition~\ref{def:backlundpair}, $\gamma$ may be thought of as $\chi^3(0)$.

\begin{definition}[\cite{nc95,bpdisosurfaces}] \label{def:intcrmap}
        Let $\alpha,\beta: \Z \rightarrow \C\setminus\{0\}$ be \emph{edge-labels}. An \emph{integrable cross-ratio map} is a map $z:\Z^2 \rightarrow \hC$ such that, for all $(i,j)\in\Z^2$,
        \begin{align}
          \label{eq:intcrmapdef}
                \cro(z_{i,j},z_{i+1,j},z_{i+1,j+1},z_{i,j+1}) :=\frac{(z_{i,j}-z_{i+1,j})(z_{i+1,j+1}-z_{i,j+1})}{(z_{i+1,j}-z_{i+1,j+1})(z_{i,j+1}-z_{i,j})}  = \frac{\alpha_i}{\beta_j}.
        \end{align}
\end{definition}

\begin{definition}\label{def:backlundpair}
        Let $\alpha,\beta: \Z \rightarrow \C\setminus\{0\}$ and $\gamma \in \C\setminus\{0\}$. A \emph{Bäcklund pair of integrable cross-ratio maps $z,w$} is a pair of integrable cross-ratio maps such that, for all $(i,j)\in\Z^2$,
        \begin{align}
            \cro(z_{i,j},z_{i+1,j},w_{i+1,j},w_{i,j}) &= \frac{\alpha_i}{\gamma}, \label{eq:Bäck_1}\\
            \cro(z_{i,j},z_{i,j+1},w_{i,j+1},w_{i,j}) &= \frac{\beta_j}{\gamma}.\label{eq:Bäck_2}
        \end{align}
\end{definition}

It is not trivial that Bäcklund pairs exist, but 
in fact, for any choice of $\gamma$ and integrable cross-ratio map $z$ there is a one complex parameter family of integrable cross-ratio maps $w$ such that $z,w$ is a Bäcklund pair, \cite{bmsanalytic}. To give the reader an improved understanding of 
this statement, and because we will need the construction to study singularities, let us explain how to obtain $w$ from $z$. For some $(i,j)\in \Z^2$ choose $w_{i,j}\in \C$ such that $w_{i,j}$ is not equal to $z_{i,j}$. Then $w_{i+1,j}$ is determined by Definition \ref{def:backlundpair}, more concretely the formula is
\begin{align}
	w_{i+1,j} = \frac{w_{i,j}(\gamma z_{i,j} + (\alpha_i-\gamma) z_{i+1,j} ) - \alpha_i z_{i,j}z_{i+1,j} }{w_{i,j} \alpha_i  + \gamma(z_{i,j} - z_{i+1,j}) - \alpha_i z_{i,j} }.
\end{align}
Therefore, $w_{i+1,j}$ as a function of $w_{i,j}$ is a Möbius transformation $M^1_{i,j}(w_{i,j})$ of $w_{i,j}$, with coefficients depending on $z$. Analogously, it is possible to express $w_{i,j+1}$ as a Möbius transformation $M^2_{i,j}(w_{i,j})$ of $w_{i,j}$. Note that Möbius transformations form a group: the \emph{Möbius group} $\mathrm{PGL}(2, \C)$. Therefore the composition of two Möbius transformations is also a Möbius transformation. By composing the Möbius transformations of type $M^1_{k,\ell}$ and $M^2_{k,\ell}$, it is possible to obtain all of $w$ as Möbius transformations of an initial $w_{i,j}$. However, it is not immediately clear that this construction is well-defined. For example, we observe that
\begin{align}
	w_{i+1,j+1} = M^2_{i+1,j} \circ M^1_{i,j+1}(w_{i,j}),
\end{align}
but also
\begin{align}
 	w_{i+1,j+1} = M^1_{i,j+1} \circ  M^2_{i+1,j}(w_{i,j}).
\end{align}
A priori, it is not clear why these two definitions of $w_{i+1,j+1}$ should coincide. However, it is indeed the case that $M^1_{i,j+1} \circ  M^2_{i+1,j} = M^2_{i+1,j} \circ M^1_{i,j+1}$, which is due to the \emph{multi-dimensional consistency} of the integrable cross-ratio maps, see \cite{absquads}. As a result, the Bäcklund pair $z,w$ is indeed uniquely defined from $z$, $w_{i,j}$ for some $i,j\in \Z$ and the edge-labels.

The following lemma is a natural observation for integrable cross-ratio maps, and we are certainly not the first to notice. However, an important consequence is that integrable cross-ratio maps are a reduction of dSKP lattices. This fact, to the best of our knowledge, is new. It is made explicit and used in the next section for proving an explicit formula for the solution.

\begin{lemma}\label{lem:intcrdskp}
        Let $z,w$ be a Bäcklund pair of integrable cross-ratio maps. Then, the following equations hold for all  $(i,j)\in \Z^2$,
        \begin{align}
            \frac{(z_{i,j}-z_{i+1,j})(w_{i+1,j}-w_{i+1,j+1})(w_{i,j+1}-z_{i,j+1})}{(z_{i+1,j}-w_{i+1,j})(w_{i+1,j+1}-w_{i,j+1})(z_{i,j+1}-z_{i,j})} &= -1,\label{eq:intcrdskpa}\\
            \frac{(w_{i,j}-w_{i+1,j})(z_{i+1,j}-z_{i+1,j+1})(z_{i,j+1}-w_{i,j+1})}{(w_{i+1,j}-z_{i+1,j})(z_{i+1,j+1}-z_{i,j+1})(w_{i,j+1}-w_{i,j})} &= -1.\label{eq:intcrdskpb}
        \end{align}
\end{lemma}

\proof{
    The left-hand side of the first equation can be decomposed into the product
    \begin{align*}
            -\cro(z_{i,j}, z_{i+1,j}, w_{i+1,j}, w_{i,j}) \cro(w_{i,j}, w_{i,j+1}, w_{i+1,j+1}, w_{i+1,j}) \cro(w_{i,j+1}, z_{i,j+1}, z_{i,j}, w_{i,j}),
    \end{align*}
    of three cross-ratios. But by definition of integrable cross-ratio maps and Bäcklund pairs the cross-ratios are $\alpha_i \gamma^{-1}$, $\beta_j \alpha_i^{-1}$ and $\gamma \beta_j^{-1}$, thus the first equation is proven. The proof of the second equation uses the same arguments.\qed
}

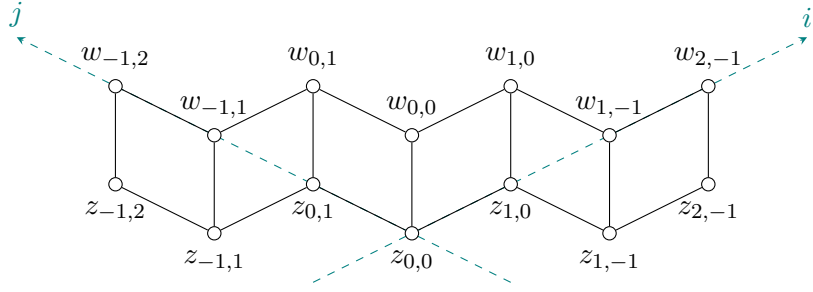
\begin{figure}[tb]
        \centering
        \begin{tikzpicture}[scale=1.3,baseline={([yshift=0ex]current bounding box.center)}]
                \draw[->, >=stealth, dashed, color=teal] (-2,0) -- (3,2.5);
                \draw[->, >=stealth, dashed, color=teal] (0,0) -- (-5,2.5);
                \draw (3,2.5) node [above, color=teal] {$i$};
                \draw (-5,2.5) node [above, color=teal] {$j$};

                \node[wvert,label=below:$z_{1,-1}$] (z1m1) at (1,0.5) {};
                \node[wvert,label=below:$z_{0,0}$] (z00) at (-1,0.5) {};
                \node[wvert,label=below:$z_{1,0}$] (z10) at (0,1) {};
                \node[wvert,label=below:$z_{0,1}$] (z01) at (-2,1) {};
                \node[wvert,label=below:$z_{2,-1}$] (z2m1) at (2,1) {};
                \node[wvert,label=below:$z_{-1,1}$] (zm11) at (-3,0.5) {};
                \node[wvert,label=below:$z_{-1,2}$] (zm12) at (-4,1) {};
                \node[wvert,label=above:$w_{1,-1}$] (w1m1) at (1,1.5) {};
                \node[wvert,label=above:$w_{0,0}$] (w00) at (-1,1.5) {};
                \node[wvert,label=above:$w_{1,0}$] (w10) at (0,2) {};
                \node[wvert,label=above:$w_{0,1}$] (w01) at (-2,2) {};
                \node[wvert,label=above:$w_{2,-1}$] (w2m1) at (2,2) {};
                \node[wvert,label=above:$w_{-1,1}$] (wm11) at (-3,1.5) {};
                \node[wvert,label=above:$w_{-1,2}$] (wm12) at (-4,2) {};
                \draw[-]
                        (zm12) -- (zm11) -- (z01) -- (z00) -- (z10) -- (z1m1) -- (z2m1)
                        (wm12) -- (wm11) -- (w01) -- (w00) -- (w10) -- (w1m1) -- (w2m1)
                        (zm12) -- (wm12) (zm11) -- (wm11) (z01) -- (w01) (z00) -- (w00) (z10) -- (w10) (z1m1) -- (w1m1) (z2m1) -- (w2m1)
                ;
        \end{tikzpicture}
        \caption{
          Labeling of initial data for the
          propagation in a Bäcklund pair of integrable cross-ratio
          maps.
        }
        \label{fig:intcraztec}
\end{figure}

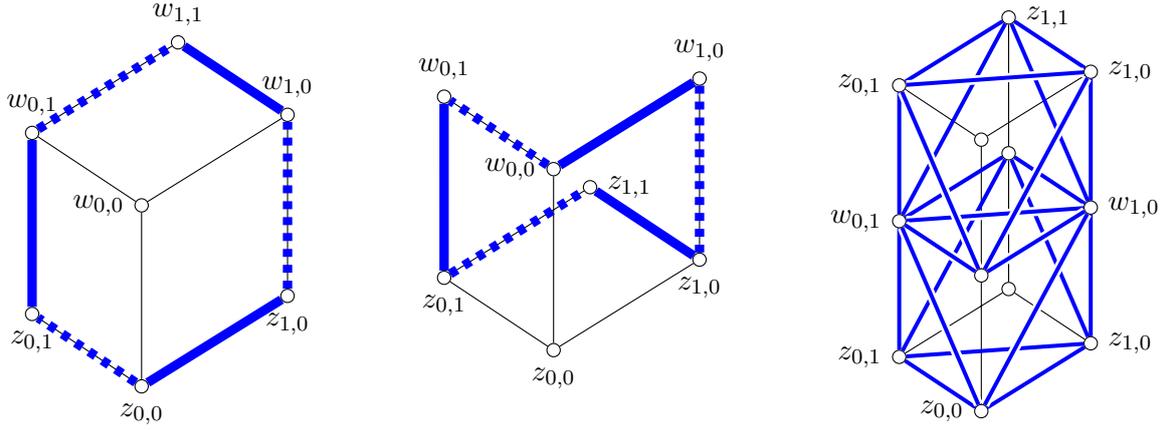
\begin{figure}[tb]
	\centering
	\begin{tikzpicture}[scale=1.2,baseline={([yshift=0ex]current bounding box.center)}]
		\node[wvert,label=below:$z_{0,0}$] (z00) at (0,0) {};
		\node[wvert,label=below:$z_{1,0}$] (z10) at (1.6,1) {};
		\node[wvert,label=below:$z_{0,1}$] (z01) at (-1.2,0.8) {};
		\node[wvert,label=left:$w_{0,0}$] (w00) at (0,2) {};
		\node[wvert,label=above:$w_{1,0}$] (w10) at ($(w00)+(z10)$) {};
		\node[wvert,label=above:$w_{0,1}$] (w01) at ($(w00)+(z01)$) {};
		\node[wvert,label=above:$w_{1,1}$] (w11) at ($(z01)+(z10)+(w00)$) {};
		\draw[-]
		      (z01) -- (z00) 
		       (w01) -- (w00) -- (w10) -- (w11) -- (w01)
		      (z01) -- (w01) (z00) -- (w00) (z10) -- (w10) 
		;
		\draw[blue, solid, line width=3.5pt]
		             	(z00) -- (z10) (w10) -- (w11) (w01) -- (z01)                
		;
		\draw[blue, dashed, line width=3.5pt]
		             	(z10) -- (w10) (w01) -- (w11) (z00) -- (z01)                
		;
	\end{tikzpicture}\hspace{1cm}
	\begin{tikzpicture}[scale=1.2,baseline={([yshift=0ex]current bounding box.center)}]
		\node[wvert,label=below:$z_{0,0}$] (z00) at (0,0) {};
		\node[wvert,label=below:$z_{1,0}$] (z10) at (1.6,1) {};
		\node[wvert,label=below:$z_{0,1}$] (z01) at (-1.2,0.8) {};
		\node[wvert,label=left:$w_{0,0}$] (w00) at (0,2) {};
		\node[wvert,label=above:$w_{1,0}$] (w10) at ($(w00)+(z10)$) {};
		\node[wvert,label=above:$w_{0,1}$] (w01) at ($(w00)+(z01)$) {};
		\node[wvert,label=right:$z_{1,1}$] (z11) at ($(z01)+(z10)$) {};
		\draw[-]
		      (z01) -- (z00) -- (z10) -- (z11) -- (z01)
		       (w01) -- (w00)
		      (z01) -- (w01) (z00) -- (w00) (z10) -- (w10) 
		;
		\draw[blue, solid, line width=3.5pt]
		             	(w00) -- (w10) (z10) -- (z11) (z01) -- (w01)                
		;
		\draw[blue, dashed, line width=3.5pt]
		             	(w10) -- (z10) (z01) -- (z11) (w00) -- (w01)                
		;
	\end{tikzpicture}\hspace{1cm}
	\begin{tikzpicture}[scale=0.9,baseline={([yshift=0ex]current bounding box.center)}]
		\node[wvert,label=left:$z_{0,0}$] (z00) at (0,0) {};
		\node[wvert,label=right:$z_{1,0}$] (z10) at (1.6,1) {};
		\node[wvert,label=left:$z_{0,1}$] (z01) at (-1.2,0.8) {};
		\node[wvert] (w00) at (0,2) {};
		\node[wvert,label=right:$w_{1,0}$] (w10) at ($(w00)+(z10)$) {};
		\node[wvert,label=left:$w_{0,1}$] (w01) at ($(w00)+(z01)$) {};
		\node[wvert] (w11) at ($(z01)+(z10)+(w00)$) {};
		\node[wvert] (z11) at ($(z01)+(z10)$) {};
		\node[wvert] (zz00) at ($(w00)+(w00)$) {};
		\node[wvert,label=right:$z_{1,0}$] (zz10) at ($(z10)+2*(w00)$) {};
		\node[wvert,label=left:$z_{0,1}$] (zz01) at ($(z01)+2*(w00)$) {};
		\node[wvert,label=right:$z_{1,1}$] (zz11) at ($(z11)+2*(w00)$) {};

		\draw[-]
		      (z01) -- (z00) -- (z10) -- (z11) -- (z01)
		       (w01) -- (w00) -- (w10) -- (w11) -- (w01)
		      (z01) -- (w01) -- (zz01) (z00) -- (w00) -- (zz00)
		      (z10) -- (w10) -- (zz10) (z11) -- (w11) -- (zz11)
		      (zz01) -- (zz00) -- (zz10) -- (zz11) -- (zz01)
		;
		\draw[blue, solid, line width=1.5pt]
            	(z00) edge (z01) edge (z10) edge (w01) edge (w10)
            	(w00) edge (w01) edge (w10) edge (zz01) edge (zz10)
            	(w11) edge (z01) edge (z10) edge (w01) edge (w10)
            	(zz11) edge (w01) edge (w10) edge (zz01) edge (zz10)
            	(z01) -- (z10) -- (w10) -- (w01) -- (z01)
            	(w01) -- (w10) -- (zz10) -- (zz01) -- (w01)
		;
		\draw[white, solid, line width=3.5pt]
 			(w01) -- (w10)
		;		
		\draw[blue, solid, line width=1.5pt]
 			(w01) -- (w10)
		;		
		\draw[white, solid, line width=3.5pt]
 			(zz01) -- (zz10) -- (w00) -- (zz01)
 			(w10) -- (w00) -- (w01) -- (z00) -- (w10)
		;		
		\draw[blue, solid, line width=1.5pt]
 			(zz01) -- (zz10) -- (w00) -- (zz01)
 			(w10) -- (w00) -- (w01) -- (z00) -- (w10)
		;		
	\end{tikzpicture}		
	\caption{
		The initial data consists of $z_{0,0},z_{1,0},z_{0,1},w_{0,0},w_{1,0},w_{0,1}$. On the left we use Equation~\eqref{eq:intcrdskpa} to determine $w_{1,1}$, in the center we use Equation~\eqref{eq:intcrdskpb} to determine $z_{1,1}$. On the right we illustrate how to see the octahedra (blue) in $\Z^3 / (0,0,2)$.
	}
	\label{fig:mrinint}
\end{figure}

Before we give a formal explanation of the explicit solution, let us outline how we propagate initial data in a Bäcklund pair and how we use Lemma~\ref{lem:intcrdskp} to relate this propagation to the dSKP equation. Assume we know the initial data $(z_{i,j})_{i+j\in \{0,1\}}$ and $(w_{i,j})_{i+j\in \{0,1\}}$ of a Bäcklund pair (see Figure~\ref{fig:intcraztec}), as well as the edge-labels $\alpha_i,\beta_i$ for all $i\in \Z$ and $\gamma$. Then the $z$-initial data together with the edge-labels and Equation~\ref{eq:intcrmapdef} determine $(z_{i,j})_{i+j = 2}$. Analogously, the $w$-initial data together with the edge-labels and Equation~\ref{eq:intcrmapdef} determine $(w_{i,j})_{i+j = 2}$. By iterating this procedure forwards and backwards, we obtain the whole Bäcklund pair on $\Z^2$. Note that the propagation \emph{decouples}: the $z$-data only depends on the edge-labels and the $z$-initial data, while the $w$-data only depends on the edge-labels and the $w$-initial data.

In contrast, when using dSKP the propagation is \emph{coupled}, as we explain in the following. Given the same initial data, the $z$- and $w$-initial data and Equation~\ref{eq:intcrdskpa} determine $(w_{i,j})_{i+j = 2}$, see Figure~\ref{fig:mrinint} (left). Analogously, the $z$- and $w$-initial data and Equation~\ref{eq:intcrdskpb} determine $(z_{i,j})_{i+j = 2}$, see Figure~\ref{fig:mrinint} (center). Note that propagation via dSKP does not need the edge-label data, this data is implicitly contained in the initial data of the Bäcklund pair.

Finally, let us attempt to visualize how to see the octahedral lattice $\calL$ together with the Bäcklund pair. Figure~\ref{fig:mrinint} (right) is a visualization of the following explanation. Consider the following identification of $\Z^3$ with the initial data of the Bäcklund pair:
\begin{align}
	(i,j,k) \rightarrow \begin{cases}
		z_{i,j} & \mbox{if } [k]_2 = 0, \\
		w_{i,j} & \mbox{if } [k]_2 = 1.
	\end{cases}
\end{align}
Note that this identification is well defined on $\Z^3 / (0,0,2)$. We obtain the edges of the octahedral lattice in $\Z^3$ by all integer shifts of the edges
\begin{align}
	((0,0,0), (1,0,0)), \quad ((0,0,0), (0,1,0)),\quad ((0,0,0), (0,0,1))&, \\((0,0,0), (1,0,1)),\quad ((0,0,0), (0,1,1)),\quad ((1,0,0), (0,1,0))&.
\end{align}
With these edges, the vertex set of each octahedron is an integer shift of the vertices
\begin{align}
	(0,0,0), \quad (1,0,0), \quad (0,1,0), \quad (1,0,1), \quad (0,1,1), \quad (1,1,1).
\end{align}
Therefore each octahedron in $\calL$ corresponds to a cube in $\Z^3$. The propagation of the initial data of a Bäcklund pair in $\Z^3$ using dSKP via Lemma~\ref{lem:intcrdskp}, corresponds to the usual propagation of dSKP on $\calL$ as explained in the introduction.

\subsection{Explicit solution}

In the same spirit as for P-nets, but more involved in this case, the following theorem expresses the points $z_{i,j},w_{i,j}$ for all $(i,j)\in\Z^2$ such that $i+j\geq 1$, as the ratio function of oriented dimers of an Aztec diamond subgraph of $\Z^2$, with face weights being a subset of
$(z_{i,j}),(w_{i,j})_{i+j\in\{0,1\}}$, see Figure~\ref{fig:backlund_ic} for an example.

\begin{figure}[tb]
  \centering
  \includegraphics[width=13cm]{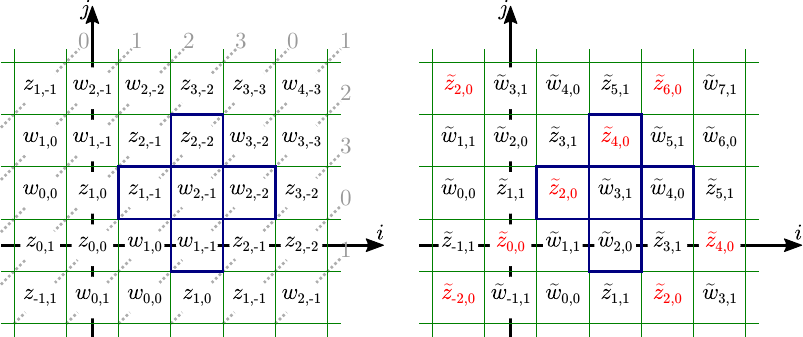}
  \caption{Left: face weights for the explicit solution of Bäcklund pairs in Theorem~\ref{theo:explintcr}. The
  dotted diagonals correspond to constant values of $[i-j]_4$, with
  that value indicated at the end. The sample blue Aztec diamond corresponds to the computation of $z_{3,0}$. Right: the same
  values after the change of variable
  \eqref{eq:rotated_weights}. Variables in red are those that are
  set to $0$ in Theorem \ref{theo:intcrsingular}, giving $(m,2)$-Devron initial conditions.}
  \label{fig:backlund_ic}
\end{figure}

\begin{theorem}\label{theo:explintcr}
        Let $z,w:\Z^2 \rightarrow \hC$ be a Bäcklund pair of integrable cross-ratio maps, and consider the graph $\Z^2$ with face weights $(a_{i,j})_{(i,j)\in\Z^2}$ given by,
        \[
        a_{i,j}=
        \begin{cases}
        z_{\frac{i+j+[i+j]_2}{2},\frac{-(i+j)+[i+j]_2}{2}} & \text{ if }[i-j]_4\in\{0,3\},\\
        w_{\frac{i+j+[i+j]_2}{2},\frac{-(i+j)+[i+j]_2}{2}} & \text{ if }[i-j]_4\in\{1,2\}.
        \end{cases}
        \]
        Then, for all $(i,j)\in\Z^2$ such that $i+j\geq 1$, introducing the notation
        \[
        (i',j'):=\Bigl(\frac{i-j+[i-j]_2}{2},\frac{-(i-j)+[i-j]_2}{2}\Bigr),
        \]
        we have
        \begin{align*}
        z_{i,j} &=
        \begin{cases}
        Y(A_{i+j-1}[z_{i',j'}],a)&
        \text{  if $[i+j]_4 \in \{0,1\}$},\\
        Y(A_{i+j-1}[w_{i',j'}],a)&
        \text{  if $[i+j]_4 \in \{2,3\}$},
        \end{cases} \\
        w_{i,j}&=
        \begin{cases}
        Y(A_{i+j-1}[z_{i',j'}],a)&
        \text{  if $[i+j]_4 \in \{2,3\}$},\\
        Y(A_{i+j-1}[w_{i',j'}],a)&
        \text{  if $[i+j]_4 \in \{0,1\}$.}
        \end{cases}
        \end{align*}
\end{theorem}

\begin{proof}
Consider the function $x:\calL \rightarrow \hat{\C}$ given by
\begin{align}\label{eq:proof_Backlund}
x(i,j,k) =
\begin{cases}
z_{\frac{i+j+k}{2},\frac{-(i+j)+k}{2}} & \mbox{ for } [i-j+k]_4 = 0, \\
w_{\frac{i+j+k}{2},\frac{-(i+j)+k}{2}} & \mbox{ for } [i-j+k]_4 = 2.
\end{cases}
\end{align}
As a consequence of Lemma~\ref{lem:intcrdskp}, we have that, for all $(i,j)\in\Z^2$, and all $k\geq 1$, the function $x$ satisfies the dSKP recurrence. Consider $(i,j)\in\Z^2$, such that $[i-j]_4\in\{0,3\}$ (the argument when $[i-j]_4\in\{1,2\}$ is similar), then $[i-j+[i+j]_2]_4=[i-j+[i-j]_2]_4=0$, and the function $x$ satisfies the initial condition
\[
a_{i,j}=x(i,j,[i+j]_2)=z_{\frac{i+j+[i+j]_2}{2},\frac{-(i+j)+[i+j]_2}{2}},
\]
giving the face weights of the statement.
As a consequence of Theorem~\ref{theo:expl_sol}, we know that $x(i,j,k)=Y(A_{k-1}[a_{i,j}],a)$. Using Equation~\eqref{eq:proof_Backlund}, we have that for every $m\in\Z$,
\begin{equation}\label{eq:proof_Backlund_1}
x(m,i-j-m,i+j)=
\begin{cases}
z_{i,j}&\text{ if $m+j\in 2\Z$},\\
w_{i,j}&\text{ if $m+j\in 2\Z+1$}.
\end{cases}
\end{equation}
Applying this at $m=j$, we get that
$z_{i,j}=x(j,i-2j,i+j)=Y(A_{i+j-1}[a_{j,i-2j}])$. An explicit
computation of the face weight $a_{j,i-2j}$ gives
\begin{itemize}
\item if $[-i+3j]_4 \in \{0,3\}$ (which is equivalent to $[-i-j]_4
  \in \{0,3\}$, or $[i+j]_4 \in \{0,1\}$), then
  $a_{j,i-2j}
  = z_{i',j'}$.
\item if $[-i+3j]_4 \in \{1,2\}$ (which is equivalent to $[-i-j]_4
  \in \{1,2\}$, or $[i+j]_4 \in \{2,3\}$), then
  $a_{j,i-2j}
  = w_{i',j'}$.
\end{itemize}
This shows that $z_{i,j}$ is the ratio of partition functions of
the Aztec diamond of size $i+j-1$ centered at $(j,i-2j)$, whose
central face has weight $z_{i',j'}$. However, note that the whole
solution~\eqref{eq:proof_Backlund} is invariant under translations of
multiples of $(2,-2,0)$. Therefore, \emph{any} Aztec diamond of size
$i+j-1$ whose central face has weight $z_{i',j'}$ is a translate of
the first one by such a translation, and has the same face weights and
ratio of partition functions. As a consequence, it is not important that the
Aztec diamond is centered at $(j,i-2j)$, but only that it has the
announced central face weight.

Doing the same for $m=j+1$ gives the expression of $w_{i,j}$.
\end{proof}

\begin{remark}
Due to the definition of integrable cross-ratio maps it is possible to express $(z_{i,j})_{i+j > 1}$ in terms of $(z_{i,j})_{i+j \in \{0,1\} }$ and $(\alpha_i, \beta_j)_{i,j\in \Z}$. However, these expressions can become arbitrarily complicated rational functions, so this is not in the spirit of giving explicit combinatorial expressions. It would be interesting to find an explicit combinatorial expression for the data of $z$ only in terms of $(z_{i,j})_{i+j \in \{0,1\} }$ and $(\alpha_i, \beta_j)_{i,j\in \Z}$, without resorting to a Bäcklund partner $w$.
\end{remark}

\subsection{Dual map}

The following construction helps both in the proof and the statement of singularities. Given an integrable cross-ratio map, we define the discrete 1-form $\dd{z}(v,v') = z(v') - z(v)$ for any two adjacent $v,v' \in \Z^2$. This 1-form is closed by construction. Additionally, we consider the \emph{dual discrete 1-form $\dd{z^*}$} given by
\begin{align}
        \dd{z^*}((i,j),(i+1,j)) &= \frac{\alpha_i}{z_{i,j}-z_{i+1,j}}, &
        \dd{z^*}((i,j),(i,j+1)) &= \frac{\beta_j}{z_{i,j}-z_{i,j+1}}.
\end{align}

The idea of this 1-form goes back to the study of \emph{discrete isothermic surfaces} \cite{bpdisosurfaces}. Our definition of the dual form uses slightly different conventions, with no minus sign introduced for different lattice directions and no complex conjugation. This simplifies the calculations and is in closer correspondence with the conventions used for the \emph{cross-ratio system} in \cite{bsintquad}.

The 1-form $\dd{z^*}$ is closed, which follows by elementary calculation from the cross-ratio
condition of Definition \ref{def:intcrmap}, also compare with \cite[Theorem 6]{bpdisosurfaces}. Any map
$z^*$ obtained by integrating $\dd{z^*}$ is called a \emph{dual map to
$z$}. Every dual map $z^*$ is also an integrable cross-ratio map with the same edge-labels as $z$, which may also be verified by elementary calculation. Additionally, assume $z,w$ are a Bäcklund pair.
By employing the cross-ratio conditions
of Definition \ref{def:backlundpair} we may moreover assume that
the dual maps $z^*, w^*$ satisfy
$w^*-z^* = \gamma(z-w)^{-1}$. We may see $z^*$ (resp.~$w^*$) as defined on the plane
at height $0$ (resp.~$1$) in $\Z^3$ as in Figure~\ref{fig:intcraztec},
so the previous formula gives vertical increments.
Finally, the following identity \cite[Section 5.4]{bsintquad}
{\small \begin{align}\label{eq:intcrstarform}
\dd{z^*}((i+1,j),(i,j)) +  \dd{z^*}((i,j),(i,j+1))  =
\frac{\alpha_i} {z_{i+1,j} - z_{i,j}} + \frac{\beta_{j}}{z_{i,j}- z_{i,j+1}}  = \frac{\alpha_i - \beta_{j}}{z_{i+1,j+1} - z_{i,j}},
\end{align}}
is called the \emph{three-legged form} and is useful in some of the upcoming calculations.

\subsection{Change of coordinates}

The following change of coordinates helps to better outline the structure of the data. It is also useful to describe singularities in Section~\ref{sec:sing_Backlund_pairs}. For $(i,j)\in\Z^2$, let
\begin{equation}
	z_{i,j} = \tilde{z}_{i-j,i+j}, \qquad \mbox{and } \tilde{z}_j=(\tilde{z}_{i,j})_{[i+j]_2=0}, 
  \label{eq:rotated_weights}
\end{equation}
 noting that
$\tilde z_{i,j}$ is only defined for $i+j\in 2\Z$.
We will use the tilde to signify this change of coordinates in the remainder of the paper. The inverse coordinate transform is
\begin{equation} \label{eq:rotated_weights_1}
\tilde z_{i,j} = z_{\frac{i+j}{2},\frac{-i+j}{2}}.
\end{equation}
An example is provided in Figure~\ref{fig:backlund_ic}.

\begin{remark}
In the rotated coordinates, the explicit expression of Theorem~\ref{theo:explintcr} takes the following simpler form: for all $(i,j)\in\Z^2$ such that $[i+j]_2=0$,
\begin{align*}
\tilde{z}_{i,j} =
\begin{cases}
Y(A_{j-1}[\tilde{z}_{i,[i]_2}],a)&
\text{  if $[j]_4 \in \{0,1\}$},\\
Y(A_{j-1}[\tilde{w}_{i,[i]_2}],a)&
\text{  if $[j]_4 \in \{2,3\}$},
\end{cases} \quad \quad
\tilde{w}_{i,j}=
\begin{cases}
Y(A_{j-1}[\tilde{z}_{i,[i]_2}],a)&
\text{  if $[j]_4 \in \{2,3\}$},\\
Y(A_{j-1}[\tilde{w}_{i,[i]_2}],a)&
\text{  if $[j]_4 \in \{0,1\}$.}
\end{cases}
\end{align*}

\end{remark}

\subsection{Singularities}\label{sec:sing_Backlund_pairs}

We now turn to studying singularities of integrable cross-ratio maps.
If we know $\tilde{z}_{j}$, $\tilde{z}_{j+1}$ and the functions $\alpha$ and $\beta$, then by Definition~\ref{def:intcrmap}, we also know $\tilde{z}_{j+2}$. Let us denote by $T$ the map representing the propagation of data
\[
T:(\hC)^{\Z_2\times \Z}\rightarrow (\hC)^{\Z_2\times \Z},\quad (\tilde{z}_{j},\tilde{z}_{j+1})\mapsto (\tilde{z}_{j+1},\tilde{z}_{j+2}).
\]

Let $m\geq 1$. An integrable cross-ratio map $z$ is said to be
\emph{$m$-closed} if $\alpha_i=\alpha_{i+m}$,
$\beta_i=\beta_{i+m}$, and $z_{i,j} = z_{i+m,j-m}$ for all $(i,j)\in
\Z^2$; the last condition is equivalent to $\tilde z_{i+2m,j} = \tilde z_{i,j}$ for all $(i,j) \in \Z^2$ such that $[i+j]_2=0$. A dual map $z^*$ to a closed map $z$ is not necessarily closed, instead $z^*$ has an additive monodromy $M[z^*] = z^*_{i+m,j-m} - z^*_{i,j}$ that does not depend on $i,j$ or the choice of the dual map $z^*$. A Bäcklund pair of integrable cross-ratio maps $z,w$ is called \emph{$m$-closed} if both $z$ and $w$ are $m$-closed. We claim that generically, for a given $m$-closed integrable cross-ratio map $z$ there exists an integrable cross-ratio map $w$, such that $z,w$ is an $m$-closed Bäcklund pair. Recall that we explained how to construct a Bäcklund pair from $w$ in the non-closed case in Section~\ref{subsec:Backlund_pairs_defi}. In the closed case, assume we begin by setting $\tilde w_{0,0} = \hat w$ for some formal variable $\hat w$. Then we iteratively express $\tilde w_{1,1}, \tilde w_{2,0}, \tilde w_{3,1}, \tilde w_{4,0}, \dots, \tilde w_{2m,0}$ as Möbius transforms of $\hat w$. As a Möbius transform generically has two fixed points, we can choose $\hat w$ to equal one of the fixed points. The remainder of $\tilde w$ is determined by the propagation $T$ and thus we have constructed an $m$-closed Bäcklund pair.

We prove the following singularity result.
\begin{theorem}\label{theo:intcrsingular}
  Let $m\geq 1$, and let $\tilde z$ be an $m$-closed
  integrable cross-ratio map such that $\tilde z_{i,0} = 0$ for all
  $i\in 2\Z$. Assume we can apply the propagation map $T$ to
  $(\tilde z_0, \tilde z_1)$ at least $2m-2$ times. Then
  for all $i\in 2\Z+1$,
  \begin{align}
     \tilde{z}_{i,2m-1} = \frac{\sum_{\ell=0}^{m-1}\left( \alpha_\ell - \beta_{\ell}\right) }{\sum_{\ell=0}^{m-1} \frac1{\tilde z_{2\ell+1,1}} \left(\alpha_{\ell} - \beta_{-\ell-1} \right)} = \frac1{M[z^*]}\sum_{\ell=0}^{m-1}\left( \alpha_\ell - \beta_{\ell}\right) .
  \end{align}
\end{theorem}

\begin{remark}\label{rk:alphabeta}
 Prior to proving the theorem, let us mention again that this may be
  seen as an equality on rational functions of formal variables
  $\alpha,\beta$ and $\tilde{z}_1$. Thus we may assume that there is
  an $\ell$ such that $\alpha_\ell \neq \beta_{-\ell-1}$, indeed, when
  this is not the case, the conclusion of the theorem (taken as a
  formal expression specified to that case) shows that
  $\tilde{z}_{i,2m-1}$ is undefined, so the propagation map $T$ could
  not, in fact, be applied $2m-2$ times.
\end{remark}

\begin{proof}
Note that the right equality follows immediately from the definition of the monodromy $M[z^*]$ of the dual 1-form.

Let $\tilde w$ be an $m$-closed integrable cross-ratio map such that
$\tilde z, \tilde w$ is a Bäcklund pair with $\gamma=1$.

The initial condition in the claim is a case of $(m,2)$-Devron initial
condition defined in Definition~\ref{def:sing}, see Figure~\ref{fig:backlund_ic}. Therefore, by
Theorem~\ref{theo:devron_sing}, we know that for the corresponding
solution of the dSKP recurrence, if $[i-j-2m]_4=0$ then $x(i,j,2m-2)=x(i+1,j+1,2m-2)$. This
solution is also explicitly described in Equation \eqref{eq:proof_Backlund},
and noting that in this case $[i-j+(2m-2)]_4 = 2$, the previous
equation becomes
\begin{equation*}
  w_{\frac{i+j}{2}+m-1,\frac{-(i+j)}{2}+m-1} =
  w_{\frac{i+j}{2}+m,\frac{-(i+j)}{2}+m-2},
\end{equation*}
or equivalently, $\tilde{w}_{i+j,2m-2}=\tilde{w}_{i+j+2,2m-2}$. In
other words, $\tilde{w}_{2m-2}$ is constant. Generally, if $\tilde
w_{i,j} = \tilde w_{i+2,j} \neq \tilde w_{i+1,j-1}$ then $\tilde
w_{i+1,j+1}$ is not defined by Equation \eqref{def:intcrmap}. However,
if $\tilde z_{i,j}, \tilde z_{i+2,j}, \tilde z_{i+1,j-1}$ are pairwise
different and different from $\tilde w_{i,j}, \tilde w_{i+2,j}, \tilde
w_{i+1,j-1}$ then $\tilde z_{i+1,j+1}$ is defined and equal to $\tilde
w_{i,j}$. As a consequence, $\tilde{z}_{2m-1}$ is constant.

We now compute this constant, for which we choose to
compute $\tilde w_{2m-2,2m-2}$ using Theorem~\ref{theo:D}. The
corresponding Aztec diamond has size $2m-3$, and its weights can be
obtained from the previous solution of dSKP; a quick check shows
that they are those of Figure~\ref{fig:Dbacklund}. Let $D$ be
the corresponding operator as in the discussion preceding
Theorem~\ref{theo:D}, and let $v\in \C^F$ be a non-zero element of
$\ker D^T$, which exists by a dimension argument.

\begin{figure}[tb]
  \centering
  \includegraphics[width=10cm]{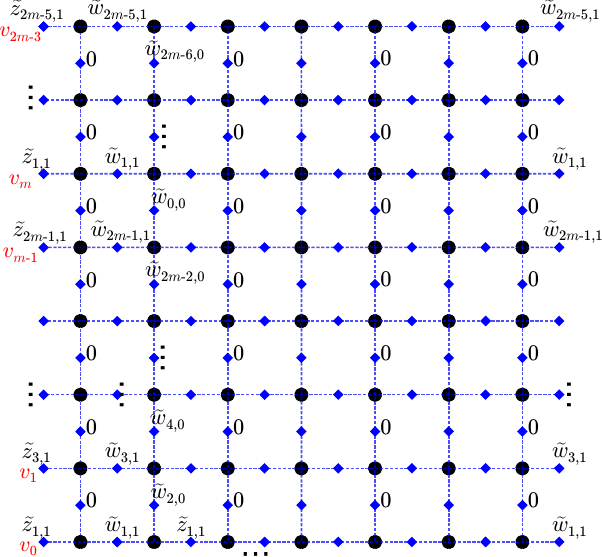}
  \caption{The weights on the graph defining $D$ in the case of a
    $m$-closed Bäcklund pair with $\tilde{z}_0 \equiv 0$. The same
    $4$ columns of weights repeat throughout the graph. In particular
    there are $m-1$ columns with weights $0$.
  }
  \label{fig:Dbacklund}
\end{figure}

\begin{figure}[tb]
  \centering
  \includegraphics[width=10cm]{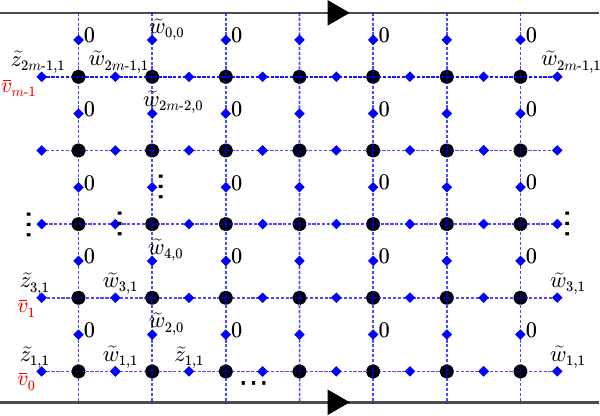}
  \caption{The graph on a cylinder obtained from that of
      Figure~\ref{fig:Dbacklund} by quotienting the $m$-periodicity.}
  \label{fig:Dbacklundcyl}
\end{figure}

Consider the graph on a cylinder obtained by a $(0,m)$ quotient of the
previous graph, inheriting the face weights consistently, see
Figure~\ref{fig:Dbacklundcyl}; we denote its black vertices by $\bar{B}$,
its faces by $\bar{F}$, and the corresponding operator from
$\C^{\bar{B}}\oplus \C^{\bar{B}}$ to $\C^{\bar{F}}$ by $\bar{D}$. Each
element of $\bar{F}$ corresponds to one or two elements of $F$. By
summing the entries of $v$ from these one or two elements, we get a
vector $\bar{v}\in\C^{\bar{F}}$. It is easy to see that $\bar{v}\in
\ker \bar{D}^T$. Note also that the weighted sums of the leftmost entries of
$v$, as in Theorem~\ref{theo:D}, can be computed from
$\bar{v}$, using the closedness of $\tilde z_1$:
\begin{equation}
  \label{eq:sumvbar}
  \begin{split}
    \sum_{\ell=0}^{2m-3}v_{\ell} & = \sum_{\ell=0}^{m-1}\bar{v}_\ell,\\
    \sum_{\ell=0}^{2m-3}a_{\ell} v_{\ell} & =
    \sum_{\ell=0}^{2m-3} \tilde z_{2\ell +1,1} v_{\ell} =
    \sum_{\ell=0}^{m-1}\tilde z_{2\ell + 1,1}\bar{v}_\ell.
  \end{split}
\end{equation}

If $\bar v \equiv 0$, then these two sums are $0$, and by
Theorem~\ref{theo:D}, $\tilde{w}_{2m-2,2m-2}$ is undefined, which means
that for these initial conditions the map $T$ cannot be applied
$2m-2$ times. We now suppose that $\bar v$ is not the zero vector.

We have $|\bar F|=2|\bar B|+m$, so
  $\dim \ker \bar{D}^T \geq m$. We will show that this dimension is in
  fact $m$ and provide a basis of $\ker \bar{D}^T$, which we will use
  to get information on $\bar v$. First, for each of the $m-1$ columns
  of zeros in the quotient graph, we can put ones on this whole
  column, and zeros everywhere else; it is easy to check that this
  gives $m-1$ linearly independent vectors in $\ker \bar{D}^T$. The $m$-th vector we
  introduce has the following entries:
  \begin{itemize}
  \item on each element of $\bar{F}$ carrying a weight $\tilde
    z_{2\ell+1,1}$, the entry is $\frac{\alpha_\ell-\beta_{-\ell-1}}{\tilde z_{2\ell+1,1}}$;
  \item on each element of $\bar{F}$ carrying a weight $\tilde w_{2\ell+1,1}$, the
    entry is $\frac{\alpha_\ell-\beta_{-\ell-1}}{\tilde w_{2\ell+1,1}}$;
  \item on each element of $\bar{F}$ carrying a weight $\tilde w_{2\ell,0}$, the
    entry is $\frac{-1}{\tilde w_{2\ell,0}}$;
  \item on each element of $\bar{F}$ carrying a weight $0$ and on the
    same row as $\tilde w_{2\ell,0}$, the
    entry is $\frac{-1}{\tilde w_{2\ell,0}}$.
  \end{itemize}
  By Remark~\ref{rk:alphabeta}, this is not in the span of the
  previous $m-1$ vectors. We now check that this indeed defines a vector in $\ker
  \bar{D}^T$, which means we have to check two linear conditions for
  each element of $\bar{B}$. By periodicity of the vector, we just
  have to check it for the first two columns of $\bar{B}$. On the
  first column at row $\ell$, the first relation to check is
  \begin{equation}
    \label{eq:proofvbar}
    \frac{\alpha_\ell-\beta_{-\ell-1}}{\tilde w_{2\ell+1,1}} + \frac{-1}{\tilde w_{2\ell+2,0}}
    = \frac{\alpha_\ell-\beta_{-\ell-1}}{\tilde z_{2\ell+1,1}} + \frac{-1}{\tilde w_{2\ell,0}}.
  \end{equation}
        Recall that $\tilde z_{2\ell,0},
          \tilde z_{2\ell +2,0}$ are zero, so that we may assume that
          $\tilde w_{2\ell+1,-1}$ is zero as well; a short computation
          shows that \eqref{eq:intcrmapdef} still holds for
          $w$. Moreover, applying Equation~\eqref{eq:intcrstarform}
          for $w$ at $i=\ell, j=-\ell-1$, we obtain
        \begin{align}
                \frac{\alpha_\ell - \beta_{-\ell-1}}{\tilde
          w_{2\ell+1,1}} = \frac{\alpha_\ell} {\tilde w_{2\ell+2,0}} -
          \frac{\beta_{-\ell-1}}{\tilde w_{2\ell,0}}.
        \end{align}
        Substituting this expression into Equation \eqref{eq:proofvbar}, we obtain the integral of the dual 1-form along the closed path corresponding to $\tilde z_{2\ell,0},\tilde w_{2\ell,0}, \tilde w_{2\ell+1,-1}, \tilde w_{2\ell+2,0}, \tilde z_{2\ell+2,0}, \tilde z_{2\ell+1,1}, \tilde z_{2\ell,0}$. Thus Equation \eqref{eq:proofvbar} follows from the fact that the dual 1-form is closed.

  The second condition for the first column is
  \begin{equation}\label{equ:Back_relation}
    \tilde w_{2\ell+1,1} \times \frac{\alpha_\ell-\beta_{-\ell-1}}{\tilde w_{2\ell+1,1}} + 0 \times \frac{-1}{\tilde w_{2\ell+2,0}}
    = \tilde z_{2\ell+1,1} \times \frac{\alpha_\ell-\beta_{-\ell-1}}{\tilde z_{2\ell+1,1}} + 0 \times\frac{-1}{\tilde w_{2\ell,0}},
  \end{equation}
  which is trivial. We now turn to the second column. The first
  condition is again \eqref{eq:proofvbar}.
  The second condition is
  \begin{equation}
    \tilde w_{2\ell+1,1} \times \frac{\alpha_\ell-\beta_{-\ell-1}}{\tilde w_{2\ell+1,1}} + \tilde w_{2\ell+2,0} \times \frac{-1}{\tilde w_{2\ell+2,0}}
    = \tilde z_{2\ell+1,1} \times \frac{\alpha_\ell-\beta_{-\ell-1}}{\tilde z_{2\ell+1,1}} + \tilde w_{2\ell,0} \times\frac{-1}{\tilde w_{2\ell,0}},
  \end{equation}
  which is also trivial.

Thus we have found $m$ linearly independent vectors in
  $\ker \bar{D}^T$. We claim that this space has dimension
  $m$. Indeed, if its dimension was higher, using the rank-nullity
  theorem we would have $\dim \ker \bar{D} \geq 1$, so there would be
  a non-zero vector $\bar{u}\in \C^{\bar B}\oplus \C^{\bar B}$ in this
  subspace. By putting the same entries as those of $\bar{u}$ on the
  initial graph, we would get a non-zero vector
  $u\in \C^B \oplus \C^B$, and it is easy to see that $u\in \ker
  D$. By the rank-nullity theorem again, this implies that
  $\dim \ker D^T \geq 2$. Using Proposition~\ref{prop:dimker2}, this
  implies again that the map $T$ could not be
  applied $2m-2$ times.

  Therefore, $\bar{v}$ is a combination of the $m$
  vectors we described. This implies that there is a constant $\lambda$ such that
$(\bar v_\ell)_{0\leq \ell \leq m-1} = \lambda \
(\frac{\alpha_\ell-\beta_{-\ell-1}}{\tilde z_{2\ell+1,1}})_{0\leq \ell
\leq m-1}$. By \eqref{eq:sumvbar} and Theorem~\ref{theo:D}, this gives
the claimed value; note that if $\lambda=0$, again the ratio is not
defined and we conclude as before.
\qedhere

\end{proof}

\begin{remark}
        Let us write $S[\alpha,\beta] := \sum_{\ell=0}^{m-1}(\alpha_\ell-\beta_\ell)$. Assuming $\tilde z$ is $m$-closed and singular, in the sense that $\tilde z_0$ is constant, Theorem \ref{theo:intcrsingular} states that
        \begin{align}
                \tilde z_{2m-1} - \tilde z_0 = \frac{S[\alpha, \beta]}{M[z^*]}.
        \end{align}
        Thus there is a simple relation between the positions of the two singularities and the additive monodromy of $z^*$. Moreover, note that $z^*$ features a repeating singularity as well. In particular, one can immediately verify that $\tilde z^*_{-1} = \tilde z^*_{2m} \equiv \infty$. It would be interesting to understand if and how singularity theorems also hold in more general setups when $z$ is only quasi-periodic.
\end{remark}

\section{Discrete holomorphic functions and orthogonal circle patterns}\label{sec:dhol}

\subsection{Definitions and relation to Bäcklund pairs}
There are different kinds of maps that are considered to be discretizations of holomorphic functions in the literature \cite{bmsanalytic, clrtembeddings, duffin, kenyonisoradial, schramm, smirnov, stephenson} . The definition we use here is due to Bobenko and Pinkall \cite{bpdisosurfaces} and independently Nijhoff and Capel \cite{nc95}. It is a specific instance of an integrable cross-ratio map when $\alpha_i \equiv -\mu, \beta_j \equiv \mu^{-1}$ for some $\mu \in \C\setminus\{0\}$, although without loss of generality we assume $\mu =1$ and therefore  $\alpha_i\equiv -1,\beta_j\equiv 1$. These discrete holomorphic functions are also solutions to the discrete KdV equation \cite[Section 5]{ksclifford}.

\begin{definition}\label{def:dhol}
A \emph{discrete holomorphic function} is a map $z: \Z^2 \rightarrow \hC$ such that, for every quad of $\Z^2$, the following holds
\begin{equation}\label{eq:crmone}
\cro(z_{i,j},z_{i+1,j},z_{i+1,j+1},z_{i,j+1}) = -1.
\end{equation}
\end{definition}

Recall that for a given $\gamma$ and a given point in $\C$, there is a unique integrable cross-ratio map $w$ such that $z,w$ is a Bäcklund pair. If $z$ is a discrete holomorphic function, then $w$ is also a discrete holomorphic function. Note that the cross-ratios on the side quads $(z_{i,j}, z_{i+1,j}, w_{i+1,j}, w_{i,j})$ and $(z_{i,j}, z_{i,j+1}, w_{i,j+1}, w_{i,j})$ are $\alpha_i\gamma^{-1} \equiv -\gamma^{-1}$ and $\beta_i\gamma^{-1} \equiv \gamma^{-1}$ respectively. Therefore, the cross-ratios cannot be $-1$ on all quads simultaneously.

A useful fact is that, in the specific case of a discrete holomorphic function $z$ and $\gamma = 1$, a Bäcklund pair $z,w$ can be constructed explicitly as stated by the following.

\begin{lemma}\label{lem:holom_Backlund}
Consider a discrete holomorphic function $z:\Z^2\rightarrow\hC$. Set $\gamma=1$, and for a given $(i',j')\in\Z^2$, set $w_{i',j'}=z_{i',j'+1}$. Then, the discrete holomorphic function $w:\Z^2\rightarrow\hC$ such that $z,w$ is a Bäcklund pair is explicitly given by
\begin{equation}\label{equ:holom_Backlund}
	w_{i,j}=z_{i,j+1},
\end{equation}
for all $(i,j)\in\Z^2$.
\end{lemma}
\begin{proof}
The proof consists in an explicit reconstruction of $w$ starting from the point $w_{i',j'}=z_{i',j'+1}$ using Definition~\ref{def:backlundpair} as described in Section~\ref{subsec:Backlund_pairs_defi}.
\end{proof}

Let us discuss briefly the geometry of a Bäcklund pair as in Lemma~\ref{lem:holom_Backlund}. On the bottom quads $(z_{i,j}, z_{i+1,j}, z_{i+1,j+1}, z_{i,j+1})$ and the top quads $(w_{i,j}, w_{i+1,j}, w_{i+1,j+1}, w_{i,j+1})$ the cross-ratios are $-1$ by the definition of discrete holomorphic functions. On the first kind of side quads $(z_{i,j}, z_{i+1,j}, w_{i+1,j}, w_{i,j})$ the cross-ratios are the ratio of $\alpha_i \equiv -1$ and $\gamma \equiv 1$, which is therefore also $-1$. In contrast, on the second kind of side quads $(z_{i,j}, z_{i,j+1}, w_{i,j+1}, w_{i,j})$  the cross-ratios are the ratio of $\beta_i \equiv 1$ and $\gamma \equiv 1$, which is therefore $+1$. A cross-ratio of four points is only $+1$ if two non-consecutive points coincide, which is satisfied on the second kind of side faces due to $w_{i,j} = z_{i,j+1}$. 

\begin{remark}
	Let us add an observation which we do not need in the following, but which may be of separate interest. The explicit construction of the Bäcklund pair as given in Lemma~\ref{lem:holom_Backlund} is not unique to the discrete holomorphic case. Instead, it suffices that the edge labels $\alpha_i$ and $\beta_i$ are \emph{constant}. In this case we may construct the Bäcklund pair in the same way, except that this is a Bäcklund pair for $\gamma = \beta$.
\end{remark}

Using Lemma~\ref{lem:holom_Backlund}, we now state the dSKP relation given by Lemma~\ref{lem:intcrdskp} in the setting of discrete holomorphic functions, thus recovering a result of~\cite[Section 5]{ksclifford}.

\begin{corollary}\label{lem:dholdskp}
Let $z$ be a discrete holomorphic function. Then, the following equation holds for all $(i,j)\in\Z^2$,
\begin{equation}\label{eq:dholdskp}
\frac{(z_{i,j}-z_{i+1,j})(z_{i+1,j+1}-z_{i+1,j+2})(z_{i,j+2}-z_{i,j+1})}{(z_{i+1,j}-z_{i+1,j+1})(z_{i+1,j+2}-z_{i,j+2})(z_{i,j+1}-z_{i,j})} = -1.
\end{equation}
\end{corollary}

\subsection{Explicit solution}

We now apply Theorem~\ref{theo:explintcr} to the Bäcklund pair $z,w$ corresponding to holomorphic functions thus giving an explicit expression for $z_{i,j}$ for all $(i,j)\in\Z^2$ such that $i+j\geq 1$, as a function of $(z_{i',j'})_{i'+j'\in\{0,1,2\}}$. Observe that Theorem~\ref{theo:explintcr} gives two ways of expressing $z_{i,j}$: one using $z_{i,j}$ of course, and the other using $z_{i,j}=w_{i,j-1}$. As a consequence, $z_{i,j}$ can be expressed using an Aztec diamond of size $i+j-1$ or $i+j-2$; we use the second way in Corollary~\ref{cor:expl_dholom} below since the Aztec diamond is smaller. Note that one can also see on the level of combinatorics that the two ways are consistent:
some edges of the Aztec diamond can be removed, which impose some edges to be present on the boundary, and allow to reduce the size of the Aztec diamond by one in the case where we use the one having size $i+j-1$, see also Remark~\ref{rem:dholom_explicit}. An illustration of Corollary~\ref{cor:expl_dholom} is given in Figure~\ref{fig:dhol_ic}.

\begin{corollary}\label{cor:expl_dholom}
Let $z:\Z^2 \rightarrow \hC$ be a discrete holomorphic function, and consider the graph $\Z^2$ with face weights $(a_{i,j})_{(i,j)\in\Z^2}$ given by,
\[
a_{i,j}=
\begin{cases}
z_{\frac{i+j+[i+j]_2}{2},\frac{-(i+j)+[i+j]_2}{2}} & \text{ if }[i-j]_4\in\{0,3\},\\
z_{\frac{i+j+[i+j]_2}{2},\frac{-(i+j)+[i+j]_2}{2}+1} & \text{ if }[i-j]_4\in\{1,2\}.
\end{cases}
\]
Then, for all $(i,j)\in\Z^2$ such that $i+j\geq 2$, we have
\begin{align*}
z_{i,j}&=
\begin{cases}
Y(A_{i+j-2}[z_{i',(j-1)'}],a)&
\text{  if $[i+j]_4 \in \{0,3\}$},\\
Y(A_{i+j-2}[z_{i',(j-1)'+1}],a)&
\text{  if $[i+j]_4 \in \{1,2\}$.}
\end{cases}
\end{align*}
\end{corollary}

\begin{figure}[tb]
  \centering
  \includegraphics[width=14cm]{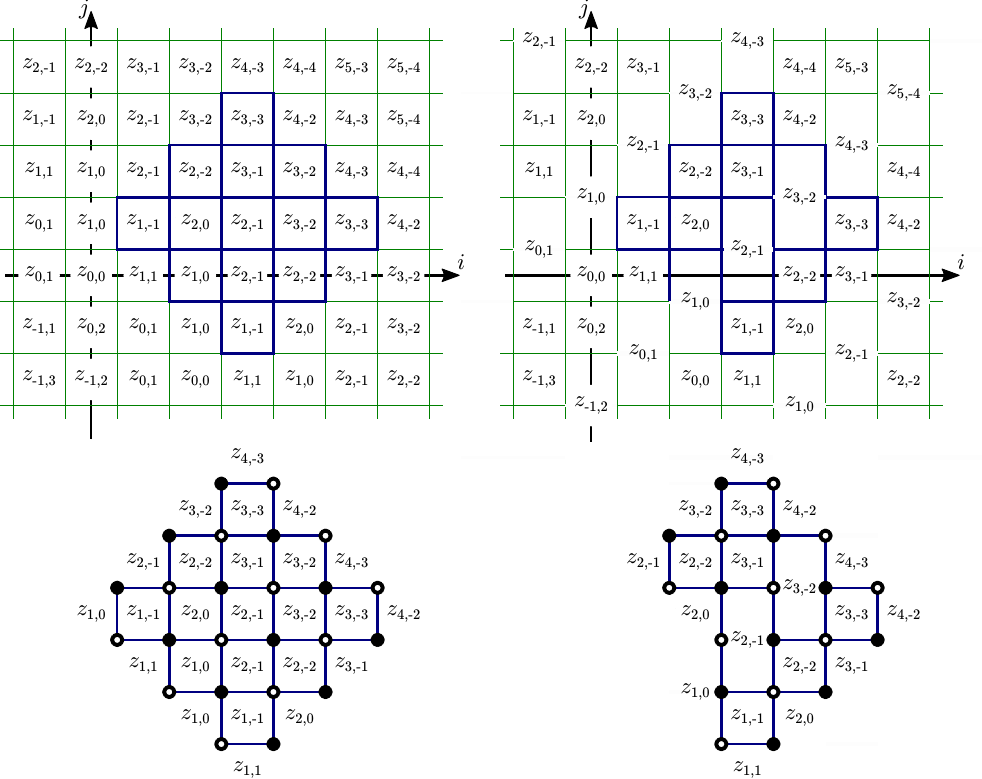}
  \caption{Face weights for the explicit solution of discrete
    holomorphic functions. Top left: face weights deduced from the
    Bäcklund pair case with $w_{i,j}=z_{i,j+1}$. Top right: graph
    obtained after removing edges of resulting dimer weight
    $0$. Bottom left: Aztec diamond of size $3$ used to compute $z_{4,1}$,
    using $w_{4,0}$ in the Bäcklund pair case. Bottom right: graph
    with the same ratio of partition functions, obtained from the
    Aztec diamond via the existence of forced dimers.}
  \label{fig:dhol_ic}
\end{figure}

\begin{remark}\label{rem:dholom_explicit}\leavevmode
There are pairs of diagonals with identical weights. Returning
to the definition of the oriented dimer partition function, an edge
separating two squares with the same face has two possible
orientations that cancel each other, and can thus be removed. As a consequence, in this case, the lattice $\Z^2$ can be transformed into a lattice consisting of a repeating pattern of pairs of diagonals of squares, followed by a diagonal of hexagons, see Figure~\ref{fig:dhol_ic}.
One can also obtain this resulting graph made of hexagons and squares via the method
of \emph{crosses and wrenches}, using the solution to the dSKP recurrence associated to the height function of an initial condition different than $[i+j]_2$.
This method was developed by Speyer for
the dKP recurrence \cite{Speyer}, and extended to the dSKP
recurrence in \cite[Section~2]{paper1}.
\end{remark}

\subsection{Discrete holomorphic functions and P-nets}\label{sec:discrete_holom_P_nets_0}

We now explain a connection between discrete holomorphic functions and P-nets. This link provides an alternative explicit expression for $z_{i,j}$, and is also of use in the next section on singularities.

Let $\Z^2_\pm$ denote the even and odd sublattices $\{(i,j) \in \Z^2 : (-1)^{i+j} = \pm1\}$, where we consider two vertices in $\Z_\pm^2$ to be adjacent if they are at graph distance 2 in $\Z^2$.

Consider a discrete holomorphic function $z:\Z^2\rightarrow \hC$, written in the rotated coordinates $(\tilde{z}_{i,j})_{[i+j]_2=0}$ given by Equation~\eqref{eq:rotated_weights}. Let $p$ be the restriction of $\tilde{z}$ to $\Z^2_+$, and $q$ the restriction to $\Z^2_-$, \emph{i.e.}, for all $(i,j)\in\Z^2$,
\begin{align}\label{eq:dholpnet}
\begin{split}
p_{i,j} & = \tilde{z}_{2i,2j} = z_{i+j,-i+j},\\
q_{i,j} & = \tilde{z}_{2i+1,2j+1} = z_{i+j+1,-i+j}.
\end{split}
\end{align}
Then, we have the following.

\begin{lemma}\emph{(\cite[Lemma 6]{bpdiscsurfaces})}\label{lem:dholpnet}
        Let $z:\Z^2\rightarrow \hC$ be a discrete holomorphic function. Then both $p$ and $q$ are P-nets. Conversely, given a P-net $p$ there is a (complex) one-parameter family of P-nets $q'$ such that $p$ and $q'$ together are a discrete holomorphic function.
\end{lemma}

\begin{question}\label{que:pnetdhol} As a consequence of Lemma~\ref{lem:dholpnet} one can use the P-net explicit solution of Theorem~\ref{theo:explpnet} to give an explicit expression of $z_{i,j}$, separating the even and odd cases. It is striking to see that the Aztec diamond with this approach is typically only half the size of the one given by Corollary~\ref{cor:expl_dholom}. At this stage we do not understand how to prove directly that the ratio partition functions of these two Aztec diamonds are equal, and pose this as an intriguing open question.
\end{question}

\subsection{Singularities}

In this section, we study singularities of discrete holomorphic functions. The first part consists in writing down the integrable cross-ratio/Bäcklund pairs result (Theorem~\ref{theo:intcrsingular}) in the specific case of discrete holomorphic functions, thus recovering a result of~\cite{yao}. Next, using the connection between discrete holomorphic functions and P-nets, we provide an alternative proof of this result, as well as a refined version thereof.

\subsubsection{Immediate consequences}

Recall the definition of the map $T$ describing the propagation of data
\[
T:(\hC)^{\Z_2\times \Z}\rightarrow (\hC)^{\Z_2\times \Z},\quad (\tilde{z}_{j},\tilde{z}_{j+1})\mapsto (\tilde{z}_{j+1},\tilde{z}_{j+2}).
\]

Let $m\geq 1$. A discrete holomorphic function $z$ is said to be
\emph{m-closed} if it is $m$-closed as integrable cross-ratio map,
that is if $z_{i,j} = z_{i+m,j-m}$ for all $(i,j)\in \Z^2$, or
equivalently $\tilde{z}_{i+2m,j}=\tilde{z}_{i,j}$ for all $i,j\in\Z$
such that $[i+j]_2=0$. Then, using that $\alpha_i\equiv -1$, $\beta_j\equiv 1$, Theorem~\ref{theo:intcrsingular} becomes ~\cite[Theorem 1.5]{yao}.
\begin{corollary}[\cite{yao}]
Let $m\geq 1$, and let $\tilde{z}$ be an $m$-closed discrete holomorphic function such that $\tilde{z}_{i,0}=0$ for all $i\in 2\Z$. Assume we can apply the propagation map $T$ to $(\tilde{z}_0,\tilde{z}_1)$ at least $2m-2$ times. Then,
for all $i\in 2\Z+1$,
\[
\tilde{z}_{i,2m-1}=\Bigl(\frac{1}{m}\sum_{\ell=0}^{m-1}\frac{1}{\tilde{z}_{2\ell+1,1}}
\Bigr)^{-1},
\]
that is $\tilde{z}_{2m-1}$ is constant with value equal to the harmonic mean
of $\tilde{z}_1$.

\end{corollary}

\subsubsection{Further singularity results}

Let us prove a preparatory lemma.

\begin{lemma}\label{lem:pnetsinganddhol}
Let $m\in 2\N+2$, let $z:\Z^2\rightarrow\hC$ be an $m$-closed discrete holomorphic function such that, for all $i\in\Z$, $z_{i,-i} = 0$, and let $p,q$ be the associated P-nets. Then
\begin{align*}
\sum_{i=0}^{m-1}\frac{(-1)^i}{ p_{i,1}}= 0.
\end{align*}
\end{lemma}
\begin{proof}
Using $\alpha_i\equiv -1$, $\beta_j\equiv 1$, and $z_{i,-i} = 0$, Equation \eqref{eq:intcrstarform} becomes
\begin{align*}
        \frac{{-1}}{- z_{i,-i+1}} + \frac{1} {z_{i+1,-i}}  = \frac{2}{z_{i+1,-i+1}}.\label{equ:Back_relation_holom}
\end{align*}
Therefore, using the relation between P-nets and holomorphic functions \eqref{eq:dholpnet}, the claim becomes
\begin{align*}
        \sum_{i=0}^{m-1} \frac{(-1)^i}2\left(\frac{{1}}{ z_{i,-i+1}} + \frac{1} {z_{i+1,-i}}\right) = 0,
\end{align*}
which is clearly true because of the closedness of $z$ and the fact that $m$ is even.\qedhere

\end{proof}

We are now ready to state our refined singularity result. More precisely, we reprove that for $m$ odd, $\tilde{z}_{2m-1}$ is indeed identically equal to this harmonic mean, while for $m$ even, we show that $\tilde{z}_{2m-2}$ is in fact already identically equal to this harmonic mean, \emph{i.e.}, the singularity appears one step earlier than predicted by Yao. Note that this is not a contradiction, it just means that the map $T$ cannot be iterated $2m-2$ times.

\begin{figure}[t]
        \centering
        \frame{\includegraphics[height=5.7cm]{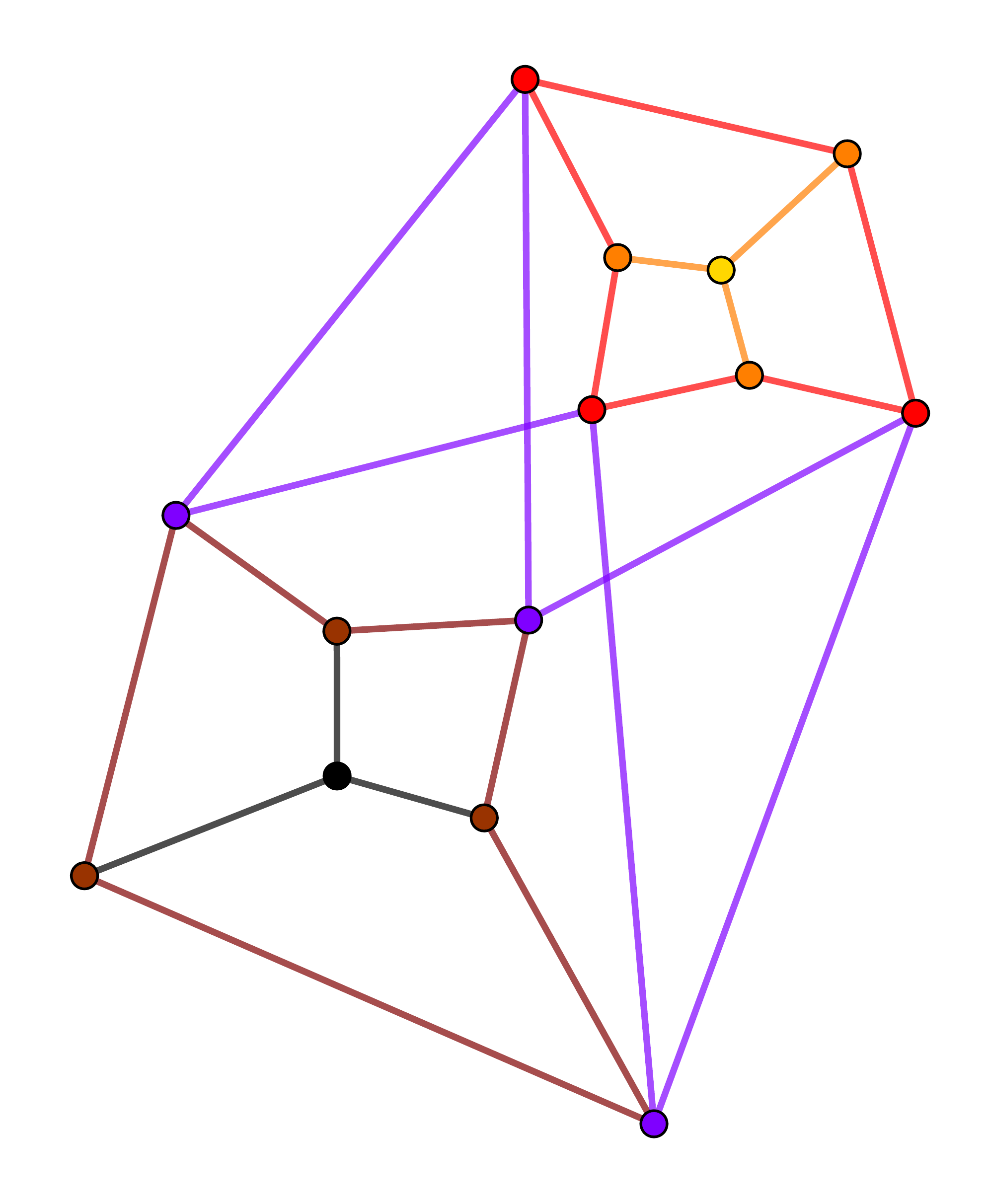}}
        \frame{\includegraphics[height=5.7cm,angle=0]{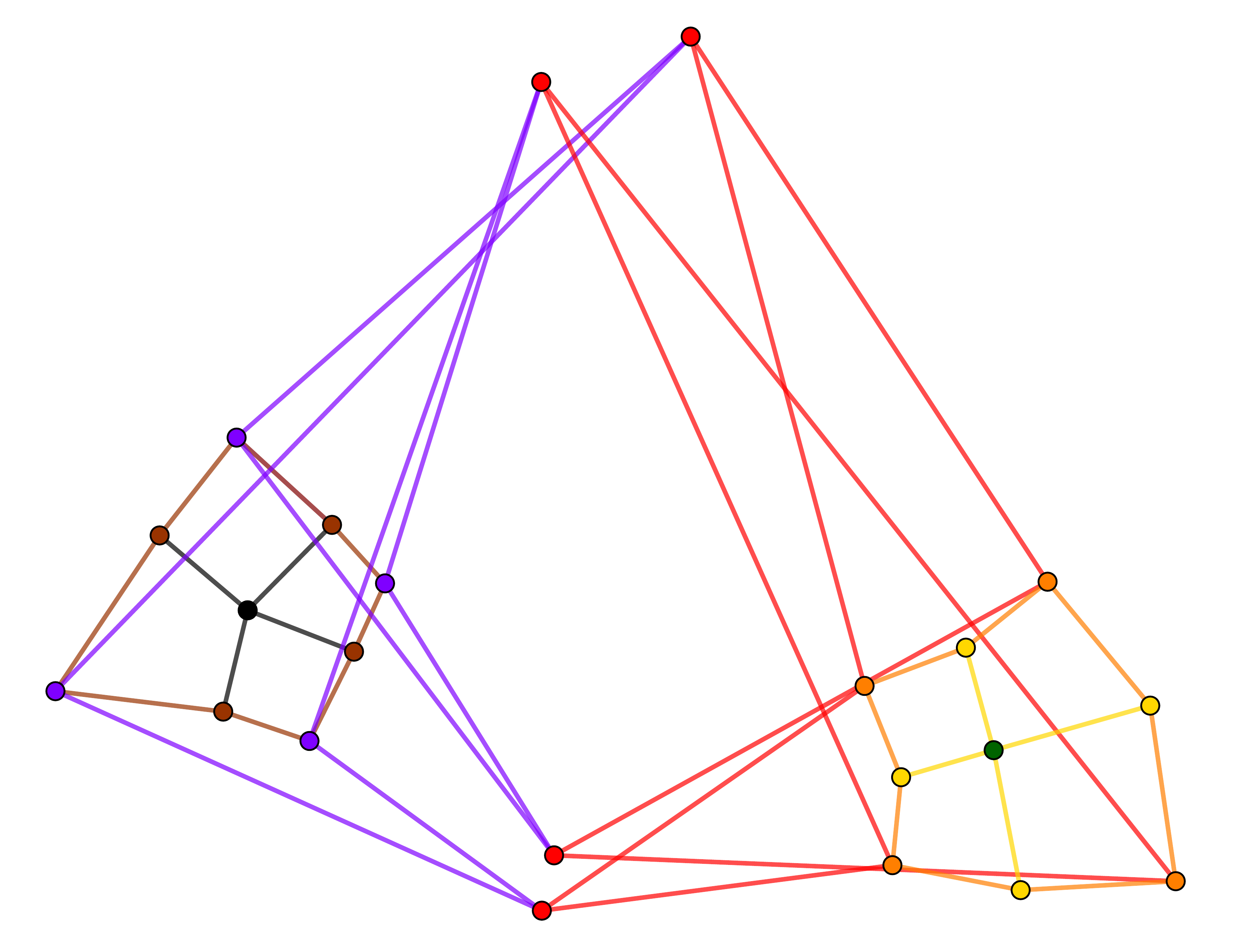}}
        \frame{\includegraphics[width=5.7cm,angle=-90]{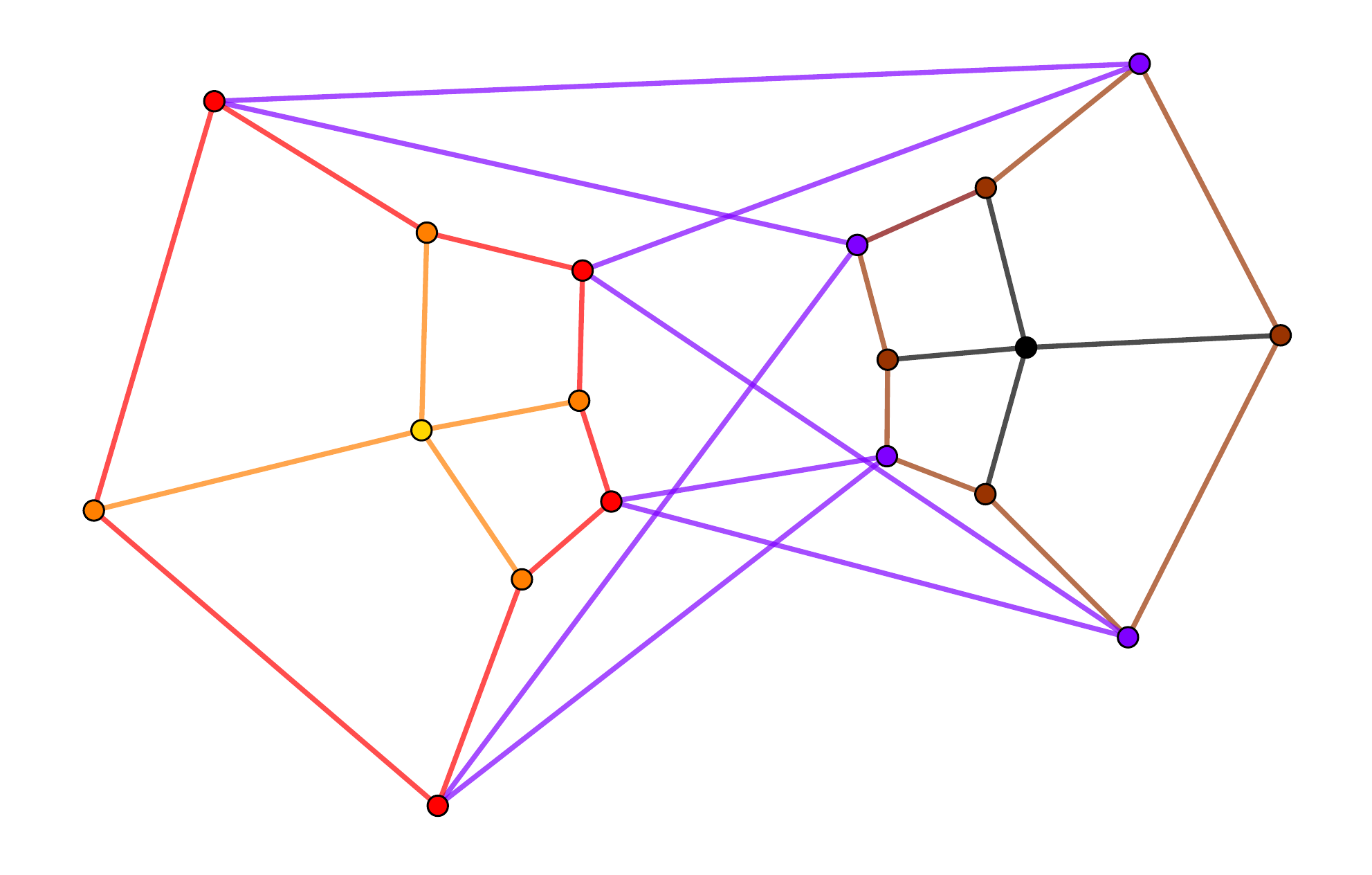}}
        \caption{Propagation of discrete holomorphic functions with an
          initial singularity at the black point. Left: 3-closed
          case, which degenerates after 4
          steps. Center and right: 4-closed case, which degenerates generically after 5 steps
          (center), and in special cases after 4 steps (right).}
        \label{fig:dholsingularity}
\end{figure}

\begin{theorem}[\cite{yao} ($m$ odd)]\label{theo:dholsingularity}
Let $m\geq 1$ and let $n=2m-2$ if $m$ is odd, and $n=2m-3$ if $m$ is even. Let $\tilde z$ be an $m$-closed discrete holomorphic function such that $\tilde z_{i,0} = 0$ for all $i\in 2\Z$. Assume we can apply the propagation map $T$ to $(\tilde z_0, \tilde z_1)$ at least $n$ times. Then, for all $i\in\Z$ such that $[i+n+1]_2=0$, we have
\begin{align*}
\tilde z_{i,n+1} = \left(\frac{1}{m}\sum_{\ell=0}^{m-1} \frac1{\tilde z_{2\ell+1,1}} \right)^{-1}.
\end{align*}
In other words, $\tilde z_{n+1}$ is constant with value equal to the harmonic mean of $\tilde z_1$.
\end{theorem}
\proof{
Consider the P-nets $p,q$ associated to $z$, where $p,q$ are defined on $\{(i,j) \in \Z^2 : j \geq 0\}$. The row $p_0$ corresponds to the singular initial row that is mapped to 0, see Figure \ref{fig:dholnetexplanation}. We consider the case of odd $m$ first.
Define a map $r$ on $\{(i,j) \in \Z^2 : j \geq -1\}$ as a continuation of $q$ by
\begin{align*}
                r_{i,j} = \begin{cases}
                q_{i,j} & j \geq 0, \\
                0 & j = -1.
                \end{cases}
\end{align*}
A small calculation shows that $r$ satisfies the P-net condition of Definition \ref{def:pnet} also in row $0$, where the calculation uses that $r_1=\tilde{z}_3$ is completely determined by $p_0=\tilde{z}_0, r_0=\tilde{z}_1$ via the propagation of discrete holomorphic functions. Thus $r$ is also a P-net.
Therefore, we can apply Theorem \ref{theo:pnetsingularity} to $r$ and obtain that $z$ is singular after $2m-2$ iterations of discrete holomorphic propagation, becoming equal to the harmonic mean of $r_0 = \tilde{z}_1$.

Next, we consider the case of even $m$. In this case, due to Lemma \ref{lem:pnetsinganddhol}, we have that
\begin{align}
                \sum_{i=0}^{m-1}(-1)^i \frac1{p_{i,1}} = 0.
\end{align}
        Therefore $p$ satisfies the assumptions of Theorem \ref{theo:pnetpremature}, and thus $z$ becomes singular after $2m-3$ iterations of discrete holomorphic propagation, becoming equal to the harmonic mean of $p_1 = \tilde{z}_2$. Due to Equation~\eqref{equ:Back_relation_holom},
        the harmonic mean of $p_1$ is the harmonic mean of $\tilde{z}_1$.\qed
}

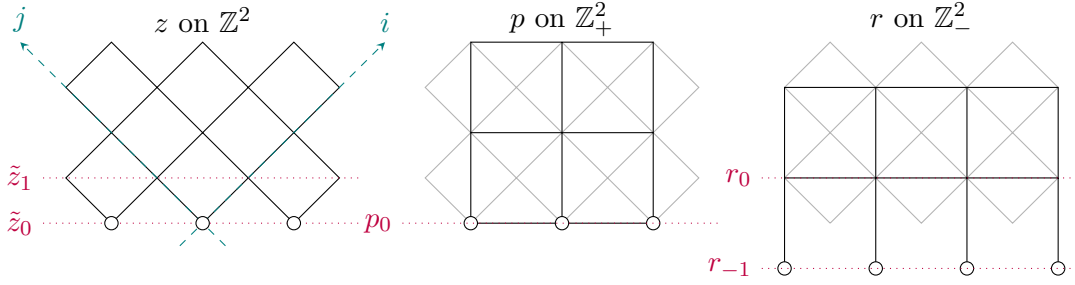
\begin{figure}
        \centering
        \begin{tikzpicture}[scale=0.6,baseline={(current bounding box.center)}]
                \useasboundingbox (0,-2) rectangle (6,4.5);
                \foreach \i in {0, 1, 2, 3, 4, 5, 6} {
                        \foreach \j in {0, 1, 2, 3, 4} {
                                \coordinate (v\i\j) at (\i,\j);
                        }
                }
                \draw[-]
                        (v10) -- (v54) -- (v63) -- (v30) -- (v03) -- (v14) -- (v50) -- (v61) -- (v34) -- (v01) -- (v10)
                ;

                \draw[->, >=stealth, dashed, color=teal] (2.5,-0.5)
                -- (7,4);
                \draw[->, >=stealth, dashed, color=teal] (3.5,-0.5)
                -- (-1,4);
                \draw (7,4) node [above, color=teal] {$i$};
                \draw (-1,4) node [above, color=teal] {$j$};
                \draw[dotted, color=purple] (-0.5,0) -- (6.5,0);
                \draw[dotted, color=purple] (-0.5,1) -- (6.5,1);
                \draw (-0.5,0) node [left, color=purple] {$\tilde{z}_0$};
                \draw (-0.5,1) node [left, color=purple] {$\tilde{z}_1$};

                \foreach \i in {1, 3, 5} {
                        \node[wvert] (w\i) at (\i,0) {};
                }
                \node (l) at (3,4.5) {$z$ on $\Z^2$};
        \end{tikzpicture}\hspace{10mm}
        \begin{tikzpicture}[scale=0.6,baseline={(current bounding box.center)}]
                \useasboundingbox (0,-2) rectangle (6,4.5);
                \foreach \i in {0, 1, 2, 3, 4, 5, 6} {
                        \foreach \j in {0, 1, 2, 3, 4} {
                                \coordinate (v\i\j) at (\i,\j);
                        }
                }
                \draw[gray!60,-]
                        (v10) -- (v54) -- (v63) -- (v30) -- (v03) -- (v14) -- (v50) -- (v61) -- (v34) -- (v01) -- (v10)
                ;
                \draw[-]
                        (v10) -- (v50) -- (v54) -- (v14) -- (v10)
                        (v30) -- (v34) (v12) -- (v52)
                ;
                \foreach \i in {1, 3, 5} {
                        \node[wvert] (w\i) at (\i,0) {};
                }
                \node (l) at (3,4.5) {$p$ on $\Z^2_+$};
                 \draw[dotted, color=purple] (-0.5,0) -- (6.5,0);
                                \draw (-0.5,0) node [left, color=purple] {$p_0$};
        \end{tikzpicture}\hspace{10mm}
        \begin{tikzpicture}[scale=0.6,baseline={(current bounding box.center)}]
                \useasboundingbox (0,-2) rectangle (6,4.5);
                \foreach \i in {0, 1, 2, 3, 4, 5, 6} {
                        \foreach \j in {0, 1, 2, 3, 4} {
                                \coordinate (v\i\j) at (\i,\j);
                        }
                }
                \draw[gray!60,-]
                        (v10) -- (v54) -- (v63) -- (v30) -- (v03) -- (v14) -- (v50) -- (v61) -- (v34) -- (v01) -- (v10)
                ;
                \draw[-]
                        (v01) -- (v61) -- (v63) -- (v03) -- (v01)
                        (2,-1) -- (v23) (4,-1) -- (v43) (0,-1) -- (0,1) (6,-1) -- (6,1)
                ;
                \foreach \i in {0,2,4,6} {
                        \node[wvert] (w\i) at (\i,-1) {};
                }
                \node (l) at (3,4.5) {$r$ on $\Z^2_-$};
                \draw[dotted, color=purple] (-0.5,-1) -- (6.5,-1);
                \draw[dotted, color=purple] (-0.5,1) -- (6.5,1);
                \draw (-0.5,1) node [left, color=purple] {$r_0$};
                \draw (-0.5,-1) node [left, color=purple] {$r_{-1}$};
        \end{tikzpicture}
        \caption{The copies of $\Z^2$, where the discrete holomorphic function
          $z$, as well as the P-nets $p,q,r$ are defined as used
          in the proofs of Theorems \ref{theo:dholsingularity} and
          \ref{theo:dholpremature}. The points that are mapped
          to the initial singularity are drawn as $\circ$.
          The initial axes labeling $z$ as in
            Definition~\ref{def:dhol} are given in the left-most
            figure.}
        \label{fig:dholnetexplanation}
\end{figure}

In the even case, it is also possible to identify a case in which the singularity appears a step earlier.
\begin{theorem}\label{theo:dholpremature}
Let $m\in 2\N+2$, and let $\tilde z$ be an $m$-closed discrete holomorphic function such that $\tilde z_{i,0} = 0$ for all $i\in 2\Z$, and such that
    \begin{align}
                \sum_{i=0}^{m-1}(-1)^i \frac1{\tilde z_{2i+1,1}} = 0.
        \end{align}
        Assume we can apply the propagation map $T$ to $(\tilde z_0, \tilde z_1)$ at least $2m-4$ times. Then, for all $i\in 2\Z+1$
        \begin{align}
                \tilde z_{i,2m-3} = \left(\frac{1}{m}\sum_{\ell=0}^{m-1} \frac1{\tilde z_{2\ell+1,1}} \right)^{-1}.
        \end{align}
        In other words, $\tilde z_{2m-3}$ is constant with value equal to the harmonic mean of $\tilde z_1$.
\end{theorem}
\proof{
Let $r$ be the P-net as in the proof of Theorem \ref{theo:dholsingularity}. By assumption of the theorem, $r$ satisfies the assumptions of Theorem \ref{theo:pnetpremature}, which proves the claim.\qed
}

\begin{remark}\label{rem:disosing}
        There is a generalization of discrete holomorphic functions \cite[Definition 6]{bpdisosurfaces} to $\R^3$, for which simulations show analogous singularity behaviour. Call a map $\xi: \Z^2 \rightarrow \R^3$ \emph{discrete isothermic surface}, if the image of the vertices of each quad is contained in a circle, and if the cross-ratio of the four points on each circle is $-1$. We conjecture that singularities repeat for discrete isothermic surfaces as they do for discrete holomorphic functions, in the sense of Theorem \ref{theo:dholsingularity}. However, we do not currently have the tools to approach this problem. It is possible that the problem for discrete isothermic surfaces can be reduced to the problem for discrete holomorphic functions, for example via a suitable discrete Weierstraß representation. Another possibility is that one identifies $\R^3$ with the set of quaternions with vanishing real part, and finds a generalization of our general results to the non-commutative setting.
\end{remark}

\subsection{Orthogonal circle patterns}\label{sec:orthogonal_circle_patterns}

\begin{definition}
  A \emph{(square grid) orthogonal circle pattern} $p:\Z^2 \rightarrow \hC$ is a circle pattern, see Definition \ref{def:cp}, such that adjacent circles intersect orthogonally. As for general circle patterns, we denote by $c:\Z^2\rightarrow \{\mbox{Circles of } \CP^1\}$ the circles and by $t:\Z^2\rightarrow \C$ the circle centers. An example is provided in Figure~\ref{fig:ocpaztec}.
\end{definition}

If we only look at every second circle of an orthogonal circle pattern, we recover a \emph{circle packing}, sometimes called \emph{Schramm circle packing} \cite{schramm}.

\begin{figure}[tb]
    \centering
    \begin{tikzpicture}[baseline={([yshift=0ex]current bounding box.center)}]
        \draw (1,0.5) circle (1.12);
        \draw (3,0.5) circle (1.12);
        \draw (1,1.5) circle (1.12);
        \draw (3,1.5) circle (1.12);
        \node[wvert,label=left:$p_{0,0}$] (p00) at (0,0) {};
        \node[wvert,label=left:$p_{1,0}$] (p10) at (2,0) {};
        \node[wvert,label=left:$p_{2,0}$] (p20) at (4,0) {};
        \node[wvert,label=left:$p_{0,1}$] (p01) at (0,1) {};
        \node[wvert,label=left:$p_{1,1}$] (p11) at (2,1) {};
        \node[wvert,label=left:$p_{2,1}$] (p21) at (4,1) {};
        \node[wvert,label=left:$p_{0,2}$] (p02) at (0,2) {};
        \node[wvert,label=left:$p_{1,2}$] (p12) at (2,2) {};
        \node[wvert,label=left:$p_{2,2}$] (p22) at (4,2) {};
        \node[bvert,label={[shift={(0.2,-0.3)}]left:$t_{0,0}$}] (t00) at (1,0.5) {};
        \node[bvert,label={[shift={(0.2,-0.3)}]left:$t_{1,0}$}] (t10) at (3,0.5) {};
        \node[bvert,label={[shift={(0.2,-0.3)}]left:$t_{0,1}$}] (t01) at (1,1.5) {};
        \node[bvert,label={[shift={(0.2,-0.3)}]left:$t_{1,1}$}] (t11) at (3,1.5) {};        
    \end{tikzpicture}\hspace{2mm}
    \begin{tikzpicture}[baseline={([yshift=0ex]current bounding box.center)}]
        \draw (1,0.5) circle (1.12);
        \draw (3,0.5) circle (1.12);
        \draw (1,1.5) circle (1.12);
        \draw (3,1.5) circle (1.12);
        \node[wvert,label=left:$\tilde z_{0,0}$] (p00) at (0,0) {};
        \node[wvert,label=left:$\tilde z_{2,0}$] (p10) at (2,0) {};
        \node[wvert,label=left:$\tilde z_{4,0}$] (p20) at (4,0) {};
        \node[wvert,label=left:$\tilde z_{0,2}$] (p01) at (0,1) {};
        \node[wvert,label=left:$\tilde z_{2,2}$] (p11) at (2,1) {};
        \node[wvert,label=left:$\tilde z_{4,2}$] (p21) at (4,1) {};
        \node[wvert,label=left:$\tilde z_{0,4}$] (p02) at (0,2) {};
        \node[wvert,label=left:$\tilde z_{2,4}$] (p12) at (2,2) {};
        \node[wvert,label=left:$\tilde z_{4,4}$] (p22) at (4,2) {};
        \node[bvert,label={[shift={(0.2,-0.3)}]left:$\tilde z_{1,1}$}] (t00) at (1,0.5) {};
        \node[bvert,label={[shift={(0.2,-0.3)}]left:$\tilde z_{3,1}$}] (t10) at (3,0.5) {};
        \node[bvert,label={[shift={(0.2,-0.3)}]left:$\tilde z_{1,3}$}] (t01) at (1,1.5) {};
        \node[bvert,label={[shift={(0.2,-0.3)}]left:$\tilde z_{3,3}$}] (t11) at (3,1.5) {};        
    \end{tikzpicture}\hspace{2mm}
    \begin{tikzpicture}[baseline={([yshift=0ex]current bounding box.center)}]
        \draw (1,0.5) circle (1.12);
        \draw (3,0.5) circle (1.12);
        \draw (1,1.5) circle (1.12);
        \draw (3,1.5) circle (1.12);
        \node[wvert,label=left:$z_{0,0}$] (p00) at (0,0) {};
        \node[wvert,label=left:$z_{1,-1}$] (p10) at (2,0) {};
        \node[wvert,label=left:$z_{2,-2}$] (p20) at (4,0) {};
        \node[wvert,label=left:$z_{1,1}$] (p01) at (0,1) {};
        \node[wvert,label=left:$z_{2,0}$] (p11) at (2,1) {};
        \node[wvert,label=left:$z_{3,-1}$] (p21) at (4,1) {};
        \node[wvert,label=left:$z_{2,2}$] (p02) at (0,2) {};
        \node[wvert,label=left:$z_{3,1}$] (p12) at (2,2) {};
        \node[wvert,label=left:$z_{4,0}$] (p22) at (4,2) {};
        \node[bvert,label={[shift={(0.2,-0.3)}]left:$z_{1,0}$}] (t00) at (1,0.5) {};
        \node[bvert,label={[shift={(0.2,-0.3)}]left:$z_{2,-1}$}] (t10) at (3,0.5) {};
        \node[bvert,label={[shift={(0.2,-0.3)}]left:$z_{2,1}$}] (t01) at (1,1.5) {};
        \node[bvert,label={[shift={(0.2,-0.3)}]left:$z_{3,0}$}] (t11) at (3,1.5) {};        
    \end{tikzpicture}
	\caption{The labelings of the intersection points (white) and circle centers (black) in an orthogonal circle pattern.}
    \label{fig:ocpaztec}    
\end{figure}
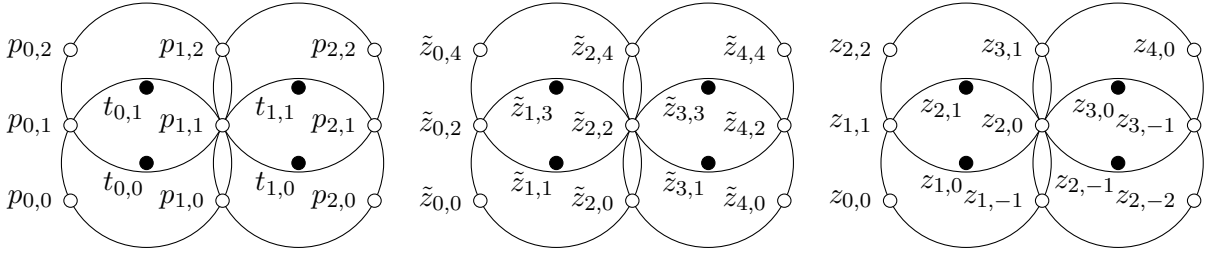

Orthogonal circle patterns provide a classical example of discrete holomorphic functions. Indeed, consider an orthogonal circle pattern $p:\Z^2 \rightarrow\hC$, and its circle centers $t:\Z^2\rightarrow\hat{\C}$. From $p$ and $q:=t$ construct a map $\tilde{z}$ using Equation~\eqref{eq:dholpnet}. Then, we have

\begin{lemma}\emph{\cite[Equation (8.1)]{ddgbook}}
Consider an orthogonal circle pattern $p$, the corresponding map $\tilde{z}$ constructed from $p$ and $t$ as above, and let $z:\Z^2\rightarrow\hat{\C}$ be the rotated version of $\tilde{z}$. Then, $z$ is a discrete holomorphic function.
\end{lemma}

As a consequence of Lemma~\ref{lem:dholpnet}, we obtain the following. Note that this result was independently obtained in the context of orthogonal circle patterns by~\cite{ksclifford} for the intersection points, and by~\cite{klrr} for the circle centers.

\begin{lemma}\label{lem:ortho_Pnets}
Let $p:\Z^2 \rightarrow \hC$ be an orthogonal circle pattern. Then both
the intersection points $p$ and the circle centers $t$ are P-nets.
\end{lemma}

Putting together Lemma~\ref{lem:ortho_Pnets} and Theorem~\ref{theo:explpnet} yields an explicit expression for the points $p_{i,j}$, $i\in\Z,j\geq 1$ as a function of $(p_{i,0})_{i\in\Z}$, $(p_{i,1})_{i\in\Z}$. By exchanging the role of $p$ and $t$, a similar expression can be obtained for $t_{i,j}$.
\begin{corollary}
Let $p:\Z^2\rightarrow \hC$ be an orthogonal circle pattern, and consider the graph $\Z^2$ with face weights $(a_{i,j})=(p_{i,[i+j]_2})_{(i,j)\in\Z^2}$. Then, for all $i\in\Z,j\geq 1$, we have
\[
p_{i,j}=Y(A_{j-1}[p_{i,[j]_2}],a).
\]
\end{corollary}

Now, applying Theorem~\ref{theo:dholsingularity} and its proof in the case where $m$ is even yields the following singularity result, see Figure~\ref{fig:ocpsing} for an example.
\begin{corollary}\label{cor:ocpsing}
Let $m\geq 1$, and let $p:\Z^2\rightarrow \hC$ be an $m$-closed orthogonal circle pattern, such that $p_{i,0}=0$ for all $i\in 2\Z$. Assume we can apply the propagation map $T$ to $(\tilde{z}_0,\tilde{z}_1)$ at least $m-2$ times. Then,
for all $i\in \Z$,
\[
p_{i,m-1}=\Bigl(\frac{1}{m}\sum_{\ell=0}^{m-1} \frac{1}{p_{i,1}}\Bigr)^{-1}=
\Bigl(\frac{1}{m}\sum_{\ell=0}^{m-1} \frac{1}{t_{i,0}}\Bigr)^{-1}.
\]
\end{corollary}

That is, the point of the singularity is the harmonic mean of the first row of the circle pattern, and at the same time the harmonic mean of the first row of circle centers.

\begin{figure}[tb]
        \centering
        \frame{\includegraphics[height=5cm]{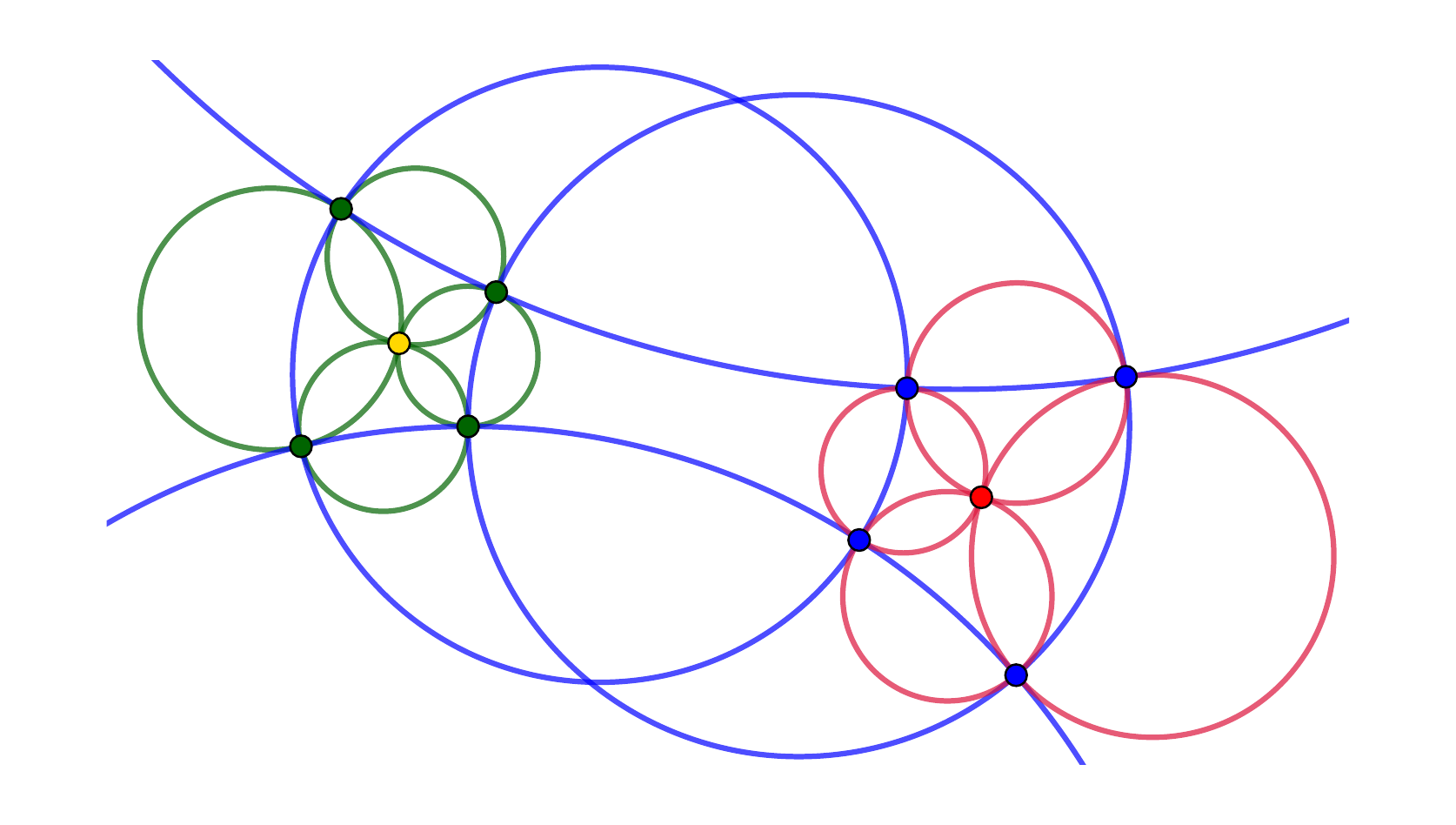}}
        \caption{Devron property in an $m$-closed orthogonal circle pattern for $m=4$. The initial $(m,2)$-Devron singularity is at the yellow point and the final singularity after 2 iterations at the red point.}
        \label{fig:ocpsing}
\end{figure}

\section{Polygon recutting}\label{sec:recut}

Polygon recutting was introduced by Adler \cite{adlerrecutting}, as an integrable dynamical system acting on polygons. Its integrable properties have been studied in a number of papers, see the introduction of \cite{izosimovpolyrecut} and references therein. In this section we explain how polygon recutting arises as a special case of integrable cross-ratio maps, which enables us to provide an explicit expression for the iteration of polygon recutting. We are also able to show that the Devron phenomenon for polygon recutting is a case of a $(m,1)$-Devron singularity from Definition~\ref{def:sing}, and thus follows as a consequence of our general results.

\subsection{Definitions and explicit solution}

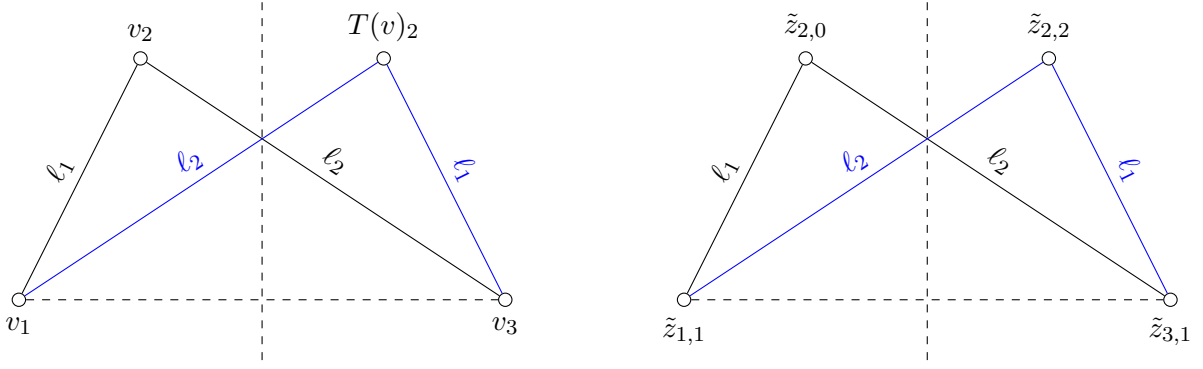
\begin{figure}[tb]
	\centering
	\begin{tikzpicture}[baseline={([yshift=0ex]current bounding box.center)},scale=1.6]
		\node[wvert,label=below:$v_{1}$] (v0) at (0,0) {};
		\node[wvert,label=above:$v_{2}$] (v1) at (1,2) {};
		\node[wvert,label=below:$v_{3}$] (v2) at (4,0) {};
		\node[wvert,label=above:$T(v)_{2}$] (tv1) at (3,2) {};
		
		\draw[-]
			(v0) to node [midway, above, sloped] {$\ell_1$}  (v1) to node [midway, above, sloped] {$\ell_2$}  (v2)					
		;
		\draw[-, blue]
			(v0)  to node [midway, above, sloped] {$\ell_2$}  (tv1)  to node [midway, above, sloped] {$\ell_1$} (v2)					
		;
		\draw[dashed]
			(v0) -- (v2)
			(2,-0.5) -- (2, 2.5)
		;
	\end{tikzpicture}\hspace{1.5cm}
	\begin{tikzpicture}[baseline={([yshift=0ex]current bounding box.center)},scale=1.6]
		\node[wvert,label=below:$\tilde z_{1,1}$] (v0) at (0,0) {};
		\node[wvert,label=above:$\tilde z_{2,0}$] (v1) at (1,2) {};
		\node[wvert,label=below:$\tilde z_{3,1}$] (v2) at (4,0) {};
		\node[wvert,label=above:$\tilde z_{2,2}$] (tv1) at (3,2) {};
		
		\draw[-]
			(v0) to node [midway, above, sloped] {$\ell_1$}  (v1) to node [midway, above, sloped] {$\ell_2$}  (v2)					
		;
		\draw[-, blue]
			(v0)  to node [midway, above, sloped] {$\ell_2$}  (tv1)  to node [midway, above, sloped] {$\ell_1$} (v2)					
		;
		\draw[dashed]
			(v0) -- (v2)
			(2,-0.5) -- (2, 2.5)
		;
	\end{tikzpicture}
	\caption{ An illustration of the effect of polygon recutting at index $k=2$, before (black) and after (blue), together with the two different labelings.}
	\label{fig:polyrecutting}
\end{figure}

\begin{definition}
        Consider points in the complex plane $(v_i)_{i\in\Z}$. A step of a \emph{polygon recutting} consists of fixing an index $k$, and reflecting the vertex $v_k$ with respect to the perpendicular bisector of the segment $[v_{k-1}v_{k+1}]$ joining its neighbors; note that this is an involution.
\end{definition}

Let $\tilde{z}_j=(\tilde{z}_{i,j})_{[i+j]_2=0}$ be such that $\tilde z_{i,j}$ is the reflection of $\tilde z_{i,j-2}$ about the perpendicular bisector of the segment $[\tilde z_{i-1,j-1}\tilde z_{i+1,j-1}]$ for all $i,j\in \Z$ with $i+j\in 2\Z$. We refer to $\tilde{z}$ as a \emph{polygon recutting lattice map}.

Consider the dynamics $T$ mapping $(\tilde{z}_0,\tilde{z}_1)$ to $(\tilde{z}_1,\tilde{z}_2)$. This map
is such that, for every $i\in2\Z$, $\tilde{z}_{i,2}$ is the reflection of $\tilde{z}_{i,0}$ with respect to the perpendicular bisector of $[\tilde{z}_{i-1,1}\tilde{z}_{i+1,1}]$, for every $i\in2\Z$.
Note that we can apply polygon recutting dynamics to points $(v_i)_{i\in\Z}$ by identifying $v$ with $\tilde z$ via
\begin{equation*}
v_i = \begin{cases}
\tilde{z}_{i,0} \text{ if $i$ is even},\\
\tilde{z}_{i,1} \text{ if $i$ is odd}.
\end{cases}
\end{equation*}
The map $T$ naturally extends to $T:\hat{\C}^{\Z_2\times\Z}\rightarrow\hat{\C}^{\Z_2\times\Z}$, mapping $(\tilde{z}_{j-1},\tilde{z}_{j})\mapsto (\tilde{z}_{j},\tilde{z}_{j+1})$, and is referred to as the \emph{polygon recutting dynamics}.

We now consider $z:\Z^2\rightarrow \hat{\C}$ obtained from $\tilde{z}$ by the change of coordinates of Equation~\eqref{eq:rotated_weights}, which we recall for convenience
\[
z_{i,j}=\tilde{z}_{i-j,i+j};
\]
we also refer to $z$ as a \emph{polygon recutting lattice map}.
For every $i\in\Z$, define
\begin{align*}
\ell_{2i}&:=|v_{2i}-v_{2i+1}|=|\tilde{z}_{2i,0}-\tilde{z}_{2i+1,1}|,\\
\ell_{2i-1}&:=|v_{2i-1}-v_{2i}|=|\tilde{z}_{2i-1,1}-\tilde{z}_{2i,0}|.
\end{align*}
Then, we have the following.

\begin{lemma}\label{lem:icr_recutting}
The map $z:\Z^2\rightarrow \hat{\C}$ is an integrable cross-ratio map  with edge-labels
\[
\forall\ i,j\in\Z,\quad \alpha_i=(\ell_{2i})^2,\quad \beta_j=(\ell_{-2j-1})^2.
\]
\end{lemma}
\begin{proof}
By definition of $z$, we have that, for all $i,j\in\Z$, $z_{i+1,j+1}$ is the reflection of $z_{i,j}$ with respect to the perpendicular bisector $M$ of the segment $[z_{i+1,j}z_{i,j+1}]$; therefore the four points are symmetric with respect to reflection about $M$ and thus on a common circle $C$, which shows that the cross-ratio $\cro(z_{i,j},z_{i+1,j},z_{i+1,j+1},z_{i,j+1})$ is real. The vertices $z_{i+1,j},z_{i,j+1}$ define two arcs and by construction, $z_{i+1,j+1}$ and $z_{i,j}$ are on the same one, implying that the cross-ratio is real positive. Then, using symmetries we deduce that the cross-ratio is equal to
        \begin{align*}
        \cro(z_{i,j},z_{i+1,j},z_{i+1,j+1},z_{i,j+1}) = \frac{|z_{i,j}-z_{i+1,j}|^2}{|{z_{i+1,j}}-z_{i+1,j+1}|^2}.
        \end{align*}
Using symmetries again we have that, for every $j\in\Z$, $|z_{i,j}-z_{i+1,j}|=|z_{i,j+1}-z_{i+1,j+1}|$. Iterating this along the column corresponding to $i,i+1$, we deduce that for all $j$,
\[
|z_{i,j}-z_{i+1,j}|=|z_{i,-i}-z_{i+1,-i}|=|\tilde{z}_{2i,0}-\tilde{z}_{2i+1,1}|=|v_{2i}-v_{2i+1}|=\ell_{2i},
\]
where we used the relation between $z$ and $\tilde{z}$, and the definition of $\tilde{z}_0,\tilde{z}_1$. In a similar way, iterating along rows we obtain for all $i$,
\[
|z_{i,j}-z_{i,j+1}|=|z_{-j,j}-z_{-j,j+1}|=|\tilde{z}_{-2j,0}-\tilde{z}_{-2j-1,1}|=|v_{-2j}-v_{-2j-1}|=\ell_{-2j-1},
\]
allowing to conclude the proof.\qedhere
\end{proof}

As a consequence of Lemma~\ref{lem:icr_recutting} and Theorem~\ref{theo:explintcr}, we immediately obtain the following explicit expression for $z_{i,j}$, when $i+j\geq 1$.
\begin{corollary}
  \label{cor:polyrec_dskp}
Let $z:\Z^2\rightarrow \hC$ be a polygon recutting lattice map. Let $w:\Z^2\rightarrow \hC$ be an integrable cross-ratio map such that $z,w$ is a Bäcklund pair. Then, for all $(i,j)\in\Z^2$ such that $i+j\geq 1$, $z_{i,j}$ is equal to the ratio function of oriented dimers of an Aztec diamond with face weights a subset of $(z_{i,j})_{i+j\in\{0,1\}},(w_{i,j})_{i+j\in\{0,1\}}$ as described in Theorem~\ref{theo:explintcr}.
\end{corollary}

\begin{remark}
	If $w$ in Corollary \ref{cor:polyrec_dskp} is such that $z,w$ is a Bäcklund pair with edge-label $\gamma = |z_i -w_i|^2$, then $w$ is also called a \emph{discrete bicycle transformation} of $z$. Moreover, it was shown that the discrete bicycle transformation commutes with polygon recutting \cite{ttbicycle}. This follows also as a corollary from the fact that $z,w$ are a Bäcklund pair of integrable cross-ratio maps.
\end{remark}

\begin{remark}
    Izosimov introduced a quiver of a cluster algebra which he associated to polygon recutting \cite{izosimovpolyrecut}. We do not need cluster algebras in this paper, but each local occurrence of the dSKP equation corresponds to a mutation in a cluster algebra \cite{athesis}. The weights of the Aztec diamond that we use to express propagation of integrable cross-ratio maps (and thus of polygon recutting) are in fact periodic, that is they are well defined on $\Z^2/\{(2,-2)\}$. In the case of $m$-closed initial conditions the weights are doubly periodic, that is well defined on $\Z^2/\{(2,-2), (m,m)\}$. The quiver that appears in \cite{izosimovpolyrecut} has the combinatorics of $\Z^2/\{(2,-2), (m,m)\}$ as well.
\end{remark}

\subsection{Singularities}

We now consider singularities, that is we assume that, for every $i\in\Z$, $z_{i,-i}=0$. In this case, the first two diagonals of any Bäcklund partner can be explicitly computed as stated by the following.

\begin{lemma}\label{lem:Bäck_partner_recutting}
Consider a polygon recutting lattice map $z:\Z^2\rightarrow\hC$, and suppose that for all $i\in\Z$, $z_{i,-i}=0$. Set $\gamma=1$, and for given $i'\in\Z$ and $x\in \hC\setminus \{0\}$, set $w_{i',-i'}=x$. Then, the first two diagonals of the Bäcklund partner $w$ of $z$ are explicitly given by
\begin{equation*}
\forall \, i\in\Z, \quad w_{i,-i}=x,\quad w_{i+1,-i}=\frac{\ell_{2i}^2-1}{x\ell_{2i}^2 - z_{i+1,-i}}xz_{i+1,-i}.
\end{equation*}
\end{lemma}
\begin{proof}
We have that, for every $i\in\Z$,
\begin{align*}
|z_{i,-i}-z_{i+1,-i}|&=|\tilde{z}_{2i,0}-\tilde{z}_{2i+1,1}|=\ell_{2i},\\
|z_{i+1,-i}-z_{i+1,-(i+1)}|&=|\tilde{z}_{2i+1,1}-\tilde{z}_{2i+2,0}|=\ell_{2i+1}.
\end{align*}
Since, for all $i\in\Z$, $z_{i,-i}= 0$, both left-hand-sides are equal, therefore we obtain
$\ell_{2i}=\ell_{2i+1}$. Using that, for every $i\in\Z$, $\ell_{2i}=\ell_{2i+1}$, an explicit computation proves that Equation~\eqref{eq:Bäck_1} is satisfied for $j=-i$, and Equation~\eqref{eq:Bäck_2} is satisfied for $j=-(i+1)$. The proof is concluded by recalling that these two sets of equations determine the first two diagonals of the Bäcklund partner $w$ of $z$.\qedhere
\end{proof}

Let $m\geq 1$. A polygon recutting lattice map $z:\Z^2\rightarrow\hC$ is said to be \emph{$m$-closed} if it is $m$-closed as an integrable cross-ratio map, that is, for all $i,j\in\Z$, $z_{i,j}=z_{i+m,j-m}$, or equivalently, for all $i,j$ such that $[i+j]_2=0$, $\tilde{z}_{i+2m,j}=\tilde{z}_{i,j}$.

\begin{figure}
        \centering
        \frame{\includegraphics[height=6cm]{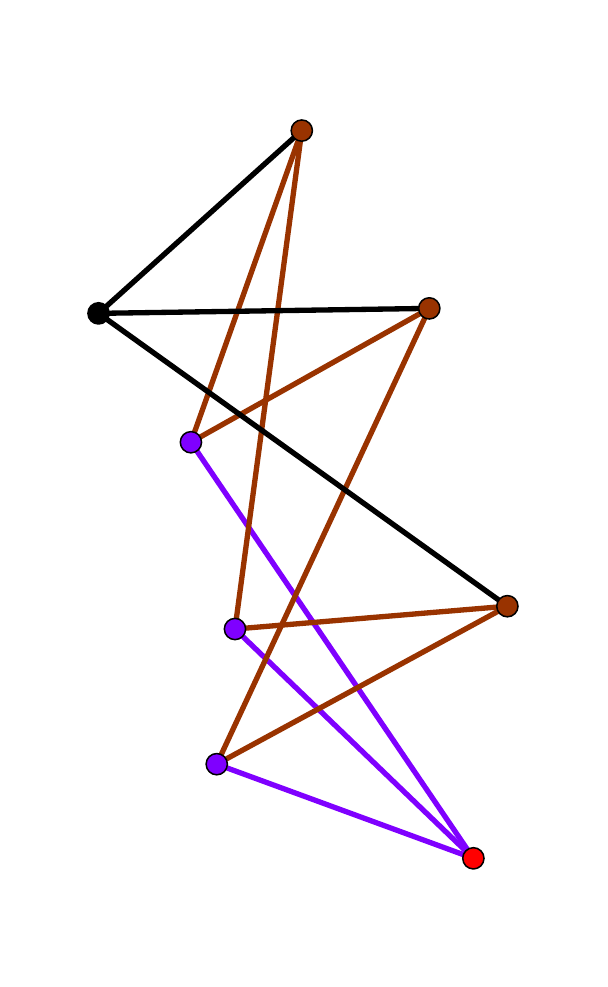}}\hspace{10mm}
        \frame{\includegraphics[height=6cm,angle=0]{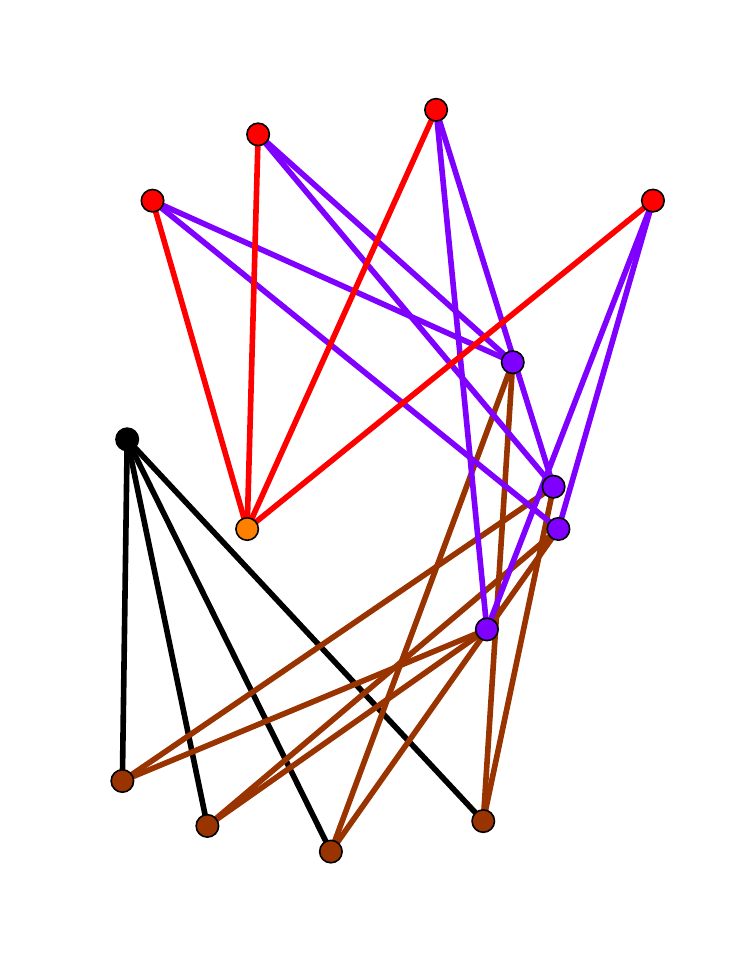}}
        \caption{Polygon recutting singularities. The black dot at $0$
            corresponds to $\tilde{z}_0$, and the brown dots are the values of
          $\tilde{z}_1$. Those are $m$-closed with $m=3$ (left), resp. $m=4$
        (right). Note that $\tilde{z}_m$ is constant.}
        \label{fig:polyrecutsing}
\end{figure}

The following Devron phenomenon has already been observed and proven by Glick \cite[Theorem 7.3]{gdevron}.

\begin{theorem}\label{theo:recutsing}
Let $m\in\N$, and let $\tilde{z}$ be an $m$-closed polygon recutting lattice map such that $\tilde{z}_{i,0}=0$ for all $i\in2\Z$. Assume we can apply the propagation map $T$ to $(\tilde{z}_0,\tilde{z}_1)$ at least $m-1$ times. Then, there is $z'\in\hC$ such that
\[
\tilde{z}_{k,m}=z'
\]
for all $k$ such that $[k+m]_2=0$. That is $\tilde{z}_{m}$ is constant and thus the singularity repeats after $m-1$ steps.
\end{theorem}
\begin{proof}
  By Corollary~\ref{cor:polyrec_dskp}, we know that the values
  of $z$ and a Bäcklund partner $w$ can be made into a solution of the
  dSKP recurrence, with initial values as in
  Figure~\ref{fig:backlund_ic}. When $\tilde{z}$ is $m$-closed and
  $\tilde{z}_0$ is constant, by
  Lemma~\ref{lem:Bäck_partner_recutting}, $\tilde{w}_0$ is constant as
  well. As a result, this initial condition is $(m,1)$-Devron. By
  Theorem~\ref{theo:devron_sing}, for the solution $x$ of the dSKP
  recurrence, whenever $[i-j-m]_2=0$, $x(i,j,m)=x(i+1,j+1,m)$; note
  that the condition $[i-j-m]_2=0$ is in fact automatic as
  $(i,j,m)\in \calL$. Using the expression \eqref{eq:proof_Backlund}
  for the solution, this gives that $\tilde{z}_m$ (as well as
  $\tilde{w}_m$) is constant.
\end{proof}

We provide a conjecture on the exact position of the singularity.

\begin{conjecture}\label{conj:recutsing} 
        Let $m\in \Z$ and let $\tilde z$ be an $m$-closed polygon recutting lattice map, such that $\tilde z_{i,0} = 0$ for all $i\in 2\Z$. Let $y_i = \tilde z_{2i+1,1}$ and $\varrho_i= \ell_{2i} = \ell_{2i+1}$. Let $\mathcal I_k$ be the set of $k$-subsets of $\{0,1,\dots,m-1\}$. Assume we can apply polygon recutting $T$ to $\tilde z$ at least $m-1$ times. For $m \in 2\Z+1$ and $k=\frac{m-1}{2}$ holds
        \begin{align}
        \tilde z_{i,m}= (-1)^k\frac{\sum\limits_{I \in \mathcal I_k }\left[\prod\limits_{i\in I}(-1)^i \prod\limits_{i\notin I}y_i \prod\limits_{i,j\in I, i < j} (\varrho_i^2-\varrho_j^2) \prod\limits_{i,j\notin I, i < j} (\varrho_i^2-\varrho_j^2)\right]}{\sum\limits_{I \in \mathcal I_k }\left[ \prod\limits_{i\in I}(-1)^i\prod\limits_{i\in I}y_i \prod\limits_{i,j\in I, i < j} (\varrho_i^2-\varrho_j^2) \prod\limits_{i,j\notin I, i < j} (\varrho_i^2-\varrho_j^2)\right]},
        \end{align}
        for all $i\in 2\Z+1$. For $m \in 2\Z$ and $k=\frac{m}{2}$ holds
        \begin{align}
                \tilde z_{i,m} = (-1)^{k+1}\frac{\sum\limits_{I \in \mathcal I_{k-1} }\left[ \prod\limits_{i\in I}(-1)^i\varrho_i^2 \prod\limits_{i\notin I}y_i \prod\limits_{i,j\in I, i < j} (\varrho_i^2-\varrho_j^2) \prod\limits_{i,j\notin I, i < j} (\varrho_i^2-\varrho_j^2)\right]}{\sum\limits_{I \in \mathcal I_k }\left[ \prod\limits_{i\in I}(-1)^i\varrho_i^2 \prod\limits_{i\in I}y_i \prod\limits_{i,j\in I, i < j} (\varrho_i^2-\varrho_j^2) \prod\limits_{i,j\notin I, i < j} (\varrho_i^2-\varrho_j^2)\right]},
        \end{align}
        for all $i\in 2\Z$.
\end{conjecture}

\section{Circle intersection dynamics and integrable circle patterns}\label{sec:cid}
When Glick investigated the Devron property \cite[Section 9]{gdevron}, he also proposed a new dynamics called \emph{circle intersection dynamics}. It can be thought of as a loose generalization of the pentagram map, see Section~\ref{sec:pent}, replacing the process of intersecting lines through pairs of points, by the process of intersecting circles through triplets of points. We show in this section how to relate circle intersection dynamics to integrable cross-ratio maps, which enables us to give an explicit expression for the iteration of circle intersection dynamics and to prove Glick's conjecture on the Devron property.

\subsection{Definitions and explicit solution}

\begin{definition}
Consider points in the complex plane $(v_i)_{i\in\Z}$. For all $i\in\Z$, let $c_i$ denote the circle through $v_{i-1},v_i,v_{i+1}$, and let $t_i$ be the center of $c_i$. A step of \emph{(local) circle intersection dynamics} consists in fixing an index $k$, and replacing the vertex $v_k$ with the other intersection point of $c_{k-1}$ and $c_{k+1}$; note that this is an involution. Observe that only the circle $c_k$, and hence $t_k$ change, and all the other circles do not.
\end{definition}

Let $\tilde{z}_j=(\tilde{z}_{i,j})_{[i+j]_2=0}$, $\tilde{w}_j=(\tilde{w}_{i,j})_{[i+j]_2=0}$ be such that
\begin{align}
\forall\ i\in2\Z+1, \quad\quad \tilde z_{i+1,j+1}, \quad \tilde z_{i-1,j+1}, \quad \tilde z_{i-1,j-1}, \quad \tilde z_{i+1,j-1}, \quad \tilde w_{i,j},
\end{align}
are on a circle with center $\tilde z_{i,j}$, and such that
\begin{align}
\forall i\in2\Z,\quad\quad \tilde w_{i+1,j+1}, \quad \tilde w_{i-1,j+1}, \quad \tilde w_{i-1,j-1}, \quad \tilde w_{i+1,j-1}, \quad \tilde z_{i,j},
\end{align}
are on a circle with center $\tilde w_{i,j}$, see Figure \ref{fig:cidandintcp}.
We refer to $\tilde{z},\tilde{w}$ as a \emph{pair of circle intersection dynamics lattice maps.}

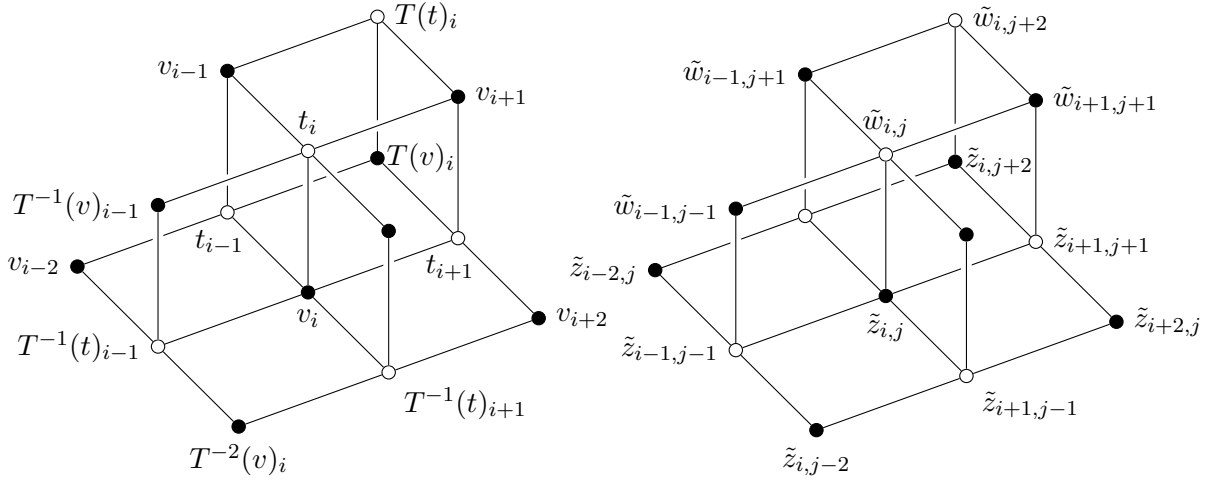
\begin{figure}[tb]
    \centering
    \begin{tikzpicture}[scale=1.5]
                \coordinate (e1) at (20:1.4);
                \coordinate (e2) at (135:1);
                \coordinate (e3) at (270:1.25);
                \node[wvert,label=above:$t_{i}$] (v) at (0,0) {};
                \node[bvert,label=right:$v_{i+1}$] (v1) at (e1) {};
                \node[bvert,label=left:$v_{i-1}$] (v2) at (e2) {};
                \node[bvert,label=below:$v_{i}$] (v3) at (e3) {};
                \node[wvert,label=right:$T(t)_{i}$] (v12) at ($(e1)+(e2)$) {};
                \node[wvert,label={below,xshift=-3:$t_{i-1}$}] (v23) at ($(e3)+(e2)$) {};
                \node[wvert,label={below,xshift=-3:$t_{i+1}$}] (v13) at ($(e1)+(e3)$) {};
                \node[bvert,label={right,xshift=-3:$T(v)_{i}$}] (v123) at ($(e1)+(e2)+(e3)$) {};
                \node[wvert,label=left:$T^{-1}(t)_{i-1}$] (v34) at ($(e3)-(e1)$) {};
                \node[wvert,label=below right:$T^{-1}(t)_{i+1}$] (v35) at ($(e3)-(e2)$) {};
                \node[bvert,label=left:$v_{i-2}$] (v234) at ($(e3)+(e2)-(e1)$) {};
                \node[bvert,label=right:$v_{i+2}$] (v135) at ($(e3)-(e2)+(e1)$) {};
                \node[bvert,label=below:$T^{-2}(v)_{i}$] (v345) at ($(e3)-(e2)-(e1)$) {};
                \node[bvert,label=left:$T^{-1}(v)_{i-1}$] (v4) at ($(0,0)-(e1)$) {};
                \node[bvert] (v5) at ($(0,0)-(e2)$) {};
                \draw[-]
                        (v) edge (v1) edge (v2) edge (v3) edge (v4) edge (v5)
                        (v12) edge (v1) edge (v2) edge (v123)
                        (v23) edge (v3) edge (v2) edge (v123)
                        (v13) edge (v1) edge (v3) edge (v123)
                        (v3) edge (v34) edge (v35)
                        (v23) -- (v234) -- (v34) -- (v345) -- (v35) -- (v135) -- (v13)
                ;
                \draw[-,white,line width=4pt]
                        (v) edge (v3) edge (v1) edge (v4) edge (v5)
                        (v4) -- (v34) (v5) -- (v35)
                ;
                \draw[-]
                        (v) edge (v1) edge (v3) edge (v4) edge (v5)
                        (v4) -- (v34) (v5) -- (v35)
                ;
        \end{tikzpicture}\hspace{-9mm}
        \begin{tikzpicture}[scale=1.5]
            \coordinate (e1) at (20:1.4);
            \coordinate (e2) at (135:1);
            \coordinate (e3) at (270:1.25);
            \node[wvert,label=above:$\tilde w_{i,j}$] (v) at (0,0) {};
            \node[bvert,label=right:$\tilde w_{i+1,j+1}$] (v1) at (e1) {};
            \node[bvert,label=left:$\tilde w_{i-1,j+1}$] (v2) at (e2) {};
            \node[bvert,label=below:$\tilde z_{i,j}$] (v3) at (e3) {};
            \node[wvert,label=right:$\tilde w_{i,j+2}$] (v12) at ($(e1)+(e2)$) {};
            \node[wvert] (v23) at ($(e3)+(e2)$) {}; 
            \node[wvert,label=right:$\tilde z_{i+1,j+1}$] (v13) at ($(e1)+(e3)$) {};
            \node[bvert,label={right,xshift=-3:$\tilde z_{i,j+2}$}] (v123) at ($(e1)+(e2)+(e3)$) {};
                        \node[wvert,label=left:$\tilde z_{i-1,j-1}$] (v34) at ($(e3)-(e1)$) {};
                        \node[wvert,label=below right:$\tilde z_{i+1,j-1}$] (v35) at ($(e3)-(e2)$) {};
                        \node[bvert,label=left:$\tilde z_{i-2,j}$] (v234) at ($(e3)+(e2)-(e1)$) {};
                        \node[bvert,label=right:$\tilde z_{i+2,j}$] (v135) at ($(e3)-(e2)+(e1)$) {};
                        \node[bvert,label=below:$\tilde z_{i,j-2}$] (v345) at ($(e3)-(e2)-(e1)$) {};
                        \node[bvert,label=left:$\tilde w_{i-1,j-1}$] (v4) at ($(0,0)-(e1)$) {};
                        \node[bvert] (v5) at ($(0,0)-(e2)$) {}; 
            \draw[-]
                    (v) edge (v1) edge (v2) edge (v3)
                    (v12) edge (v1) edge (v2) edge (v123)
                    (v23) edge (v3) edge (v2) edge (v123)
                    (v13) edge (v1) edge (v3) edge (v123)
                                (v3) edge (v34) edge (v35)
                                (v23) -- (v234) -- (v34) -- (v345) -- (v35) -- (v135) -- (v13)
            ;
            \draw[-,white,line width=4pt]
                   (v) edge (v3) edge (v1) edge (v4) edge (v5)
                   (v4) -- (v34) (v5) -- (v35)
            ;
            \draw[-]
                  (v) edge (v1) edge (v3) edge (v4) edge (v5)
                  (v4) -- (v34) (v5) -- (v35)
            ;
        \end{tikzpicture}
        \caption{ Here $i\in 2\Z$. Identifying the points of a
          Bäcklund pair of integrable circle patterns (left), with the
          points occurring in circle intersection dynamics
          (right), where we set $j=0$ on the left for improved readability.
          }
        \label{fig:cidandintcp}
\end{figure}

Consider the dynamics $T$ mapping
$(\tilde{z}_0,\tilde{w}_0,\tilde{z}_1,\tilde{w}_1)$ to
$(\tilde{z}_1,\tilde{w}_1,\tilde{z}_2,\tilde{w}_2)$. This map
corresponds to replacing $\tilde{z}_{i,0}$ by $\tilde{z}_{i,2}$, which is the other intersection point of the circles centered at $\tilde z_{i-1,1}$ and $\tilde z_{i+1,1}$, for every $i \in 2\Z$, and replacing the center $\tilde w_{i,0}$ by $\tilde{w}_{i,2}$, which is the center of the circle through $\tilde w_{i-1,1}, \tilde w_{i+1,1}$ and $\tilde z_{i,2}$. Note that we can apply circle intersection dynamics to points $(v_i)_{i\in\Z}$ and centers $(t_i)_{i\in\Z}$ by identifying $t,v$ with $\tilde z, \tilde w$ via
\begin{equation*}
v_i = \begin{cases}
\tilde{z}_{i,0} \text{ if $i$ is even},\\
\tilde{w}_{i,1} \text{ if $i$ is odd},
\end{cases}\quad
{t_i} = \begin{cases}
\tilde{w}_{i,0} \text{ if $i$ is even},\\
\tilde{z}_{i,1} \text{ if $i$ is odd}.
\end{cases}
\end{equation*}

The map $T$ naturally extends to
\[
T:\hat{\C}^{\Z_4\times\Z}\rightarrow\hat{\C}^{\Z_4\times\Z},\quad (\tilde{z}_j,\tilde{w}_j,\tilde{z}_{j+1},\tilde{w}_{j+1})\mapsto(\tilde{z}_{j+1},\tilde{w}_{j+1},\tilde{z}_{j+2},\tilde{w}_{j+2}),
\] and is referred to as the \emph{circle intersection dynamics}.
This also gives a dynamics acting on points. Using the same notation
$T$, and noting that half the points do not change in one application
of $T$, we have
\begin{align*}
\forall\ i\in2\Z,\ j\in2\N+2,\quad & T^j(v)_i = T^{j-1}(v)_i = \tilde
                                     z_{i,j},\quad  T^j(t)_i = T^{j-1}(t)_i = \tilde w_{i,j},\\
\forall\ i\in2\Z+1,\ j\in2\N+1,\quad & T^{j}(v)_i = T^{j-1}(v)_i =
                                       \tilde w_{i,j},\quad  T^j(t)_i
                                       = T^{j-1}(t)_i = \tilde z_{i,j}.
\end{align*}
Whether $T$ is used for the dynamics acting on points or quadruples of points should be clear from the context.

We now consider $z,w:\Z^2\rightarrow \hC$ obtained from $\tilde{z},\tilde{w}$ by the change of coordinates of Equation~\eqref{eq:rotated_weights}.

\begin{lemma}\label{lem:icd_Backlund}
The pair $z,w$ of circle intersection dynamics lattice maps is a Bäcklund pair of integrable cross-ratio maps with $\gamma=1$,
and edge-labels given by
\[
\forall\ i,j\in\Z,\quad \alpha_i = \cro(z_{i,-i},z_{i+1,-i},w_{i+1,-i},w_{i,-i}), \quad
\beta_j = \cro(z_{j,-j},z_{j,1-j},w_{j,1-j},w_{j,-j}).
\]
\end{lemma}
\proof{
    Consider the cube in $\Z_2 \times \Z$ involving the points
    \begin{align}
        z_{i,j}, w_{i,j}, z_{i+1,j}, w_{i+1,j}, z_{i,j+1}, w_{i,j+1}, z_{i+1,j+1}, w_{i+1,j+1},
    \end{align}
    for some fixed $i,j$ with $[i+j]_2=0$, see Figure \ref{fig:cid}. Then the four circles centered at $w_{i,j}, w_{i+1,j+1},$ $z_{i+1,j}, z_{i,j+1}$ have four intersection points, and these are exactly the four points $z_{i,j}, z_{i+1,j+1},$ $w_{i+1,j}, w_{i,j+1}$. Each quad, that is each face of the cube, consists of the centers of two circles and their two intersection points. Assume the factorization property holds with $\gamma=1$, that is
    \begin{align}
        \cro(z_{i,j},z_{i+1,j},w_{i+1,j},w_{i,j}) &= \alpha_i,\\
        \cro(z_{i,j},z_{i,j+1},w_{i,j+1},w_{i,j}) &= \beta_j.
    \end{align}
    The cross-ratio $\xi$ of each quad
    corresponds to the intersection-angle $\theta$ of the two circles
    via $\xi = \exp(2i\theta)$ \cite[Equation (8.1)]{ddgbook}.
    The three circle intersection angles around $z_{i,j}$
    sum to $\pi$, thus
    \begin{align}
        \cro(z_{i,j},z_{i+1,j},w_{i+1,j},w_{i,j})\cro(z_{i,j},w_{i,j},w_{i,j+1},z_{i,j+1})\cro(z_{i,j},z_{i,j+1},z_{i+1,j+1},z_{i+1,j}) = 1.
    \end{align}
    Therefore we have that
    \begin{align}
        \cro(z_{i,j},z_{i+1,j},z_{i+1,j+1},z_{i,j+1}) = \frac{\alpha_i}{\beta_j}.
    \end{align}
     It is well known that in a four-circle configuration, opposite intersection angles sum to $\pi$ \cite{vinberggeometry}. Therefore we obtain that
    \begin{align}
        \cro(z_{i,j+1},z_{i+1,j+1},w_{i+1,j+1},w_{i,j+1}) = \cro(z_{i,j},z_{i+1,j},w_{i+1,j},w_{i,j}) &= \alpha_i,\\
        \cro(z_{i+1,j},z_{i+1,j+1},w_{i+1,j+1},w_{i+1,j}) = \cro(z_{i,j},z_{i,j+1},w_{i,j+1},w_{i,j}) &= \beta_j,\\
        \cro(w_{i,j},w_{i+1,j},w_{i+1,j+1},w_{i,j+1}) = \cro(z_{i,j},z_{i+1,j},z_{i+1,j+1},z_{i,j+1}) &= \frac{\alpha_i}{\beta_j}.
    \end{align}
    Thus in this cube, the equations of Definition \ref{def:intcrmap}
    and Definition \ref{def:backlundpair} are satisfied. The argument
    proceeds analogously in the case where $[i+j]_2=1$, with the interchange $z \leftrightarrow w$. By induction over $i+j$ starting from $i+j=0$, the argument holds for all $i,j\in \Z$.\qed
}

Note that this special case of integrable cross-ratio maps in which half the points are circle centers and the other half intersection points is called \emph{integrable circle patterns} \cite[Section~10]{bmsanalytic}.

\begin{figure}[tb]
    \centering
    \frame{
        \begin{tikzpicture}[scale=1,text=blue]
        \useasboundingbox (-4.2,-3.2) rectangle (3.2,3.3);
        \clip (-4.2,-3) rectangle (3.2,3.3);
        \coordinate (p) at (0,0);
        \coordinate (pm) at (-1.8,-1.1);
        \coordinate (pp) at (2.2,-0.7);
        \coordinate (tp) at (0.2,1.6);
        \draw[circle through 3 points={p}{pm}{pp}];
        \draw[circle through 3 points={p}{tp}{pp}];
        \draw[circle through 3 points={p}{pm}{tp}];
        \draw[circle through 3 points={pp}{pm}{tp}];

        \coordinate[] (hm) at ($(p)!.5!(pm)$);
        \coordinate[] (hp) at ($(p)!.5!(pp)$);
        \coordinate[] (ht) at ($(p)!.5!(tp)$);
        \coordinate[] (am) at ($(hm)!1!90:(p)$);
        \coordinate[] (ap) at ($(hp)!1!90:(p)$);
        \coordinate[] (at) at ($(ht)!1!90:(p)$);

        \coordinate[wvert,label=right:$t_{i}$] (m) at (intersection of hm--am and hp--ap);
        \coordinate[wvert,label=left:$t_{i-1}$] (mm) at (intersection of hm--am and ht--at);
        \coordinate[wvert,label=left:$t_{i+1}$] (mp) at (intersection of ht--at and hp--ap);

        \coordinate[] (htm) at ($(pm)!.5!(tp)$);
                \coordinate[] (htp) at ($(pp)!.5!(tp)$);
        \coordinate[] (atm) at ($(htm)!1!90:(tp)$);
                \coordinate[] (atp) at ($(htp)!1!90:(tp)$);
        \coordinate[wvert,label={[fill=white, inner sep=1]below right:$T(t)_{i}$}] (tm) at (intersection of htm--atm and htp--atp);

        \draw[blue]
                (p) edge (m) edge (mm) edge (mp)
                (tp) -- (mm) -- (pm) -- (m) -- (pp) -- (mp) -- (tp)
                (tm) edge (tp) edge (pp) edge (pm)
        ;

        \coordinate[bvert,label={[fill=white, inner sep=0,xshift=-2]left:$v_i$}] (p) at (p);
        \coordinate[bvert,label={right,yshift=-1:$v_{i+1}$}] (pp) at (pp);
        \coordinate[bvert,label={below left,yshift=4:$v_{i-1}$}] (pm) at (pm);
        \coordinate[bvert,label={[fill=white, inner sep=1]right:$T(v)_i$}] (tp) at (tp);
    \end{tikzpicture}}\hspace{5mm}
          \frame{
            \begin{tikzpicture}[scale=1,text=blue]
                        \useasboundingbox (-4.2,-3.2) rectangle (3.4,3.3);
                \clip (-4.2,-3) rectangle (3.4,3.3);
                \coordinate (p) at (0,0);
                \coordinate (pm) at (-1.8,-1.1);
                \coordinate (pp) at (2.2,-0.7);
                \coordinate (tp) at (0.2,1.6);
                \draw[circle through 3 points={p}{pm}{pp}];
                \draw[circle through 3 points={p}{tp}{pp}];
                \draw[circle through 3 points={p}{pm}{tp}];
                \draw[circle through 3 points={pp}{pm}{tp}];

                \coordinate[] (hm) at ($(p)!.5!(pm)$);
                \coordinate[] (hp) at ($(p)!.5!(pp)$);
                \coordinate[] (ht) at ($(p)!.5!(tp)$);
                \coordinate[] (am) at ($(hm)!1!90:(p)$);
                \coordinate[] (ap) at ($(hp)!1!90:(p)$);
                \coordinate[] (at) at ($(ht)!1!90:(p)$);

                \coordinate[wvert,label=right:$w_{i,j}$] (m) at (intersection of hm--am and hp--ap);
                \coordinate[wvert,label=left:$z_{i,j+1}$] (mm) at (intersection of hm--am and ht--at);
                \coordinate[wvert,label=left:$z_{i+1,j}$] (mp) at (intersection of ht--at and hp--ap);

                \coordinate[] (htm) at ($(pm)!.5!(tp)$);
                \coordinate[] (htp) at ($(pp)!.5!(tp)$);
                \coordinate[] (atm) at ($(htm)!1!90:(tp)$);
                \coordinate[] (atp) at ($(htp)!1!90:(tp)$);
                \coordinate[wvert,label={[fill=white, inner sep=1,yshift=-2,xshift=-2]below right:$w_{i+1,j+1}$}] (tm) at (intersection of htm--atm and htp--atp);

                \draw[blue]
                (p) edge (m) edge (mm) edge (mp)
                (tp) -- (mm) -- (pm) -- (m) -- (pp) -- (mp) -- (tp)
                (tm) edge (tp) edge (pp) edge (pm)
                ;

                \coordinate[bvert,label={[fill=white, inner sep=1]left:$z_{i,j}$}] (p) at (p);
                \coordinate[bvert,label={right,yshift=-1,xshift=-1:$w_{i+1,j}$}] (pp) at (pp);
                \coordinate[bvert,label={below left,yshift=4:$w_{i,j+1}$}] (pm) at (pm);
                \coordinate[bvert,label={[fill=white, inner sep=1]right:$z_{i+1,j+1}$}] (tp) at (tp);
        \end{tikzpicture}}
    \caption{Labeling in terms of $t,v$ (left) and $z,w$ (right) in circle intersection dynamics, where we set $j=0$ on the left for improved readability.
    }
    \label{fig:cid}
\end{figure}
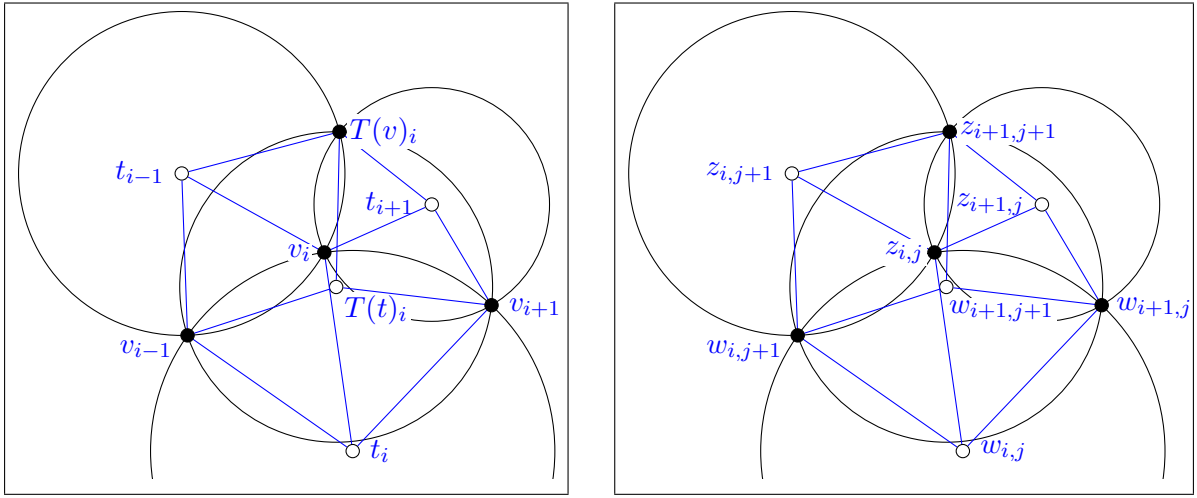

As a consequence of Lemma~\ref{lem:icd_Backlund} and Theorem~\ref{theo:explintcr}, we immediately obtain the following explicit expression for $z_{i,j},w_{i,j}$ when $i+j\geq 1$.

\begin{corollary}
Let $z,w:\Z^2\rightarrow\hC$ be a pair of integrable circle dynamics lattice maps. Then, for all $(i,j)\in\Z^2$ such that $i+j\geq 1$, $z_{i,j}$ (resp. $w_{i,j}$) is equal to the ratio function of oriented dimers of an Aztec diamond with face weights being a subset of $(z_{i,j})_{i+j\in\{0,1\}},(w_{i,j})_{i+j\in\{0,1\}}$ as described in Theorem~\ref{theo:explintcr}.
\end{corollary}

\subsection{Singularities}

To study singularities, we assume that the initial data consits of both the points $v=(v_i)_{i\in\Z}$, as well as the circle centers $t=(t_i)_{i\in\Z}$. Generically, the points determine the centers, but in the case of a singularity this is not necessarily true.

Let $m\geq 1$. The initial data $v,t$ is said to be \emph{$2m$-closed} if, for all $i\in\Z$, $v_{i+2m} = v_i$ and $u_{i+2m} = u_{i}$. This is equivalent to saying that the corresponding Bäcklund pair $z,w$ of integrable cross-ratio maps is $m$-closed.

There is a Devron-$(m,1)$ type singularity in circle intersection dynamics. This singularity has not appeared in the literature before. An example ($2m=8$) is provided in Figure \ref{fig:ciddodgson}.

\begin{figure}[tb]
        \centering
        \frame{\includegraphics[scale=0.93]{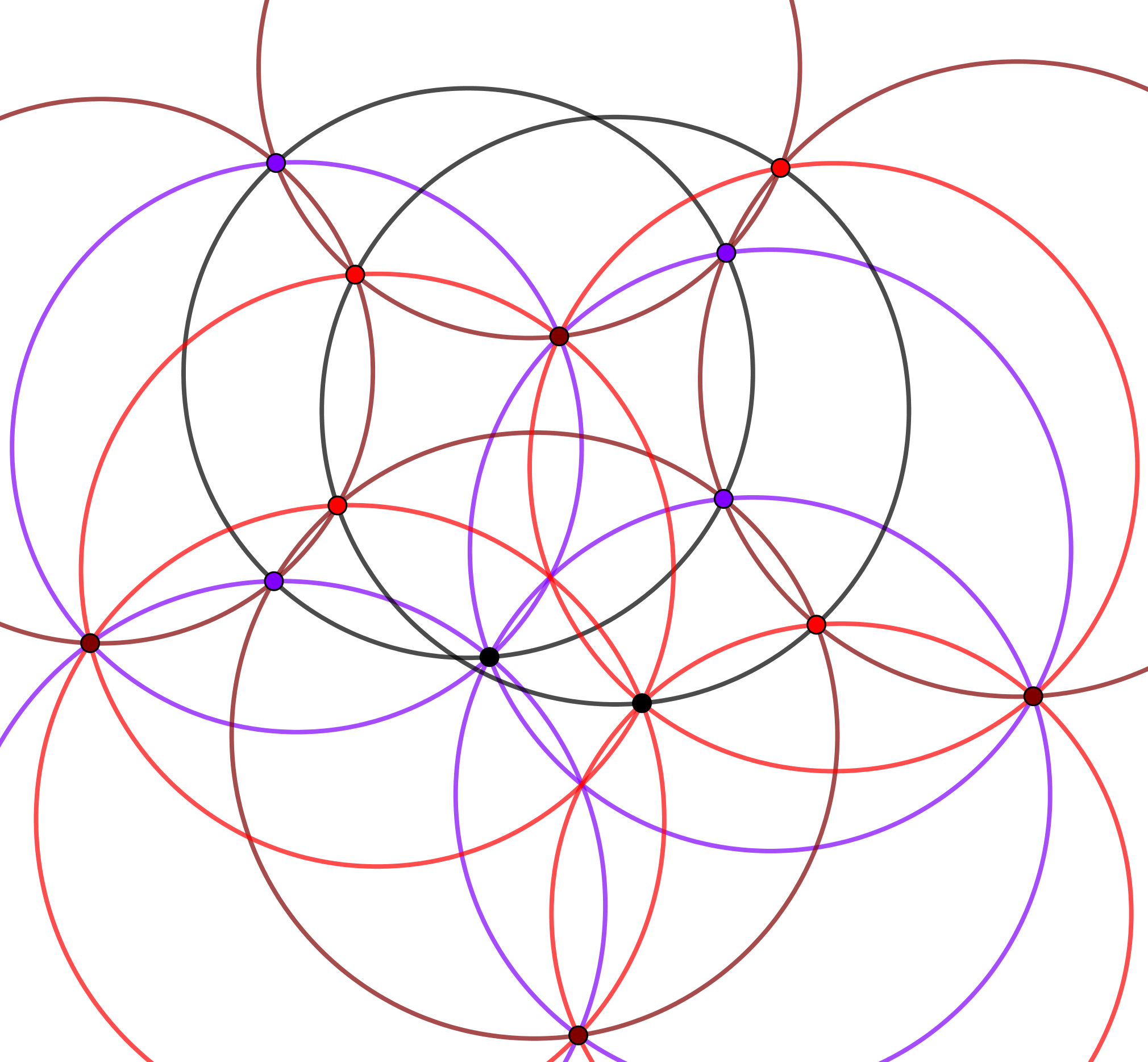}}
        \frame{\includegraphics[scale=0.93]{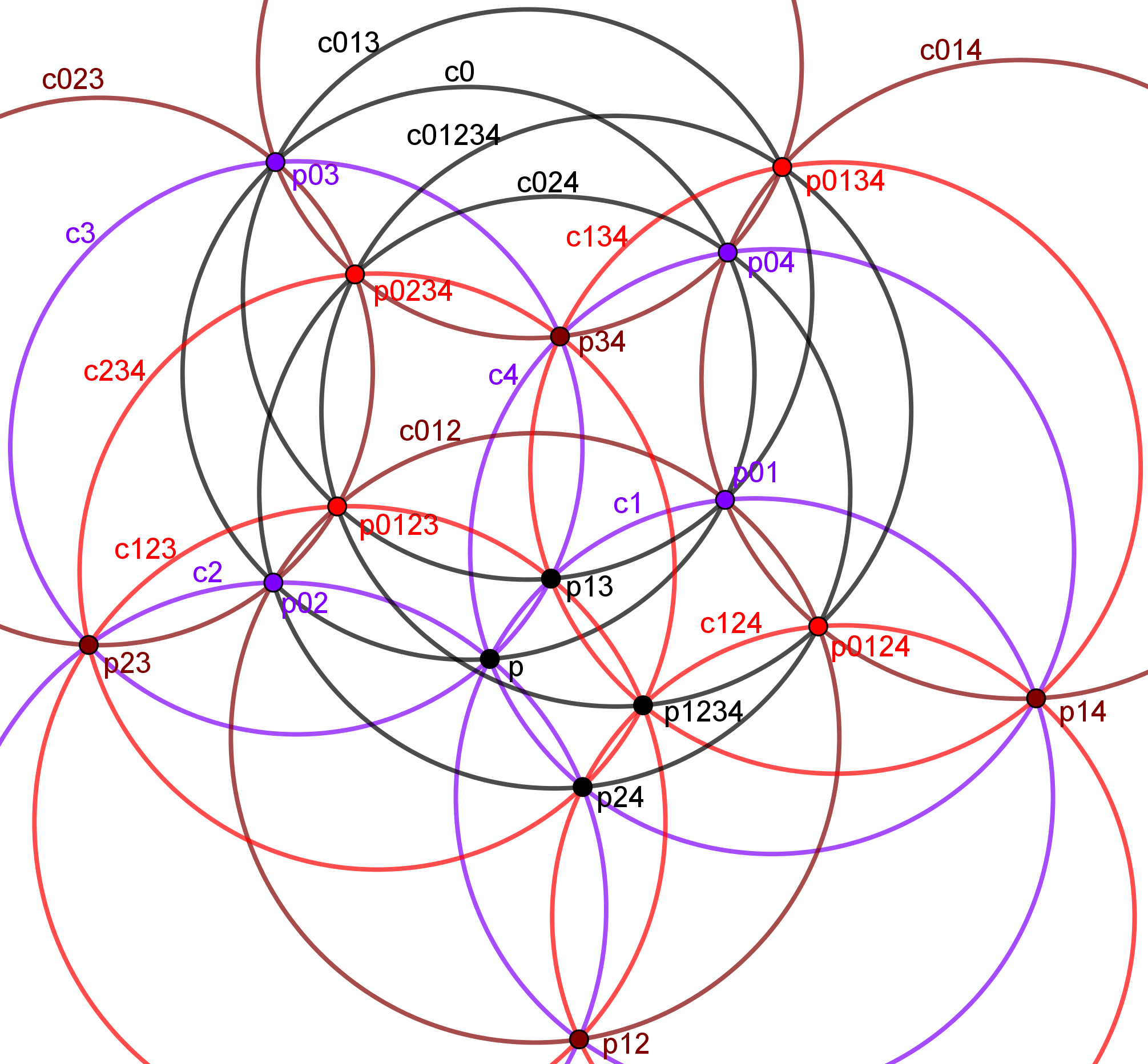}}
        \caption{Devron $(m,1)$-singularity (black to black) in circle intersection dynamics for the 8-closed case, as in Theorem \ref{theo:ciddodgson}. On the right as a subset of the Clifford 5-circle configuration with labels.}
        \label{fig:ciddodgson}
\end{figure}

\begin{theorem}\label{theo:ciddodgson}
        Let $m\geq 3$, $s,s'\in \C$ and let $v,u$ be $2m$-closed initial data such that $v_i = s$ and $t_i = s'$ for all $i\in 2\Z$. Assume we can apply the propagation map $T$ at least $m-1$ times. Then, there are $s'',s'''\in\hC$ such that
        \[
        \tilde z_m \equiv s'', \qquad \tilde w_{m} \equiv s'''.
        \]
       That is the singularity repeats after $m-1$ steps.
\end{theorem}
\proof{
  Consider the Bäcklund pair $z,w$ corresponding to $v,t$. Since
  $v_i=s$ for all $i\in 2\Z$, $\tilde{z}_0$ is constant equal to
  $s$. Similarly, $\tilde{w}_0$ is constant equal to $s'$. Therefore, the
  corresponding initial condition for the dSKP solution $x$ is $(m,1)$-Devron, see
  Figure~\ref{fig:backlund_ic}. By
  Theorem~\ref{theo:devron_sing}, for any $(i,j)$ such that
  $[i+j+m]_2=0$, $x(i,j,m)=x(i+1,j+1,m)$.
 Using the expression~\eqref{eq:proof_Backlund} for the solution, this gives that both $\tilde{z}_m$ and $\tilde{w}_m$ are constant.
  \qed
}

We now prove a Devron-$(m,2)$ type singularity in circle intersection dynamics; an example ($2m=8$) is provided in Figure \ref{fig:cidsing}. This proves a conjecture of Glick \cite[Conjecture 9.1]{gdevron}.

\begin{figure}[tb]
        \centering
        \frame{\includegraphics[scale=0.6,angle=-90]{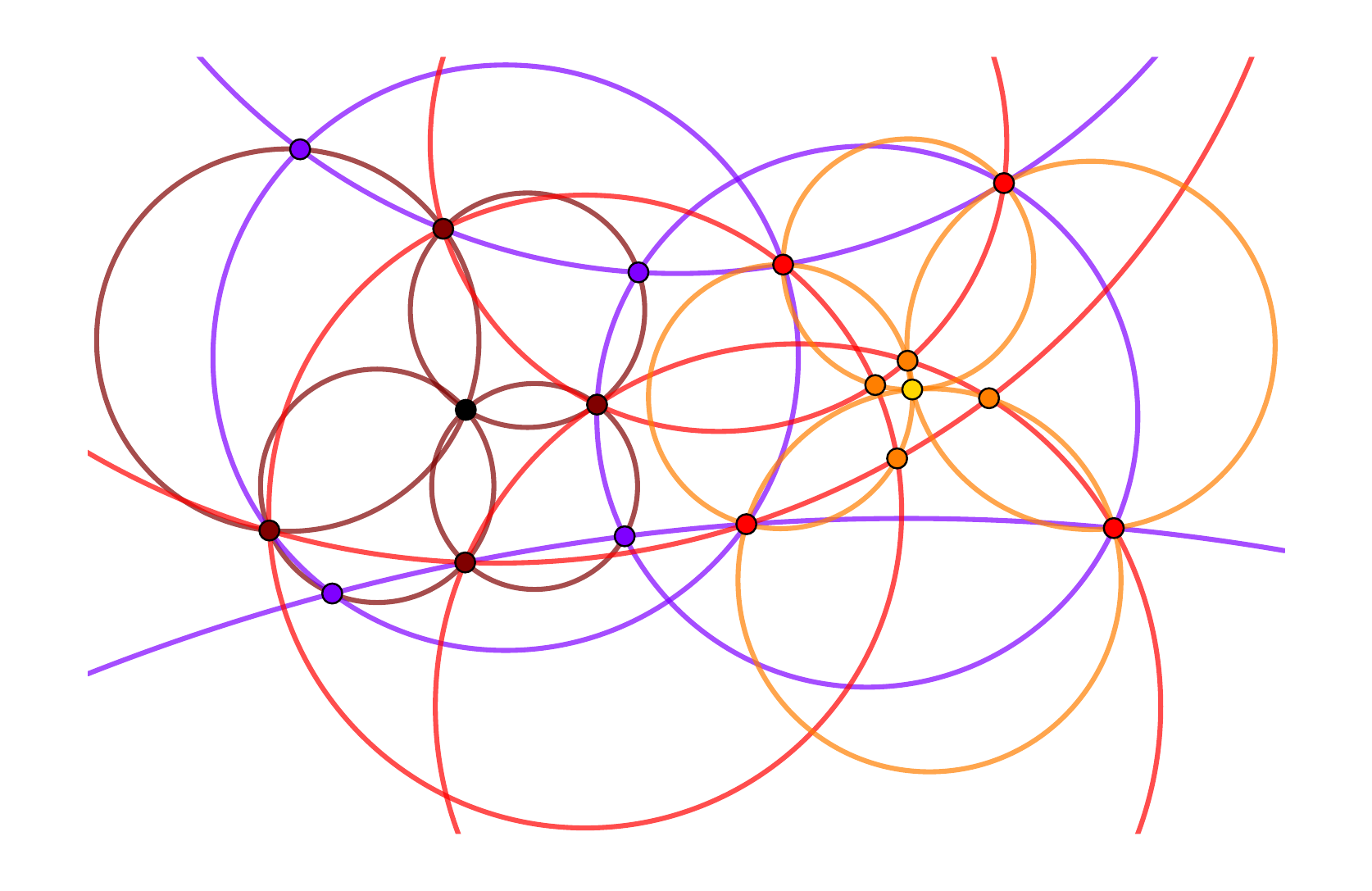}}
        \caption{Devron $(m,2)$-singularity (yellow to black) in circle intersection dynamics for the $8$-closed case, as in Theorem \ref{theo:cidsing}. Half of the initial and final circles are not drawn to improve clarity.}
        \label{fig:cidsing}
\end{figure}

\begin{theorem}\label{theo:cidsing}

        Let $m\geq 3$, and let $v,t$ be $2m$-closed initial data such that $v_i = 0$ for all $i\in 2\Z$. Assume we can apply the propagation map $T$ at least $2m-4$ times. Then, there is $s\in\hC$ such that, for all $i\in2\Z+1$,
        \[
        \tilde w_{2m-3} \equiv s.
        \]
       That is the singularity repeats after $2m-4$ steps. Note that $\tilde w_{2m-3}$ corresponds to a row of circle intersection points.
\end{theorem}
\begin{proof}
        The initial conditions here are similar to those of Theorem \ref{theo:intcrsingular}, that is $\tilde z_{i,0} = 0$ for all $i\in 2\Z$. With the same arguments as those in the proof of Theorem \ref{theo:intcrsingular}, we find that $\tilde w_{2m-2}$ is constant. We recall that $\tilde w_{2m-2}$ corresponds to a row of circle centers, therefore the corresponding circles are concentric. Moreover, for all $i\in 2\Z$ the circle centered at $\tilde w_{i, 2m-2}$ intersects the circle centered at $\tilde w_{i+2, 2m-2}$, therefore the corresponding circles coincide.
         This constricts the configuration after $2m-3$ steps of $T$ considerably. There are three possible cases. Case (i): the intersection points $\tilde z_{2m-2}$  are generic on a common circle. But this is not possible, because then $T^{-1} \circ T^{2m-3}$ would not be defined. Case (ii): the intersection points $\tilde z_{i,2m-2}$ coincide for all $i\in 2\Z$. In this case the circles centered at $\tilde z_{2m-3}$ do not need to coincide for all $i\in 2\Z+1$ and $T^{-1} \circ T^{2m-3}$ is defined. However, if both the circle centers $\tilde w_{2m-2}$ and the intersection points $\tilde z_{2m-2}$ are constant, then we are in the case of $(m,1)$-Devron initial conditions, as in Theorem \ref{theo:ciddodgson}. But then $T^{-m} \circ T^{2m-3}$ is not defined in contradiction to the assumption. Thus there is only case (iii): the circles centered at $\tilde w_{i,2m-2}$ are all the same circle of radius 0 for all $i\in 2\Z$. As a consequence, the intersection points $\tilde z_{i,2m-3}$ all coincide for $i\in 2\Z +1$.
\end{proof}

Note that Theorem \ref{theo:cidsing} is worded differently compared to Glick's conjecture, because we add the information of the circle centers when performing circle intersection dynamics. Therefore, there are no non-reversible steps in circle intersection dynamics in the case of a singularity. This is also the reason why in Glick's conjecture the number of iterations is $2m-6$ and not $2m-4$. Note that in the case of $2m=6$, Theorem \ref{theo:cidsing} is actually Miquel's theorem \cite{miquel}, and the case $2m=8$ is illustrated in Figure \ref{fig:cidsing}.

\begin{remark}\label{rem:cidclifford}
There is an alternative proof for Theorem \ref{theo:ciddodgson} that only relies on the multi-dimensional consistency of the integrable cross-ratio maps. Let us give a sketch of the argument. Let $n\geq 1$; a \emph{Clifford $n$-circle configuration} is a map from the $n$-hypercube to $\CP^1$ such that
        \begin{enumerate}
                \item every even vertex is mapped to a point,
                \item every odd vertex is mapped to a circle,
                \item the image of each even vertex is contained in the image of an odd vertex whenever the two vertices are adjacent.
        \end{enumerate}
Thus, a Clifford $n$-circle configuration consists of $2^{n-1}$ circles and $2^{n-1}$ points, such that through each point pass $n$ circles and on each circle there are $n$ points. Let us identify the $n$-hypercube with the set of subsets of $\{0,1,\dots,n-1\}$. \emph{Clifford's theorem} \cite[page 262]{coxeterintroduction} essentially states that if the images of $\emptyset$ and $\{0\},\{1\},\dots,\{n-1\}$ of a Clifford $n$-circle configuration are given, then the whole Clifford $n$-circle configuration exists and is uniquely determined. Let us write $P$ for the map that comprises a Clifford $n$-circle configuration. Consider the assumptions as in Theorem \ref{theo:ciddodgson}, and a Clifford $(m+1)$-circle configuration. Let $I^b_a = \{a,a+1,\dots,b-1,b\}$, where we understand the indices to be cyclic in $\{1,2,\dots,m\}$. Then we identify as follows, see Figure \ref{fig:cidclifford},
        \begin{align}
                v_{2i} &= P(\emptyset), & c_{2i} &= P(\{0\}),\\
                v_{2i+1} &= P(\{0,i\}), & c_{2i+1} &= P(\{i\}),\\
                T^{2\ell+1}(v)_{2i}  &= P(I^{i+\ell+1}_{i-\ell}), & T^{2\ell+1}(c)_{2i} &= P(0 \cup I^{i+\ell+1}_{i-\ell}),\\
                T^{2\ell}(v)_{2i+1} &= P(0 \cup I^{i+\ell+1}_{i-\ell+1}), & T^{2\ell}(c)_{2i+1} &= P(I^{i+\ell+1}_{i-\ell+1}).
        \end{align}
The initial conditions $(v_i)_{0\leq i<m}, (u_i)_{i<m}$ correspond to the initial conditions of Theorem \ref{theo:ciddodgson}, where $P(\emptyset) = s$ and $s'$ is the center of $P(\{0\})$. With this identification, Clifford's $(m+1)$-circle theorem guarantees that if $m$ is even then indeed $P(\{1,2,\dots,m\}) = T(v)^{m-1}_i = s''$ for all $i\in \Z$ is a single point, and if $m$ is odd then indeed $P(\{0,1,2,\dots,m\}) = T(v)^{m-1}_i = s''$ for all $i\in \Z$ is a single point, which concludes the alternate proof.

It is tempting to suspect that all recurrences of singularities are due to some version of multi-dimensional consistency. However, this is the only example for which we found such an argument. Another, similar argument appears in the case of polygon-recutting (another special case of integrable cross-ratio maps) in \cite[Section 7]{gdevron}.
\end{remark}

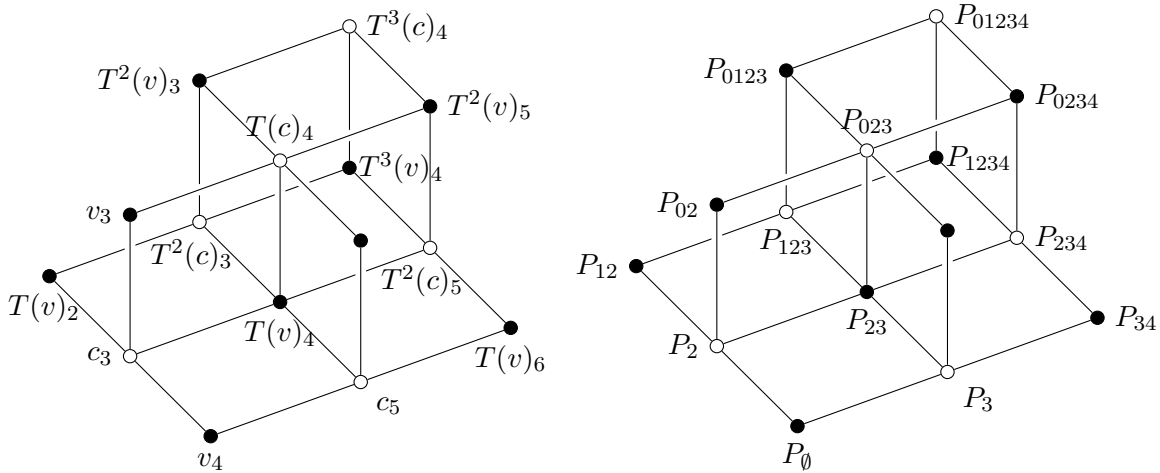
\begin{figure}
        \centering
        \begin{tikzpicture}[scale=1.5]
        \coordinate (e1) at (20:1.4);
        \coordinate (e2) at (135:1);
        \coordinate (e3) at (270:1.25);
        \node[wvert,label=above:$T(c)_{4}$] (v) at (0,0) {};
        \node[bvert,label=right:$T^2(v)_{5}$] (v1) at (e1) {};
        \node[bvert,label=left:$T^2(v)_{3}$] (v2) at (e2) {};
        \node[bvert,label=below:$T(v)_{4}$] (v3) at (e3) {};
        \node[wvert,label=right:$T^3(c)_{4}$] (v12) at ($(e1)+(e2)$) {};
        \node[wvert,label={below,xshift=-3:$T^2(c)_{3}$}] (v23) at ($(e3)+(e2)$) {};
        \node[wvert,label={below,xshift=-3:$T^2(c)_{5}$}] (v13) at ($(e1)+(e3)$) {};
        \node[bvert,label={right,xshift=-3:$T^3(v)_{4}$}] (v123) at ($(e1)+(e2)+(e3)$) {};
        \node[wvert,label=left:$c_{3}$] (v34) at ($(e3)-(e1)$) {};
        \node[wvert,label=below right:$c_{5}$] (v35) at ($(e3)-(e2)$) {};
        \node[bvert,label={[xshift=-2]below:$T(v)_{2}$}] (v234) at ($(e3)+(e2)-(e1)$) {};
        \node[bvert,label=below:$T(v)_{6}$] (v135) at ($(e3)-(e2)+(e1)$) {};
        \node[bvert,label=below:$v_{4}$] (v345) at ($(e3)-(e2)-(e1)$) {};
        \node[bvert,label=left:$v_{3}$] (v4) at ($(0,0)-(e1)$) {};
        \node[bvert] (v5) at ($(0,0)-(e2)$) {};
        \draw[-]
        (v) edge (v1) edge (v2) edge (v3) edge (v4) edge (v5)
        (v12) edge (v1) edge (v2) edge (v123)
        (v23) edge (v3) edge (v2) edge (v123)
        (v13) edge (v1) edge (v3) edge (v123)
        (v3) edge (v34) edge (v35)
        (v23) -- (v234) -- (v34) -- (v345) -- (v35) -- (v135) -- (v13)
        ;
        \draw[-,white,line width=4pt]
        (v) edge (v3) edge (v1) edge (v4) edge (v5)
        (v4) -- (v34) (v5) -- (v35)
        ;
        \draw[-]
        (v) edge (v1) edge (v3) edge (v4) edge (v5)
        (v4) -- (v34) (v5) -- (v35)
        ;
        \end{tikzpicture}\hspace{0mm}
        \begin{tikzpicture}[scale=1.5]
        \coordinate (e1) at (20:1.4);
        \coordinate (e2) at (135:1);
        \coordinate (e3) at (270:1.25);
        \node[wvert,label=above:$P_{023}$] (v) at (0,0) {};
        \node[bvert,label=right:$P_{0234}$] (v1) at (e1) {};
        \node[bvert,label=left:$P_{0123}$] (v2) at (e2) {};
        \node[bvert,label=below:$P_{23}$] (v3) at (e3) {};
        \node[wvert,label=right:$P_{01234}$] (v12) at ($(e1)+(e2)$) {};
        \node[wvert,label=below:$P_{123}$] (v23) at ($(e3)+(e2)$) {};
        \node[wvert,label=right:$P_{234}$] (v13) at ($(e1)+(e3)$) {};
        \node[bvert,label={right,xshift=-3:$P_{1234}$}] (v123) at ($(e1)+(e2)+(e3)$) {};
        \node[wvert,label={left:$P_{2}$}] (v34) at ($(e3)-(e1)$) {};
        \node[wvert,label={below right:$P_{3}$}] (v35) at ($(e3)-(e2)$) {};
        \node[bvert,label=left:$P_{12}$] (v234) at ($(e3)+(e2)-(e1)$) {};
        \node[bvert,label=right:$P_{34}$] (v135) at ($(e3)-(e2)+(e1)$) {};
        \node[bvert,label=below:$P_\emptyset$] (v345) at ($(e3)-(e2)-(e1)$) {};
        \node[bvert,label=left:$P_{02}$] (v4) at ($(0,0)-(e1)$) {};
        \node[bvert] (v5) at ($(0,0)-(e2)$) {}; 
        \draw[-]
        (v) edge (v1) edge (v2) edge (v3)
        (v12) edge (v1) edge (v2) edge (v123)
        (v23) edge (v3) edge (v2) edge (v123)
        (v13) edge (v1) edge (v3) edge (v123)
        (v3) edge (v34) edge (v35)
        (v23) -- (v234) -- (v34) -- (v345) -- (v35) -- (v135) -- (v13)
        ;
        \draw[-,white,line width=4pt]
        (v) edge (v3) edge (v1) edge (v4) edge (v5)
        (v4) -- (v34) (v5) -- (v35)
        ;
        \draw[-]
        (v) edge (v1) edge (v3) edge (v4) edge (v5)
        (v4) -- (v34) (v5) -- (v35)
        ;
        \end{tikzpicture}
        \caption{Identifying circle intersection dynamics with singular initial conditions (as in Theorem \ref{theo:ciddodgson}) with the vertices of a hypercube, as in Remark \ref{rem:cidclifford}.}
        \label{fig:cidclifford}
\end{figure}

\section{The pentagram map}\label{sec:pent}

\subsection{Prerequisites}

For the reader unfamiliar with real projective geometry, we give a short introduction. Consider the equivalence relation $\sim$ on $\R^{N+1}$, such that for $r,r'\in \R^{N+1}$ we have $r\sim r'$ if there is a $\lambda \in \R\setminus \{0\}$ such that $r = \lambda r'$. Let $\mathbf{0}$ denote the zero-vector in $\R^{N+1}$. Every point in the \emph{$n$-dimensional projective space $\RP^N$} is an equivalence class $[r] = \{r' : r' \sim r\}$ for some $r \in \R^{N+1} \setminus \{\mathbf{0}\}$, thus
\begin{align*}
        \RP^N= \{[r] : r \in \R^{N+1} \setminus \{\mathbf{0}\} \} = \left(\R^{N+1} \setminus \{\mathbf{0}\}\right)/\sim.
\end{align*}
Analogously to points in projective space, we consider $m$-dimensional subspaces in $\RP^N$ as \emph{projectivizations} of $(m+1)$-dimensional linear subspaces of $\R^{N+1}$. For any point $[r] \in \RP^N$ and $\ell\in \{0,1,\dots,N-1\}$, we consider the \emph{$\ell$-th (affine) coordinate}
\begin{align}
        \pi_\ell([r]) = \frac{r_\ell}{r_{N}} \in \hat \R,
\end{align}
where $\hat \R = \R \cup \{\infty\}$ and $\pi_\ell([r]) = \infty$ whenever $r_N = 0$. We say that points $[r]$ with $r_N = 0$ are \emph{at infinity}. Note that any
point $[r] \in \RP^N$ is uniquely defined by its $N$ coordinates, unless $r_N = 0$.

 \subsection{Definitions and explicit solution}

The pentagram map was introduced by Schwartz \cite{schwartz}. Subsequent results involve the discovery of Liouville-Arnold integrability \cite{ostpentagram}, of a cluster structure \cite{glickpentagram}, the relation to directed networks and generalizations \cite{gstv}, Liouville-Arnold integrability of higher pentagram maps \cite{kspent}, algebro-geometric integrability \cite{solovievpent,weinreich}, the relation between Liouville-Arnold integrability and the dimer cluster integrable system  \cite{gkdimers,izosimovnetworks}, and the limit points \cite{glicklimit}. This list is not exhaustive. We'll show in the following how to express the iteration of the pentagram map (and the corrugated generalizations) via ratios of oriented dimer partition functions. For future research, it would be interesting to see how the explicit expression relates to the results above, in particular to the Hamiltonians and the limit points.

\begin{definition}
Consider points $(v_i)_{i\in\Z}$ in the real projective plane $\RP^2$. The \emph{pentagram map dynamics} is the dynamics $T$ acting on $v$ such that, for all $i\in \Z$,
\[T(v)_i = v_{i-1}v_{i+1} \cap v_{i}v_{i+2},\]
where $v_iv_j$ denotes the line through $v_i$ and $v_j$, see Figure~\ref{fig:pentagram} for an example.
\end{definition}

Let $\tilde{f}_j=(\tilde{f}_{i,j})_{[i+j]_2=0}$ be points in $\RP^2$ such that, for all $i,j\in \Z$ with $[i+j]_2=0$
\[
\tilde f_{i,j}=\tilde f_{i-3,j-1}\tilde f_{i+1,j-1} \cap \tilde f_{i-1,j-1}\tilde f_{i+3,j-1}.
\]
We refer to $\tilde{f}$ as a \emph{pentagram lattice map}.

Consider the dynamics mapping $(\tilde{f}_0, \tilde{f}_1)$ to $(\tilde{f}_1, \tilde f_2)$. This dynamics coincides with the pentagram map $T$ applied to $(v_i)_{i\in\Z}$ by identifying $v$ with $\tilde f$ via
\begin{equation}~\label{equ:pentagram_relation}
v_i = \tilde f_{2i,0}, \quad \text{ and} \quad T(v)_i=\tilde{f}_{2i+1,1}.
\end{equation}
We then have, for all $i,j\in\Z$ such that $[i+j]_2=0$, $\tilde{f}_{i,j}=T^j(v)_{\frac{i-j}{2}}.$

The map $T$ naturally extends to
$T:(\RP^2)^{\Z_2\times\Z}\rightarrow(\RP^2)^{\Z_2\times\Z}$, mapping $(\tilde{f}_{j-1},\tilde f_{j})\mapsto (\tilde{f}_{j}, \tilde{f}_{j+1})$, and is also referred to as the \emph{pentagram map dynamics}. Note that we use the same notation $T$ for the dynamics acting on points or pairs of points; which one is used should be clear from the context and should not lead to confusion.
As usual, we denote by $f:\Z^2\rightarrow \RP^2$ the map obtained from $\tilde{f}$ by the change of coordinates of Equation~\eqref{eq:rotated_weights}:
\[
f_{i,j}=\tilde{f}_{i-j,i+j}.
\]

\begin{figure}[tb]
        \centering
        \small
        \begin{tikzpicture}[scale=1.3]
                \node[wvert,label={left:$v_{-1} = \tilde f_{-2,0}$}] (a1) at (0,.5) {};
                \node[wvert,label={left:$v_{0}=\tilde{f}_{0,0}$}] (a3) at (1,2) {};
                \node[wvert,label={above right:$v_{1}=\tilde{f}_{2,0}$}] (a5) at (3.5,3) {};
                \node[wvert,label={right:$v_{2}=\tilde{f}_{4,0}$}] (a7) at (6,2) {};
                \node[wvert,label={below right:$T(v)_{-2}=\tilde{f}_{-3,1}$}] (b0) at (1,0) {};
                \node[wvert,label={below:$T(v)_{2}=\tilde{f}_{5,1}$}] (b8) at (5.5,0) {};
                \node[wvert,label={below right:$T(v)_{-1}=\tilde{f}_{-1,1}$}] (b2) at (intersection of a1--a5 and b0--a3) {};
                \node[wvert,label={below right:$T(v)_0=\tilde{f}_{1,1}$}] (b4) at (intersection of a3--a7 and a1--a5) {};
                \node[wvert,label={[xshift=3]below right:$T(v)_1=\tilde{f}_{3,1}$}] (b6) at (intersection of a5--b8 and a3--a7) {};

                \draw[densely dotted]
                        (b0) -- (a1) -- (b2) -- (a3) -- (b4) -- (a5) -- (b6) -- (a7) -- (b8)
                ;
                \draw[-]
                        (b0) -- (b2) -- (b4) -- (b6) -- (b8)
                        (0,0) -- (a1) -- (a3) -- (a5) -- (a7) -- (6.2,0)
                ;
                \coordinate (s) at (0,-0.3);
        \end{tikzpicture}

\caption{Labeling of the pentagram map.
    }
    \label{fig:pentagram}
\end{figure}
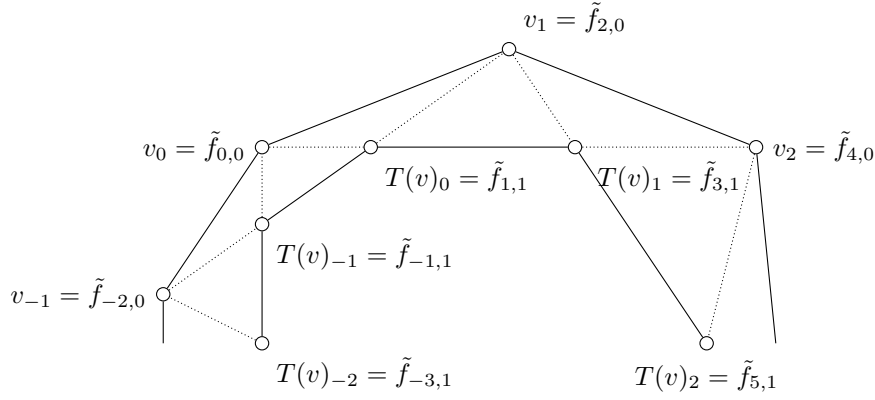

In the sequel, we will be working with affine coordinates. To that purpose, we introduce the following notation: for every $\ell \in\{0,1\}$,
\begin{equation}\label{equ:affine_coord}
\forall (i,j)\in\Z^2,\ g_{i,j}^\ell=\pi_\ell(f_{i,j}),\quad \forall\ i\in\Z,\ u_i^\ell=\pi_\ell(v_i),\quad T(u)_i^\ell=\pi_{\ell}(T(v)_i).
\end{equation}
Note that in the sequel, to simplify notation, we do not index the left-hand-sides with $\ell$, but the reader should keep in mind that whenever $g$ or $u$ is used, what is intended is `` for every $\ell\in\{0,1\}$, $g^\ell,u^\ell$ ''.

The following lemma is a consequence of Menelaus' theorem, and underlies the occurrence of dSKP in the pentagram map. This fact has not been published, but is outlined in the slides \cite{schieftalk} of a talk by Schief.

\begin{lemma}\label{lem:pentdskp}
Let $f:\Z^2\rightarrow \RP^2$ be a pentagram lattice map, and $g$ be its affine coordinates.
Then, for all $(i,j)\in\Z^2$, the following holds
\begin{align}
\frac{(g_{j+i,-i}-g_{j+i-1,-i+2})(g_{j+i,-i+1}-g_{j+i+1,-i+1})(g_{j+i+2,-i-1}-g_{j+i+1,-i})}{
(g_{j+i-1,-i+2}-g_{j+i,-i+1})(g_{j+i+1,-i+1}-g_{j+i+2,-i-1})(g_{j+i+1,-i}-g_{j+i,-i})}=-1.\label{eq:pentadskp}
\end{align}
\end{lemma}
\begin{proof}
By Menelaus's theorem \cite[Theorem 9.8]{ddgbook} we have that, for every $i\in\Z$,
\begin{align}
\frac{(v_i-T(v)_{i-2})(T(v)_{i-1}-T^2(v)_{i-1})(T(v)_{i+1}-T(v)_{i})}{
        (T(v)_{i-2}-T(v)_{i-1})(T^2(v)_{i-1}-T(v)_{i+1})(T(v)_{i}-v_i)} = -1.\label{eq:pentamenelaus}
\end{align}
Equation \eqref{eq:pentadskp} is nothing but Equation \eqref{eq:pentamenelaus} after coordinate projection. As multi-ratios (ratios of differences as on the left-hand side) are invariant under projections \cite[Theorem 9.10]{ddgbook}, the lemma is proven.
\end{proof}

A consequence of Lemma \ref{lem:pentdskp} is that $\tilde g_j$ for $j \geq 2$ depends only on $\tilde g_0, \tilde g_1$, no data of other coordinates is necessary.
The next theorem expresses the points $g_{i,j}$, $i+j\geq 1$ of the pentagram lattice map as a ratio function of oriented dimers of an Aztec diamond subgraph of $\Z^2$, with face weights being a subset of $(g_{i,j})_{i+j\in\{0,1\}}$, see also Figure~\ref{fig:pentagram_ic}. Note that since affine coordinates determine points, this theorem yields an explicit formula for the points $f$ as well.

\begin{figure}[t]
  \centering
  \includegraphics[width=15cm]{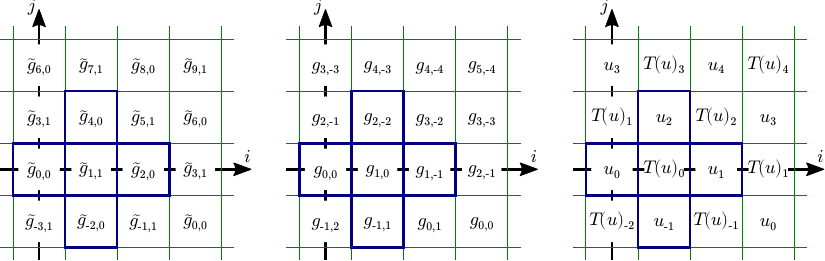}
  \caption{Face weights for the explicit solution of the pentagram
    map, in the three possible labelings. The shown Aztec diamond
    corresponds to the computation of $\tilde{g}_{1,3}=g_{2,1}$. Doing
  this for different affine coordinates $g$ of $f$ determines $f$.}
  \label{fig:pentagram_ic}
\end{figure}

\begin{theorem}\label{theo:explpent}
Let $f:\Z^2 \rightarrow \RP^2$ be a pentagram lattice map, $g$ be its affine coordinates, and $T$ be the corresponding pentagram map dynamics. Consider the graph $\Z^2$ with face-weights $(a_{i,j})_{(i,j)\in\Z^2}$ given by,
\begin{align}\label{eq:explpentweight}
a_{i,j}&=g_{\frac{i+3j+[i+j]_2}{2},\frac{-i-3j+[i+j]_2}{2}}=\tilde{g}_{i+3j,[i+j]_2},\nonumber\\
&=
\begin{cases}
u_{\frac{i+3j}{2}}& \text{ if $[i+j]_2=0$},\\
T(u)_{\frac{i+3j-1}{2}}& \text{ if $[i+j]_2=1$}.
\end{cases}
\end{align}
Then, for all $(i,j)\in\Z^2$ such that $i+j\geq 1$, we have
\begin{align}
g_{i,j}&=Y\Bigl(A_{i+j-1}\Bigl[g_{\frac{i-j+[i-j]_2}{2},\frac{-(i-j)+[i-j]_2}{2}}\Bigr],a\Bigr)
=Y\Bigl(A_{i+j-1}\bigl[\tilde{g}_{i-j,[i-j]_2}\bigr],a\Bigr)\nonumber\\
&=
\begin{cases}
Y\Bigl(A_{i+j-1}\bigl[u_{\frac{i-j}{2}}\bigr],a\Bigr)&\text{ if $[i-j]_2=0$},\\
Y\Bigl(A_{i+j-1}\bigl[T(u)_{\frac{i-j-1}{2}}\bigr],a\Bigr)&\text{ if $[i-j]_2=1$}.
\end{cases}
\end{align}

\end{theorem}
\begin{proof}
Consider the function $x:\calL\rightarrow\RP^2$ given by
\begin{equation}\label{equ:proof_pentagram}
x(i,j,k)=g_{\frac{i+3j+k}{2},\frac{-i-3j+k}{2}}.
\end{equation}
As a consequence of Lemma~\ref{lem:pentdskp}, we have that, for $k\geq 1$, the function $x$ satisfies the dSKP recurrence. Note that the function $x$ satisfies the initial condition
\[
a_{i,j}=x(i,j,[i+j]_2)=g_{\frac{i+3j+[i+j]_2}{2},\frac{-i-3j+[i+j]_2}{2}},
\]
giving the face weights~\eqref{eq:explpentweight} of the statement. The versions of the face weights involving $\tilde{g}$, $T$ and $u$ are obtained using the change of coordinates~\eqref{eq:rotated_weights} and the identification~\eqref{equ:pentagram_relation}.

As a consequence of Theorem~\ref{theo:expl_sol}, we know that
$x(i,j,k)=Y(A_{k-1}[a_{i,j}],a)$. Using
Equation~\eqref{equ:proof_pentagram}, we have that for every $m$ such that
$[i-j-m]_3=0$,
$\bigl(m,\frac{i-j-m}{3},i+j\bigr) \in
\calL$ and
\[
x\Bigl(m,\frac{i-j-m}{3},i+j\Bigr)=g_{i,j}.
\]
Applying this at $m=i-j$, we get that $g_{i,j}=Y(A_{i+j-1}[a_{i-j,0}],a)$. An explicit computation of the face weight $a_{i-j,0}$ gives
\[
a_{i-j,0}=g_{\frac{i-j+[i-j]_2}{2},\frac{-(i-j)+[i-j]_2}{2}}.
\]

This shows that $g_{i,j}$ is the ratio of partition functions for
the Aztec diamond of size $i+j-1$ centered at $(i-j,0)$, whose
central face has weight $g_{\frac{i-j+[i-j]_2}{2},\frac{-(i-j)+[i-j]_2}{2}}$. However, note that the whole
solution \eqref{equ:proof_pentagram} is invariant under translations of
multiples of $(3,-1,0)$. Therefore, \emph{any} Aztec diamond of size
$i+j-1$ whose central face has weight $g_{\frac{i-j+[i-j]_2}{2},\frac{-(i-j)+[i-j]_2}{2}}$ is a translate of
the first one, and has the same face weights and
ratio of partition functions. Therefore, it is not important that the
Aztec diamond is centered at $(i-j,0)$, but only that it has the
announced central face weight.\qedhere
\end{proof}

\begin{remark}\label{rem:pent_expl_alternative} \leavevmode
\begin{itemize}
\item Equivalently, using the change of coordinates \eqref{eq:rotated_weights_1}, we have that for
$(i,j)\in\Z^2$ such that $[i+j]_2=0$ and $j\geq 1$,
\begin{align}
T^j(u)_{\frac{i-j}{2}}=\tilde{g}_{i,j}&=Y\Bigl(A_{j-1}\bigl[\tilde{g}_{i,[i]_2}\bigr],a\Bigr)
=
\begin{cases}
Y\Bigl(A_{j-1}\bigl[u_{\frac{i}{2}}\bigr],a\Bigr)&\text{ if $[i]_2=0$},\\
Y\Bigl(A_{j-1}\bigl[T(u)_{\frac{i-1}{2}}\bigr],a\Bigr)&\text{ if $[i]_2=1$}.
\end{cases}
\end{align}
\item
            Glick introduced a quiver of a cluster algebra which he associated to the pentagram map \cite{glickpentagram}. We do not need cluster algebras in this paper, but each local occurrence of the dSKP equation corresponds to a mutation in a cluster algebra \cite{athesis}. The weights of the Aztec diamond that we use to express the iterations of the pentagram map are in fact periodic, that is they are well defined on $\Z^2/\{(3,-1)\}$. In the case of $m$-closed initial conditions the weights are doubly periodic, that is well defined on $\Z^2/\{(3,-1), (m,-m)\}$. The quiver that appears in \cite{glickpentagram} has the combinatorics of $\Z^2/\{(3,-1), (m,-m)\}$ as well.
\end{itemize}
\end{remark}

\subsection{Singularities}

Let $m\geq 1$. A pentagram lattice map $f$ is said to be \emph{$m$-closed} if, for all $(i,j)\in \Z^2$, $f_{i+m,j-m} = f_{i,j}$. The initial data $v$ is called \emph{$m$-closed} if, for all $i\in\Z$, $v_{i + m} = v_i$; the two definitions are equivalent due to identification \eqref{equ:pentagram_relation}. The statement of the recurrence of the singularity in the next theorem is a variant of a theorem already given by Schwartz
\cite{schwartzaxisaligned}.
Moreover, the explicit expression in Theorem \ref{th:pentdodgson} was conjectured by Glick (unpublished) and proven by Yao \cite[Theorem 1.3]{yao}.

\begin{figure}[tb]
  \centering
  \frame{\includegraphics[height=6cm]{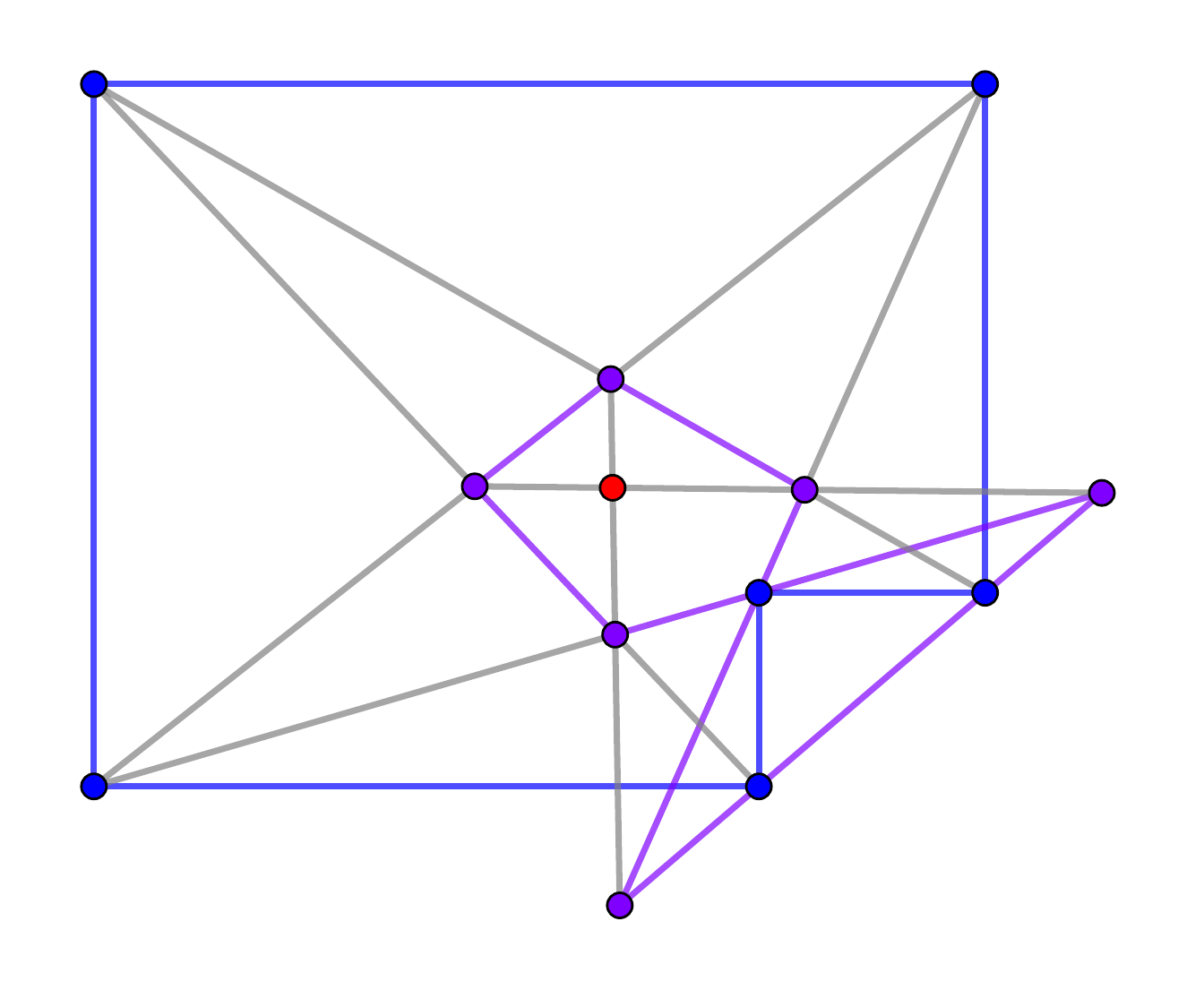}}\hspace{4mm}
  \frame{\includegraphics[width=6cm,angle=90]{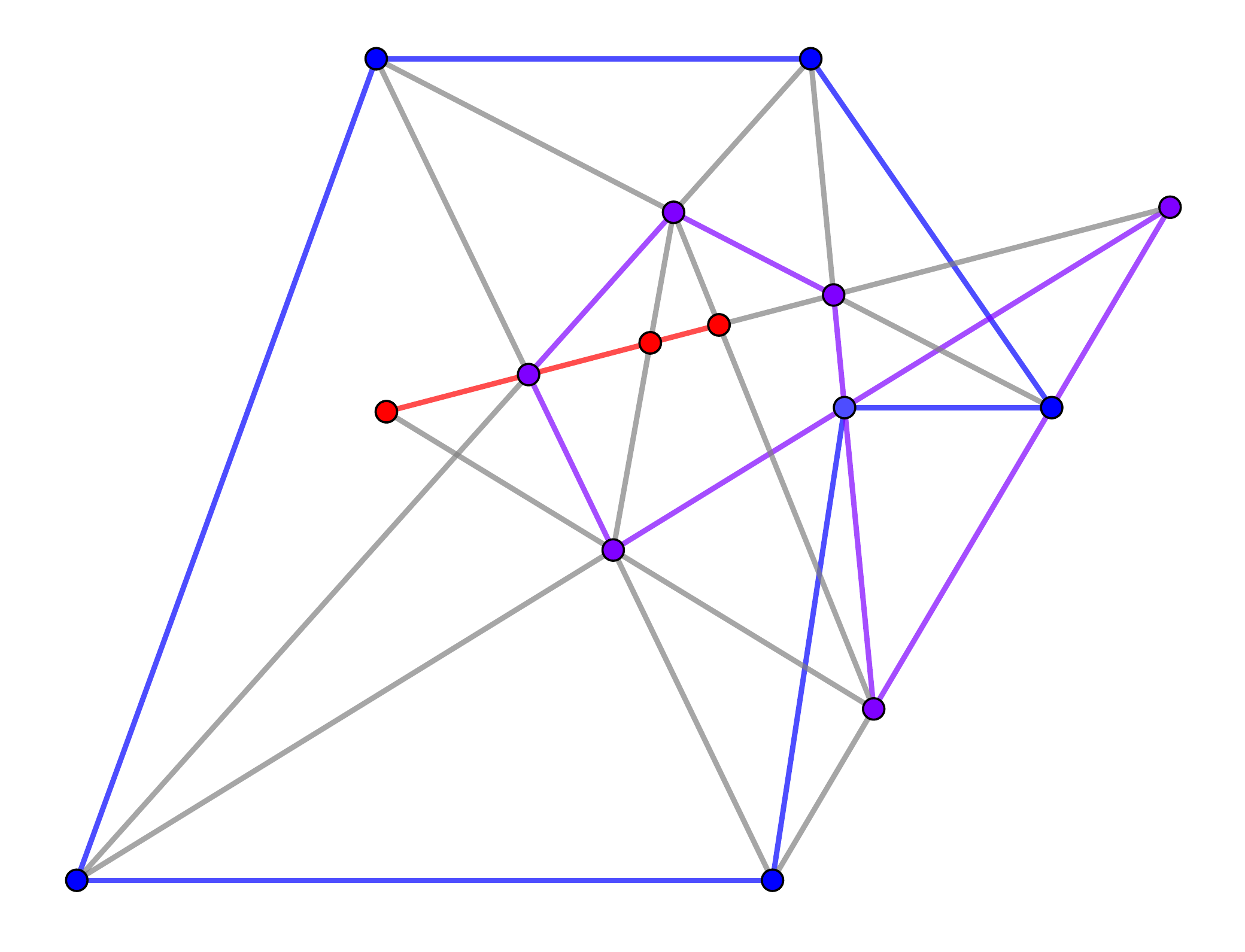}}
  \caption{A Dodgson singularity (left, Theorem~\ref{th:pentdodgson}) and an $(m,2)$-Devron singularity (right, Theorem \ref{theo:pentsing}) for two 6-closed curves by the pentagram map.}
  \label{fig:pentsing}
\end{figure}

\begin{theorem}\label{th:pentdodgson}
Let $m\geq 1$, and let $\tilde{f}:\Z^2\rightarrow\RP^2$ be a
$2m$-closed pentagram lattice map, or equivalently, let
$v:\Z\rightarrow\RP^2$ be $2m$-closed. Assume that the lines
$\tilde{f}_{2i,0}\tilde{f}_{2i+2,0}=v_{i}v_{i+1}$ are parallel to the
$x$-axis for $i\in 2\Z$, and are parallel to the $y$-axis for $i\in
2\Z+1$. Assume that we can apply the propagation map at least $m-1$
times. Then, for all $i\in\Z$ such that $[i+m-1]_2=0$,  we have
\[
\tilde{f}_{i,m-1}=\frac{1}{2m}\sum_{\ell=0}^{2m-1}\tilde{f}_{2\ell,0},\ \text{ otherwise stated, }\
T^{m-1}(v)\equiv \frac{1}{2m}\sum_{\ell=0}^{2m-1}v_\ell.
\]
In other words, $T^{m-1}(v)$ is constant equal to the center of mass of $v_0,\dots,v_{2m-1}$.
\end{theorem}
\begin{proof}
  Let $g = \pi_0(f)$ be the $x$-coordinate of $f$. Then we
  have that $\tilde g_{-1} \equiv \infty$, because all the points of
  $\tilde f_{-1}$ are at infinity.
  Additionally, let
  $\tilde g'_{i,j} = \tilde g_{i,j}^{-1}$ for all $i,j\in\Z$ with
  $[i+j]_2=0$. Note that $g'$ is $g$ after a projective
  transformation, so $g'$ satisfies dSKP in the sense of Lemma~\ref{lem:pentdskp} as well. Moreover, $\tilde g'_{-1} \equiv
  0$.
  We use $(g'_{-1},g'_0)$ as initial data, that is, we
    consider the dSKP solution adapted from the proof of
    Theorem~\ref{theo:explpent} with this offset:
    $x(i,j,k)=\tilde g'_{i+3j-1,k-1}$, corresponding to the initial
    data $a_{i,j}=\tilde g'_{i+3j-1,[i+j]_2-1}$. As
    $\tilde g'_{-1} \equiv 0$, the initial data satisfies the
    hypothesis of Proposition \ref{prop:Nmat} with $d=0$. For any $i$ such that
    $[i+m-1]_2=0$, we have $\tilde g'_{i,m-1}=x(i+1,0,m)$, and we get
    the following $m\times m$ matrix, with $I=i+3m-1$:
    \begin{equation*}
      N =
      \begin{pmatrix}
        \tilde g_{I,0} & \tilde g_{I-2,0} & \tilde
        g_{I-4,0} & \dots \\
        \tilde g_{I-4,0} & \tilde g_{I-6,0} & \tilde
        g_{I-8,0} & \dots \\
        \tilde g_{I-8,0} & \dots &  &  \\
        \vdots & & &
      \end{pmatrix}.
    \end{equation*}
    Since $v$ is $2m$-closed, we have for all
    $\ell\in \Z,\ \tilde g_{2\ell,0} = \tilde g_{2\ell+4m,0}$. By the alignment
    with the axes, we also have for all
    $\ell\in 2\Z+1,\ \tilde g_{2\ell,0} = \tilde g_{2\ell+2,0}$. We obtain that
    the $(1,1,\dots,1)^T$ vector is an eigenvector for $N^T$ with
    eigenvalue
    $\sum_{\ell=0}^{m-1} {\tilde{g}}_{4\ell,0} = \frac12\sum_{\ell=0}^{2m-1}
    {\tilde{g}}_{2\ell,0}$. Therefore, by the same argument as in
    Theorem~\ref{theo:pnetpremature} on $N^T$, we obtain that
    \begin{align}
    \tilde g'_{i,m-1} = \frac{2m}{\sum_{\ell=0}^{2m-1}
        {\tilde{g}}_{2\ell,0}},
    \end{align}
     so that $\tilde g_{m-1}$ is identically
    equal to center of mass of $\tilde g_0$. As the argument works
    analogously for $\pi_1(f)$, the claim is proven.
\end{proof}
Apart from this well known axis-aligned case, we also show another type of singularity that did not previously appear in the literature.

\begin{theorem}\label{theo:pentsing}
Let $m\geq 1$, and let $\tilde{f}:\Z^2\rightarrow\RP^2$ be a
$2m$-closed pentagram lattice map, or equivalently, let
$v:\Z\rightarrow\RP^2$ be $2m$-closed. Assume that the lines
$\tilde{f}_{2i,0}\tilde{f}_{2i+2,0}=v_{i}v_{i+1}$ are parallel to the
$x$-axis for $i\in 2\Z$. Assume that we can apply the propagation map
at least $2m-4$ times. Then, $\tilde{f}_{2m-4}$, or equivalently
$T^{2m-4}(v)$, is singular. That is $\tilde{f}_{2m-3}$, or
equivalently $T^{2m-3}(v)$, is not defined.
\end{theorem}
\proof{
We use the same notation as in the proof of
  Theorem~\ref{th:pentdodgson}. In this case $\tilde g'_{i,-1}=0$ for
  any $i\in \Z$ such that $[i]_4=3$. This implies that the dSKP
  solution $x(i,j,k)=\tilde g'_{i+3j-1,k-1}$ now has initial data that
  are $(m,2)$-Devron. Theorem~\ref{theo:devron_sing} implies that
  every other diagonal at height $k=2m-2$ is constant, which means that
half the points of $\tilde g_{2m-3}$ coincide. The same argument works for $\pi_1(f)$ and thus the observation holds for $f$ itself. However, this means that all the diagonals of $T^{2m-4}(v)$ intersect in a single point, which is not possible if $T^{2m-4}(v)$ is non-singular, which proves the theorem.
\qed

Although Theorem \ref{theo:pentsing} predicts a singularity, it does not state the type of singularity. Numeric simulations suggest the following conjecture.

\begin{conjecture}\label{conj:pentsing}
Suppose the hypotheses of Theorem~\ref{theo:pentsing} are satisfied.
    Then
    \begin{align}
                T^{2m-4}(v)_i = T^{2m-4}(v)_{i+1},
    \end{align}
    for either all $i\in 2\Z$ or all $i\in 2\Z+1$. Moreover, there is a line $L\subset \RP^2$ such that $T^{2m-4}(v)_i\in L$ for all $i\in \Z$.
\end{conjecture}
So far, the type of singularity in Conjecture~\ref{conj:pentsing} has
been treated by considering initial data $T^{-1}(u)$ and $u$, which
are $(m,2)$-Devron. However, if we start instead from initial data
$u$ and $T(u)$, we get a different type of singularity (for a specific
projection), which is not $(m,p)$-Devron, and is described in the
following conjecture. Another
way to see this singularity appear is to consider the backwards
dynamics starting from $T^{2m-4}(u),T^{2m-5}(u)$.
Moreover, the occurrence of this new singularity in Conjecture
\ref{conj:pentsing} is specific to the pentagram map, and it occurs at a later step for generic dSKP initial data with the same periodicities. Instead, we have the following conjecture for dSKP.

\begin{conjecture}\label{conj:pairsing}
  Let $x:\calL \to \hC$ be a solution of the dSKP recurrence.
  Suppose that for some $m\in 2\N + 2$ and $p\in 2\N + 2$ the initial condition $a_{i,j}=x(i,j,[i+j]_2)$ are such
  that for all $(i,j)\in\Z^2,$
  \begin{equation*}
   a_{i,j}=a_{i+m,j+m}=a_{i+p-1,j+p+1},
 \end{equation*}
  and that for all $(i,j) \in \Z^2$ with $[i+j]_4=0,$
  \begin{equation*}
    a_{i,j}=a_{i+1,j+1}.
  \end{equation*}
  Then, either for all $(i,j) \in \Z^2$ such
  that $[i+j]_4=[m]_4$, or for all $(i,j)\in\Z^2$ such that $[i+j]_4=[m+2]_4$,
  \begin{equation*}
    x(i,j,m)=x(i+1,j+1,m).
  \end{equation*}
\end{conjecture}

\subsection{Corrugated pentagram maps}\label{sec:corrugated}
We keep this section brief, as we do not need to apply any new techniques compared to the standard pentagram map. There are various generalizations of the pentagram map. In this section we consider a generalization introduced by Gekhtman, Shapiro, Tabachnikov and Vainshtein \cite{gstv}, and we suppose throughout that $N\geq 2$.

\begin{definition}
        An \emph{$N$-corrugated polygon} is a map $v: \Z \rightarrow \RP^N$ if, for all $i\in \Z$, the points $v_i,v_{i+1},v_{i+N},v_{i+N+1}$ span a plane. The \emph{$N$-corrugated pentagram map dynamics} is the dynamics $T$ acting on $v$ such that, for all $i\in\Z$,
        \begin{align}
                T(v)_i &= v_{i-1}v_{i+N-1} \cap v_{i}v_{i+N}.
        \end{align}
\end{definition}
It is straightforward to verify, that generically $T(v)$ is an $N$-corrugated polygon again \cite{gstv}.

We use the same notation for affine coordinates as in Equation~\eqref{equ:affine_coord}, except that we now have $N$ coordinates, so that $\ell\in\{0,\dots,N-1\}$. Similarly to Lemma~\ref{lem:pentdskp} we obtain the following.

\begin{lemma}\label{lem:corpentdskp}
Let $v:\Z\rightarrow \RP^N$ be an $N$-corrugated polygon, and $u$ be its affine coordinates. Consider the associated corrugated pentagram map dynamics $T$.
Then, for all $i\in \Z$, the following holds
    \begin{align}
\frac{(u_{i-1+N}-T(u)_{i}) (T(u)_{i-1}-T^2(u)_{i}) (T(u)_{i-1+N} - T(u)_{i+N})}{(T(u)_{i}-T(u)_{i-1}) (T^2(u)_{i}-T(u)_{i-1+N}) (T(u)_{i+N} - u_{i-1+N})} = -1.
    \end{align}
\end{lemma}

The $2$-corrugated pentagram map coincides with the standard pentagram map. The propagation of P-nets, that we discussed in Section \ref{sec:pnets}, can algebraically be understood as the 1-corrugated pentagram map \cite{gstv}, where by algebraically we mean that a P-net $p$ satisfies Lemma \ref{lem:corpentdskp} with $N=1$ and the identification $T^j(p)_i = p_{i,j}$.

Using the same arguments as in Theorem~\ref{theo:explpent}, see also Remark~\ref{rem:pent_expl_alternative}, and Theorem~\ref{th:pentdodgson}, we obtain the following results. Given that the proofs are so close, we choose to omit them.

\begin{theorem}\label{theo:explcorpent}
Let $v:\Z \rightarrow \RP^N$ be an $N$-corrugated polygon, $u$ be its affine coordinates, and $T$ be the associated corrugated pentagram map dynamics.
Consider the graph $\Z^2$ with face-weights $(a_{i,j})_{(i,j)\in\Z^2}$ given by,
       \begin{align}
                a_{i,j} = \begin{cases}
                u_{\frac{(N-1)i+(1+N)j}{2}} & \mbox{if } [i+j]_2 = 0, \\
                T(u)_{\frac{(N-1)(i-1)+(1+N)j}{2}} & \mbox{if } [i+j]_2 = 1.
        \end{cases}
    \end{align}
Then, for all $(i,j)\in\Z^2$ such that $j\geq 1$ we have
        \begin{align}
                T^j(u)_i &=
\begin{cases}
Y\left(\az{j-1}{u_{i+(N-1)\frac{j}{2}}},a\right)&\text{ if $[j]_2=0$},\\
Y\left(\az{j-1}{T(u)_{i+(N-1)\frac{j-1}{2}}},a\right)&\text{ if $[j]_2=1$}.
\end{cases}
\end{align}
\end{theorem}

The reocurrence of the singularity in the following theorem was already proven by Glick \cite[Theorem 6.12]{gdevron}, and the explicit formula by Yao \cite[Theorem 1.4]{yao}. In our setup, it follows as an immediate corollary.
\begin{theorem}\label{th:corpentdodgson}
        Let $v:\Z\rightarrow\RP^n$ be an $mN$-closed $N$-corrugated polygon, such that for all $i\in \Z_{mN}, \ell\in \Z_N$ all lines $v_{iN+\ell}v_{iN+\ell+1}$ are parallel to the $\ell$-th coordinate. If $T^{m-1}(v)$ exists then $T^{m-1}(v)$ is the center of mass $\frac1{mN}\sum_{\ell=0}^{mN-1}v_\ell$ of the points $v_0,v_1,\dots,v_{mN-1}$.
\end{theorem}

\section{The short diagonal hyperplane map}
\label{sec:hyppent}

In this section we study the short diagonal hyperplane map, which was introduced by Khesin and Soloviev \cite{kspent} and Marí Beffa \cite{beffapent}.

Let $(v_i)_{i\in\Z}$ be a polygon in $\RP^3$. We say $(c_i)_{i\in\Z}$
is a \emph{companion polygon} if for all $i\in \Z$ the point $c_i$ is
on the line $v_{i-1}v_{i+1}$ and on the plane spanned by $v_{i-2}$,
$v_i$ and $c_{i-2}$. The idea of employing a companion polygon is due to Glick and Pylyavskyy \cite{gpymeshes}. It is not hard to see that generically, to each
polygon there exists a two parameter family of companion polygons.

\begin{definition}\label{def:hyppent}
        Consider a polygon $(v_i)_{i\in\Z}$ and a companion polygon $(c_i)_{i\in\Z}$ in $\RP^3$. The \emph{short diagonal hyperplane map dynamics} is the dynamics $T$ acting on $v, c$ such that, for all $i\in \Z$,
        \begin{align}
                T(v)_i &= v_{i-1}v_{i+1} \cap v_{i-2}v_{i}v_{i+2}, \label{eq:hypdyn}\\
                T(c)_i &= c_{i-1}c_{i+1} \cap T(v)_{i-1}T(v)_{i+1}, \label{eq:comphypdyn}
        \end{align}
        where $v_iv_jv_k$ denotes the plane through $v_i, v_j$ and $v_k$.
\end{definition}

Let $(u_i)_{i\in\Z}$ be affine coordinates of $(v_i)_{i\in\Z}$ and
$(b_i)_{i\in\Z}$ be the same affine coordinates of $(c_i)_{i\in\Z}$.

\begin{lemma}\label{lem:dskphyperpent}
        Consider the coordinates $(u_i)_{i\in\Z}$ of a polygon and $(b_i)_{i\in\Z}$ of a companion polygon. Then for all $i\in\Z$
        \begin{align}
                \frac{(u_{i}-b_{i+1})(T(u)_{i+1}-T(b)_{i})(T(u)_{i-1}-b_{i-1})}{(b_{i+1}-T(u)_{i+1})(T(b)_{i}-T(u)_{i-1})(b_{i-1}-u_{i})} &= -1,\\
                \frac{(T^{-1}(b)_{i}-b_{i+1})(u_{i+1}-T(u)_{i})(u_{i-1}-b_{i-1})}{(b_{i+1}-u_{i+1})(T(u)_{i}-u_{i-1})(b_{i-1}-T^{-1}(b)_{i})} &= -1.
        \end{align}
\end{lemma}
\proof{
        Note that $c_i$ and $T(v)_i$ are on the line $v_{i-1}v_{i+1}$, because of the definitions of $c$ and Equation \eqref{eq:hypdyn} respectively. Therefore, for all $i\in \Z$
        \begin{align}
                v_i = c_{i-1} T(v)_{i-1} \cap c_{i+1} T(v)_{i+1}.
        \end{align}
        On the other hand Equation \eqref{eq:comphypdyn} holds. Therefore, the first equation in the claim is about the projection of Menelaus configurations, and the first equation of the lemma follows as in the proof of Lemma \ref{lem:pentdskp}.	Moreover, note that both $c_{i-1}$ and $c_{i+1}$ are in the plane $v_{i-2}v_{i}v_{i+2}$, and they are coplanar with $v_{i-1},v_{i-2}$, thus
        \begin{align}
                T(v)_i = c_{i-1}c_{i+1} \cap v_{i-1} v_{i+1}.
        \end{align}
        Therefore $T(v)_i$ is on the line $c_{i-1}c_{i+1}$ and so is $T(c)_i$ because of Equation \eqref{eq:comphypdyn}. Conversely, both $T^{-1}(c)_{i-1}$ and $T^{-1}(c)_{i+1}$ are on the line $v_ic_i$. As a consequence,
        \begin{align}
                T^{-1}(c)_i = v_{i-1}c_{i-1} \cap v_{i+1} c_{i+1}. \label{eq:reversecompanion}
        \end{align}
        Hence the second equation of the lemma is also about the projection of a Menelaus configuration, which proves the second equation.\qed
}

\begin{figure}
    \centering
  \includegraphics[width=7cm]{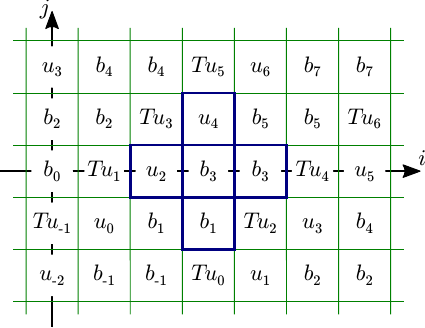}
  \caption{Face weights for the explicit solution of the short
    diagonal hyperplane map. The shown Aztec diamond
    corresponds to the computation of $T^2(u)_3$.}
  \label{fig:pent_hyp_ic}
\end{figure}

\begin{theorem}\label{theo:explhyperpent}
        Consider the coordinates $(u_i)_{i\in\Z}$ of a polygon and $(b_i)_{i\in\Z}$ of a companion polygon.
        Let $v:\Z \rightarrow \RP^3$ be a polygon and $c:\Z \rightarrow \RP^3$ a companion polygon, and let $u,b$ be the affine coordinates of $v,c$, and $T$ be the short diagonal hyperplane map dynamics.
        Consider the graph $\Z^2$ with face-weights $(a_{i,j})_{(i,j)\in\Z^2}$ given by,
               \begin{align}
                        a_{i,j} = \begin{cases}
                        b_{i+j-\lfloor\frac{i-j}{4}\rfloor} & \mbox{if } [i-j]_4 = 0, 3, \\
                        T(u)_{i+j-\lfloor\frac{i-j}{4}\rfloor} & \mbox{if } [i-j]_4 = 1, \\
                        u_{i+j-\lfloor\frac{i-j}{4}\rfloor} & \mbox{if } [i-j]_4 = 2.
                \end{cases}
            \end{align}
        Then, for all $(i,k)\in\Z^2$ such that $k\geq 1$ we have
                \begin{align}
                        T^k(u)_i &=
        \begin{cases}
        Y\left(A_{2k-2}(\frac{i}{2} + k + 2,\frac{i}{2} - k - 1),a\right)&\text{if } [i]_2 = 0,\\
        Y\left(A_{2k-2}(\frac{i-1}{2} + k,\frac{i-1}{2} - k + 1),a\right)&\text{if } [i]_2 = 1.
        \end{cases}
        \end{align}
\end{theorem}

\proof{
        Consider the function $x:\calL\rightarrow\hat \R$ given by
        \begin{equation}\label{eq:latticehyperpent}
                x(i,j,k)=\begin{cases}
                        T^{\lfloor\frac{k}2\rfloor}(b)_{i+j-\lfloor\frac{i-j-k}{4}\rfloor} & \mbox{if }[i-j+k]_4 = 0,\\T^{\lfloor\frac{k+1}2\rfloor}(u)_{i+j-\lfloor\frac{i-j-k}{4}\rfloor} & \mbox{if }[i-j+k]_4 = 2.

                \end{cases}
        \end{equation}

        This is a solution to the dSKP recurrence. Indeed, the relation
        \begin{align}
          \label{eq:dskpproofhyppent}
          \frac{(x_{-e_3}-x_{e_2})(x_{-e_1}-x_{e_3})(x_{-e_2}-x_{e_1})}{
          (x_{e_2}-x_{-e_1})(x_{e_3}-x_{-e_2})(x_{e_1}-x_{-e_3})} = -1,
        \end{align}
        at $p \in \Z^3\setminus \calL$ can be checked with some case
        handling. For instance, if $p=(i,j,k)$ with $[i-j+k]_4=1$ and
        $[k]_2=1$, one can check that
        $x_{-e_3}(p)=x_{-e_1}(p)$ while
        $x_{e_1}(p)=x_{e_3}(p)$, from which
        \eqref{eq:dskpproofhyppent} is trivial.
        If $[i-j+k]_4=1$ and $[k]_2=0$, then for some $n,\ell \in \Z$, we have
        \begin{alignat*}{3}
          &x_{-e_3}(p)=T^\ell(b)_n,& \
          &x_{e_2}(p)=T^{\ell+1}(b)_{n+1},& \
          &x_{-e_1}(p)=T^{\ell+1}(b)_{n-1}, \\
          &x_{e_3}(p)=T^{\ell+2}(u)_{n},& \
          &x_{-e_2}(p)=T^{\ell+1}(u)_{n-1},& \
          &x_{e_1}(p)=T^{\ell+1}(u)_{n+1},
        \end{alignat*}
        and \eqref{eq:dskpproofhyppent} is an application of the
        second relation of Lemma~\ref{lem:dskphyperpent}. Similarly
        for $[i-j+k]_4=3$, when $[k]_2=0$ the relation is trivial,
        while for $[k]_2=1$ it is an application of the first relation
        of Lemma~\ref{lem:dskphyperpent}.

       It is direct to check that the function $x$
        satisfies the initial condition $a_{i,j}=x(i,j,[i+j]_2)$ of
        the theorem, and that for $i$ even,
        $T^{k}(u)_i=x\left(\frac{i}{2}+k+2,\frac{i}{2}-k-1,2k-1\right)$
        while for $i$ odd,
        $T^{k}(u)_i=x\left(\frac{i-1}{2}+k,\frac{i-1}{2}-k-1,2k-1\right)$.
        Using Theorem~\ref{theo:expl_sol} yields the theorem. \qed
}

Next, we consider a conjecture of Glick \cite[Conjecture 9.2]{gdevron}. An \emph{$m$-closed polygon} is a polygon $(v_i)_{i\in\Z}$ such that $v_{i+m} = v_i$ for all $i\in \Z$.

\begin{theorem}\label{theo:short_diagonal_sing}
  
        Let $v$ be a $2m$-closed polygon and $P_0,P_1\in
        \RP^3$. Assume for all $i\in \Z$ that the plane
        $v_{i-1}v_iv_{i+1}$ contains $P_{[i]_2}$. If $T^{m-3}(v)$
        exists, then there are two planes $E_0,E_1\subset \RP^3$ such that for all $i\in \Z$ the point $T^{m-3}(v)_i$ is contained in $E_{[i]_2}$.
        
\end{theorem}

\proof{ We first prove the theorem for $(v_i)_{i\in \Z}$ a polygon in
  $\CP^3$, with all the other objects defined on $\C$ as well. One can
easily check that all the previous results in this section hold for
complex points.

 First, let us show that for any $2m$-closed polygon there is
 a $2m$-closed companion polygon. Let us denote by $\hat
 v_i$, $\hat c_i$ a homogeneous lift to $\C^4$ of
 $v_i, c_i$ for any $i\in \Z$. Assume $\hat c_{i-2}$ is given. Then
 any lift $\hat c_{i}$ can be expressed as a function of $\hat
 c_{i-2}$: by definition, it has to be in span $\langle
 \hat{v}_{i-1},\hat{v}_{i+1} \rangle$ and in the span $\langle
 \hat v_{i-2},\hat v_i,\hat c_{i-2} \rangle$, which gives the formula
        \begin{align}
                \hat c_{i} = \det(\hat v_{i-2}, \hat c_{i-2}, \hat v_i, \hat v_{i+1}) \hat v_{i-1} - \det(\hat v_{i-2}, \hat c_{i-2}, \hat v_i, \hat v_{i-1}) \hat v_{i+1}.
        \end{align}

        By iterating, we can express $\hat c_{2m}$ as a linear
        function of $\hat c_0$. This linear function takes values in the span
        $\langle  \hat v_{-1}, \hat v_1 \rangle$, so by taking $\hat c_0$
        in this subspace, we get a linear transformation of this
        $2$-dimensional subspace. This transformation has at least one eigenvector
        (this is where we use the fact that the base field is $\C$),
        so taking $\hat c_0$ to be this eigenvector, we get that $\hat
        c_{2m}$ is proportional to $\hat c_0$, therefore
        $c_{2m}=c_0$. We may do the same to get that $c_1=c_{2m+1}$ as
        well in $\CP^3$. This gives the desired $2m$-closed
        companion polygon.
        
  A consequence of Equation
    \eqref{eq:hypdyn} is that
  \begin{align}
    T^{-1}(v)_i = v_{i-3}v_{i-2}v_{i-1} \cap v_{i-1}v_iv_{i+1} \cap v_{i+1}v_{i+2}v_{i+3},
  \end{align}
  therefore $T^{-1}(v)_i = P_{[i]_2}$. Let $c$ be a companion
  polygon of $v$, and let $u,b$ be affine coordinates of
  $v,c$. $T^{-1}(u)$ is singular, so if $x$ is
  defined by Equation \eqref{eq:latticehyperpent}, we can define the new
  solution to dSKP by $y(i,j,k)=x(i,j,k-2)$ and check that $y$ is
  $(m,2)$-Devron. Applying Theorem~\ref{theo:devron_sing} to $y$, we get that $T^{m-2}(b)$ takes only two values
  (one for even indices and one for odd indices). This being true for
  any affine coordinates of $v$ and $c$ implies that there are points $B_0,B_1 \in \CP^3$ such that
  $T^{m-2}(c)_i = B_{[i]_2}$ for all $i\in \Z$. By definition of the
  companion polygon, each line $T^{m-2}(v)_{i-1}T^{m-2}(v)_{i+1}$
  contains $B_{[i]_2}$. Therefore, there are two lines $L_0,L_1
  \subset \CP^3$ such that $T^{m-2}(v)_i \in L_{[i]_2}$ for all $i\in
  \Z$.
  Additionally, by Equation \eqref{eq:hypdyn}, $T^{m-2}(v)_i \in
  T^{m-3}(v)_{i-1}T^{m-3}(v)_{i+1}$. As a consequence, each plane $E_i
  =T^{m-3}(v)_{i-2}T^{m-3}(v)_{i}T^{m-3}(v)_{i+2}$ contains the two
  points $T^{m-2}(v)_{i-1}$ and $T^{m-2}(v)_{i+1}$ and therefore also
  the line $L_{[i+1]_2}$. This in turn implies that all the even $E_i$
  coincide and that all the odd $E_i$ coincide, or that each plane
  $E_i$ is just a line, specifically the line $L_{[i+1]_2}$. In either
  case, the claim is proven, for points in $\CP^3$.

  Going back to the statement of the theorem, if points are defined
  over $\R$, then the previous reasoning applies and shows that there
  exist planes $E'_0, E'_1$ in $\CP^3$ such that for all $i\in \Z$,
  the point $T^{m-3}(v)_i$ is contained in $E'_{[i]_2}$. So for
  instance the points
  $(T^{m-3}_i)_{[i]_2=0}$ are coplanar in $\CP^3$, but also have real
  coordinates, as the dynamics \eqref{eq:hypdyn} preserves real
  coordinates. This implies that they are coplanar in $\RP^3$. To see
  this, one can consider lifts of $(T^{m-3}_i)_{[i]_2=0}$ in $\R^4$:
  these points belong to a common $3$-dimensional complex subspace, so
  there is a non-zero vector $v\in\C^4$ orthogonal to all of them, but
  then both $\RE(v)$ and $\IM(v)$ are orthogonal to all of them, and
  one of these has to be non-zero. Of course the same argument applies
  to $(T^{m-3}_i)_{[i]_2=1}$. \qed
}

\bibliographystyle{alpha}
\bibliography{references}

\end{document}